\documentclass{amsart}
\usepackage{amsmath}
\usepackage{a4wide}
\usepackage{amssymb}
\usepackage{cite} 

\usepackage{mathrsfs}
\usepackage{enumitem}

\newtheorem{thm}{Theorem}
\newtheorem*{thm*}{Theorem}

\newtheorem{prop}[thm]{Proposition}
\newtheorem{cor}[thm]{Corollary}

\theoremstyle{definition}

\newtheorem{rem}[thm]{Remark}

\newcommand{\ds}{\mathrm{\,d}s}
\newcommand{\dt}{\mathrm{\,d}t}
\newcommand{\dx}{\mathrm{\,d}x}
\newcommand{\dy}{\mathrm{\,d}y}
\newcommand{\dz}{\mathrm{\,d}z}

\newcommand{\esssup}{\operatornamewithlimits{ess\, sup\,}}

\newcommand{\R}{\mathbb{R}}

\newcommand{\N}{\mathbb{N}}

\newcommand{\I}{\mathbb{I}}
\newcommand{\Z}{\mathbb{Z}}
\newcommand{\Zm}{{\mathbb{Z}_\mu}}

\newcommand{\MM}{\mathscr M_+}
\newcommand{\MMM}{\mathscr M^\downarrow}
\newcommand{\A}{\mathbb{A}}

\newcommand{\Lpv}{L^p(v)}
\newcommand{\Lqw}{L^q(w)}

\newcommand{\fii}{\varphi}

\newcommand{\jq}{\frac1q}
\newcommand{\jp}{\frac1p}
\newcommand{\jpp}{\frac1{p'}}
\newcommand{\jr}{\frac1r}
\renewcommand{\rq}{\frac rq}
\newcommand{\rp}{\frac rp}
\newcommand{\rpp}{\frac r{p'}}
\newcommand{\rqq}{\frac r{q'}}

\renewcommand{\dj}{{\Delta_j}}
\newcommand{\dk}{{\Delta_k}}
\newcommand{\Uq}{U^q}
\newcommand{\Ur}{U^r}
\newcommand{\Upp}{U^{p'}}

\newcommand{\Uj}{U^q(\Delta_j)}
\newcommand{\Uk}{U^q(\Delta_k)}
\newcommand{\Uxt}{U^q(x,t)}

\newcommand{\Uxtk}{U^q(x,t_k)}
\newcommand{\Utkt}{U^q(t_k,t)}
\newcommand{\tm}{{t_\mu}}
\newcommand{\tk}{{t_k}}

\newcommand{\vp}{v^{1\!-p'}}

\newcommand{\Tj}{\Theta^j}
\newcommand{\Tk}{\Theta^k}
\newcommand{\Tkrq}{\Theta^\frac{kr}{q}}

\renewcommand{\theta}{\vartheta}

\renewcommand{\mod}{\operatorname{\,mod}}

\newcommand{\lt}{\left(}
\newcommand{\rt}{\right)}

\newcounter{bcount}
\newcommand{\Bdef}[1]{\refstepcounter{bcount}\label{bcounter#1}}
\newcommand{\B}[1]{B_{\ref{bcounter#1}}}

\newtoks\by
\newtoks\paper
\newtoks\book
\newtoks\jour
\newtoks\yr
\newtoks\pages
\newtoks\vol
\newtoks\publ
\newtoks\eds
\newtoks\proc
\def\ota{{\hbox{???}}}
\def\cLear{\by=\ota\paper=\ota\book=\ota\jour=\ota\yr=\ota
\pages=\ota\vol=\ota\publ=\ota}
\def\endpaper{\textsc{\the\by}, \textnormal{`\the\paper'},
\textit{\the\jour} \textnormal{\the\vol} (\the\yr), \the\pages.\cLear}
\def\endbook{\textsc{\the\by}, \textit{\the\book} (\the\publ).\cLear}
\def\endprep{\textsc{\the\by}, \textnormal{`\the\paper'} (preprint).\cLear}
\def\endproc{\textsc{\the\by}, \textit{\the\paper}, \the\publ, \the\pages.\cLear}
\def\name#1#2{#1\,#2}
\def\et{ and }

\begin{document}

\title[Hardy-type operators with a~kernel]{Boundedness of Hardy-type operators with a~kernel:\\ integral weighted conditions for the case $0<q<1\le p<\!\infty$}
\author{Martin K\v{r}epela}
\address{Karlstad University, Faculty of Health, Science and Technology, Department of Mathematics and Computer Science, 651 88 Karlstad, Sweden}
\address{Charles University, Faculty of Mathematics and Physics, Department of Mathematical Analysis, Sokolovsk\'{a} 83, 186 75 Praha 8, Czech Republic}
\email{martin.krepela@kau.se}

\begin{abstract}
  Let $U:[0,\infty)^2 \to [0,\infty)$ be a~measurable kernel satisfying:
  \begin{enumerate}
    \item[\rm(i)\ ] $U(x,y)$ is nonincreasing in $x$ and nondecreasing in $y$;
    \item[\rm(ii)\ ] there exists a~constant $\theta>0$ such that
            \[
              U(x,z) \le \theta\left( U(x,y)+U(y,z) \right)
            \]
          for all $0\le x<y<z<\infty$;
    \item[\rm(iii)\ ]
        $U(0,y)>0$ for all $y>0$.
  \end{enumerate}
  Let $0<q<1< p <\infty$. We prove that the weighted inequality
  \[  
    \left( \int_0^\infty \left( \int_0^t f(x)U(x,t)\dx \right)^q w(t) \dt \right)^\jq \le C \left( \int_0^\infty f^p(t)v(t)\dt \right)^\jp
  \]
holds for all nonnegative measurable functions $f$ on $(0,\infty)$ if and only if 
  \[
    \lt \int_0^\infty \left( \int_t^\infty w(x)\dx \right)^\rp w(t) \left( \int_0^t U^{p'}(z,t)\vp(z)\dz \right)^\rpp \dt \rt^\jr <\infty
  \]
and
  \[
    \lt \int_0^\infty \left( \int_t^\infty w(x) \Uq(t,x) \dx \right)^\rp w(t) \sup_{z\in(0,t)} \Uq(z,t)\left( \int_0^z \vp(s)\ds \right)^\rpp \dt \rt^\jr <\infty,
  \]
where $p':=\frac{p}{p-1}$ and $r:=\frac{pq}{p-q}$. Analogous conditions for the case $p=1$ and for the dual version of the inequality are also presented. 
\end{abstract}
\maketitle

\section{Introduction}

Operators of the general form 
  \[
    Tf(x) = \int_0^\infty f(y)U(x,y) \dy,
  \]
where $U$ is a~\emph{kernel}, play an~indispensable role in various areas of analysis. The means of their investigation, naturally, greatly depend on additional properties of the kernel $U$. 

In the present article, we study the so-called \emph{Hardy-type} operators 
  \begin{equation}\label{700}
    Hf(x) = \int_0^x f(y)U(y,x) \dy, \qquad \textnormal{and}\qquad  H^*f(x) = \int_x^\infty f(y)U(x,y) \dy,
  \end{equation}
where the kernel $U:[0,\infty)^2\to [0,\infty)$ is a~measurable function which has the following properties: 
  \begin{enumerate}
    \item[\rm(i)]
      $U(x,y)$ is nonincreasing in $x$ and nondecreasing in $y$; 
    \item[\rm(ii)] there exists a~constant $\theta>0$ such that for all $0\le x<y<z<\infty$ it holds
      \[
        U(x,z) \le \theta\left( U(x,y)+U(y,z) \right);
      \]
    \item[\rm(iii)]
      $U(0,y)>0$ for all $y>0$.
  \end{enumerate}
If $\theta>0$ and $U$ is a~function satisfying the conditions above with the given parameter $\theta$ in point (ii), then we, for the sake of simplicity, call $U$ a~\emph{$\theta$-regular kernel}.  
                                                                      
The simplest case of a~$\theta$-regular kernel $U$ is the constant $U\equiv 1$, with which $H$ and $H^*$ become the ordinary Hardy and Copson (``dual Hardy'') operators, respectively. Other examples of $\theta$-regular kernels include the Riemann-Liouville kernel 
  \[
    U(x,y)=(y-x)^\alpha, \quad \alpha>0,
  \]
the logarithmic kernel 
  \[
    U(x,y)=\log^\alpha \lt \frac yx \rt, \quad \alpha>0,
  \]
and the kernels 
  \[
    U(x,y)=\int_x^y u(t)\dt \textnormal{\quad and\quad }  U(x,y)=\esssup_{t\in(x,y)}u(t),
  \]
where $u$ is a~given nonnegative measurable function. These operators find applications, for instance, in the theory of differentiability of functions, interpolation theory and more topics involving function spaces. The two last-named examples of $\theta$- regular kernels prove to be particularly useful in a~study of the so-called iterated Hardy operators \cite{GKPS,GS}, for example.

The particular aspect we investigate in this paper is boundedness of the operators $H$ and $H^*$ with a~$\theta$-regular kernel $U$ between weighted Lebesgue spaces. In order to define these spaces, we need to introduce several auxiliary terms first. 

Throughout the text, by a~\emph{measurable function} we always mean a~Lebesgue measurable function (on an~appropriate subset of $\R$). The symbol $\MM$ denotes the cone of all nonnegative measurable functions on $(0,\infty)$. A~\emph{weight} is a~function $w\in\MM$ on $(0,\infty)$ such that 
  \[
    0<\int_0^t w(s)\ds<\infty \textnormal{\quad for\ all\quad} t>0.
  \]
Finally, if $v$ is a~weight and $p\in(0,\infty]$, then the \emph{weighted Lebesgue space} $\Lpv=\Lpv(0,\infty)$ is defined as the set of all real-valued measurable functions $f$ on $(0,\infty)$ such that
  \begin{align*}
    \|f\|_{\Lpv} & :=\ \left(\int_0^\infty |f(t)|^p v(t)\dt\right)^\jp <\infty & \textnormal{if\ }p<\infty, \\
    \|f\|_{L^{\!\infty}(v)} & :=\ \esssup_{t\in(0,\infty)} |f(t)|v(t) <\infty & \textnormal{if\ }p=\infty.
  \end{align*}
Note that if $p\in(0,1)$, then $(\Lpv, \|\cdot\|_{\Lpv})$ is in general not a~normed linear space because of the absence of the Minkowski inequality in this case. However, as we deal only with the case $1\le p<\infty$ anyway, this detail is not of our concern here.

Throughout the text, if $p\in(0,1)\cup(1,\infty)$, then $p'$ is defined by $p'=\frac{p}{p-1}$. Analogous notation is used for $q'$. 

In the following, assume that $\theta\in(0,\infty)$, $U$ is a~$\theta$-regular kernel, $H$ is the corresponding operator from \eqref{700} and $v$, $w$ are weights. Boundedness of the operator $H$ between $\Lpv$ and $\Lqw$ was completely characterized for $p,q\in[1,\infty]$. The authors credited for this work are Bloom and Kerman \cite{BK}, Oinarov \cite{Oi} and Stepanov \cite{S1}. The results of \cite{Oi}, for instance, have the following form. 
\begin{thm*}[\!\!{\cite[Theorem 1.1]{Oi}}]
  Let $1<p\le q<\infty$. Then $H:\Lpv\to\Lqw$ is bounded if and only if 
    \[
      E_1 := \sup_{t\in(0,\infty)} \lt \int_t^\infty \Uq(t,x)w(x)\dx \rt^\jq \lt \int_0^t \vp(x)\dx \rt^\jpp < \infty
    \]
  and
    \[
      E_2 := \sup_{t\in(0,\infty)} \lt \int_t^\infty w(x)\dx \rt^\jq \lt \int_0^t \Upp(x,t)\vp(x)\dx \rt^\jpp <\infty.
    \]
  Moreover, the least constant $C$ such that the inequality
    \begin{equation}\label{701}
      \| Hf\|_{\Lqw} \le C \|f\|_{\Lpv}
    \end{equation}
  holds for all $f\in\MM$ satisfies $C\approx E_1 + E_2$. 
\end{thm*}
\begin{thm*}[\!\!{\cite[Theorem 1.2]{Oi}}]
  Let $1<q<p<\infty$ and $r:=\frac{pq}{p-q}$. Then $H:\Lpv\to\Lqw$ is bounded if and only if 
    \[
      E_3 := \lt \int_0^\infty \lt \int_t^\infty \Uq(t,x)w(x)\dx \rt^\rq \lt \int_0^t \vp(x)\dx \rt^\rqq \vp(t) \dt \rt^\jr < \infty
    \]
  and
    \[
      E_4 := \lt \int_0^\infty \lt \int_t^\infty w(x)\dx \rt^\rp w(t) \lt \int_0^t \Upp(x,t)\vp(x)\dx \rt^\rpp \dt \rt^\jr <\infty.
    \]
  Moreover, the least constant $C$ such that \eqref{701} holds for all $f\in\Lpv$ satisfies $C\approx E_3 + E_4$. 
\end{thm*}
The conditions obtained in \cite{BK,S1} have a~slightly different form, a~more detailed comparison between them is found in \cite{S1}. 

As for the ``limit cases'', conditions for the case $p=\infty$ and $q\in(0,\infty]$ are obtained very easily, the same applies to the case $q=1$ and $p\in[1,\infty)$ in which one simply uses the Fubini theorem. Yet another possible choice of parameters is $p=1$ and $q\in(1,\infty]$. It was (at least for $q<\infty$) included in \cite[Theorem 1.2]{Oi} and the conditions may be recovered from that article by correctly interpreting the expressions involving the symbol $p'$ in there. Another option is to follow the more general theorem \cite[Chapter XI, Theorem 4]{KA}. 

If $0<p<1$, then the operator $H$ can never be bounded  (provided that $U$, $v$, $w$ are nontrivial, which is always assumed here). The problem in here lies in the fact that for each $t>0$ there exists $f_t\in L^p(v)$ which is not locally integrable at the point $t$. For more details, see e.g.~\cite{L2}.

No such difficulty arises if $0<q<1\le p<\infty$. In this case, $H$ may indeed be bounded between $\Lpv$ and $\Lqw$ and it is perfectly justified to ask for the conditions under which this occurs. As for the known answers to this question, the situation is however much worse than in the other cases.

When assumed $U\equiv 1$, i.e.~for the ordinary Hardy operator, the boundedness characterization was found by Sinnamon \cite{Si91} and it corresponds to the condition $E_3<\infty$ (with $U\equiv 1$, of course). In the general case, in \cite{S1} it was shown that the condition $E_3<\infty$ is sufficient but not necessary for $H:\Lpv\to\Lqw$ to be bounded, while the~condition   
  \[
    E_5 := \lt \int_0^\infty \lt \int_t^\infty \Uq(t,x)w(x)\dx \rt^\frac{p'}{q} \vp(t) \dt \rt^\jpp <\infty 
  \]
is necessary but not sufficient. For related counterexamples, see \cite{S2}. The fact that the two conditions do not meet is a~significant drawback. An~equivalent description of the optimal constant $C$ in \eqref{701} is usually substantial for the result to be applicable in any way.

Lai \cite{L} found equivalent conditions by proving that, with $0<q<1<p<\infty$, the operator $H$ is bounded from $\Lpv$ to $\Lqw$ if and only if
  \[
    \widetilde{D}_1 := \sup_{\{t_k\}} \sum_{k} \lt \int_{t_k}^{t_{(\!k\!+\!1\!)}} w(t)\dt \rt^\rq \lt \int_{t_{(\!k\!-\!1\!)}}^{\tk} \Upp(x,\tk) \vp(x)\dx \rt^\rpp<\infty
  \]   
as well as
  \[
    \widetilde{D}_2 := \sup_{\{t_k\}} \sum_{k} \lt \int_{\tk}^{t_{(\!k\!+\!1\!)}} w(t) \Uq(\tk,t)\dt \rt^\rq \lt \int_{t_{(\!k\!-\!1\!)}}^{t_k} \vp(x)\dx \rt^\rpp<\infty.
  \]
The suprema in here are taken over all covering sequences, i.e.~partitions of $(0,\infty)$ (see \cite{L} or Section 2 for the definitions), and $r:=\frac{pq}{p-q}$, as usual. Moreover, these conditions satisfy $\widetilde{D}_1 + \widetilde{D}_2 \approx C^r$ with the least $C$ such that \eqref{701} holds for all $f\in\MM$. Corresponding variants for $p=1$ are also provided in \cite{L}. The earlier use of similar partitioning techniques in the paper \cite{MRS} of Mart\'{i}n-Reyes and Sawyer should be also credited.

Unfortunately, even though the $\widetilde{D}$-conditions are both sufficient and necessary, they are only hardly verifiable due to their discrete form involving all possible covering sequences. This fact has hindered their use in various applications (see e.g.~\cite{GS}). In contrast, in the case $1<q<p<\infty$ it is known (see \cite{S2,L}) that $\widetilde{D}_1 + \widetilde{D}_2 \approx A_3^r + A_4^r$. This does not apply when $0<q<1\le p<\infty$, as shown by the results of \cite{S1} mentioned earlier. 

Rather recently, Prokhorov \cite{P} found conditions for $0<q<1\le p<\infty$ which have an~integral form but involve a~function $\zeta$ defined by
  \[
    \zeta(x):=\sup \left\{ y\in(0,\infty); \int_y^\infty \!w(t)\dt \ge (\theta^q\!+1) \int_x^\infty \! w(t)\dt \right\}, \quad x>0.
  \]
The conditions presented in \cite{P} even involve this function iterated three times. The presence of such an~implicit expression involving the weight $w$ virtually prevents any use of these conditions in applications which require further manipulation $w$ (see Section 4 for an~example). Finding explicit integral conditions for the case $0<q<1\le p<\infty$, which would have a~form comparable e.g.~to $E_3$ and $E_4$, hence remained an open problem.\\

In this paper, we solve this problem and provide the missing integral conditions. No additional assumptions on the weights $v$, $w$ and the $\theta$-regular kernel $U$ are required here, neither are any implicit expressions. The results are presented in Theorems \ref{7}, \ref{144} and Corollaries \ref{7*}, \ref{144*}. The proofs are based on the well-known method of \emph{dyadic discretization} (or \emph{blocking technique}, see \cite{GE} for a~basic introduction into this method). The particular variant of the technique employed here is essentially the same as the one used in \cite{K6}.

Concerning the structure of this paper, this introduction is followed by Section 2 where additional definitions and various auxiliary results are presented. Section 3 consists of the main results, their proofs and some related remarks. In the final Section 4 we present certain examples of applications of the results.

\section{Definitions and preliminaries}

Let us first introduce the remaining notation and terminology used in the paper. We say that $\I\subseteq \Z$ is an~\emph{index set} if there exist $k_{\min},\, k_{\max}\in\Z$ such that $k_{\min}\le k_{\max}$ and
    \[
      \I=\{ k\in\Z,\,k_{\min}\le k \le k_{\max}\}.
    \]
Moreover, we denote
  \[
    \I_0:=\I\setminus \{k_{\min},k_{\max}\}.
  \]
Let $\I$ be an~index set containing at least three indices. Then a~sequence of points $\{t_k\}_{k\in\I}$ is called a~\emph{covering sequence} if $t_{k_{\min}}=0$, $t_{k_{\max}}=\infty$ and $t_k<t_{(\!k\!+\!1\!)}$ whenever $k\in\I\setminus\{k_{\max}\}$. 

Next, let $z\in\N\,\cup\{0\}$ and $n,k\in\N$ are such that $0\le k <n$. We write $z\mod n = k$ if there exists $j\in\N\,\cup\{0\}$ such that $z=jn + k$. In other words, $k$ is the remainder after division of the number $z$ by the number $n$.\\

In the next part, we present various auxiliary results which will be needed later. The first two propositions are simple consequences of the H\"{o}lder inequality (for functions and sequences) and of the characterizations of the dual spaces to the spaces $L^p$ and $l^p$. For more details see e.g.~\cite{K6}.

\begin{prop}\label{107}
  Let $v$ be a~weight, $p\in(1,\infty)$ and $0\le x<y\le \infty$. Let $f$ and $\fii$ be nonnegative measurable functions on $(x,y)$. Then the H\"older inequality gives
    \[
      \int_x^y f(s)\fii(s)\ds \le \left(\int_x^y f^p(s)v(s)\ds \right)^\jp \left(\int_x^y \fii^{p'}\!(s)\vp(s)\ds \right)^\frac1{p'}.
    \]
  Moreover, if $\int_x^y \fii^{p'}\!(s)\vp(s)\ds <\infty$, then there exists a~nonnegative measurable function $g$ supported in $[x,y]$ and such that $\int_x^y g^p(s)v(s)\ds = 1$ and  
    \[
      \left(\int_x^y \fii^{p'}\!(s)\vp(s)\ds \right)^\frac1{p'} \le 2 \int_x^y g(s)\fii(s)\ds <\infty.
    \]
  If $\int_x^y \fii^{p'}\!(s)\vp(s)\ds =\infty$, then for every $M\in\N$ there exists a~nonnegative measurable function $h$ supported in $[x,y]$ and such that $\int_x^y h^p(s)v(s)\ds = 1$ and
    \[
      \int_x^y h(s)\fii(s)\ds > M.
    \]
\end{prop}

\begin{prop}\label{108}
  Let $\I$ be an~index set. Let $\{a_k\}_{k\in\I}$ and $\{b_k\}_{k\in\I}$ be two nonnegative sequences. Assume that $0<q<p<\infty$. Then the H\"older inequality gives
    \[
      \left(\sum_{k\in\I} a_k^q b_k\right)^\frac1q \le \left(\sum_{k\in\I} a_k^p\right)^\frac1p \left( \sum_{k\in\I}b_k^\frac{p}{p-q}\right)^\frac{p-q}{pq}.
    \]  
  Moreover, if $\sum_{k\in\I}b_k^\frac{p}{p-q}\!<\infty$, then there exists a~nonnegative sequence $\{c_k\}_{k\in\I}$ such that \linebreak $\sum_{k\in\I} c_k^p=1$ and
    \[
      \left( \sum_{k\in\I}b_k^\frac{p}{p-q}\right)^\frac{p-q}{pq} \le 2 \left(\sum_{k\in\I} c_k^q b_k\right)^\frac1q <\infty.
    \]
\end{prop} 

Proposition \ref{107} has also a~variant for $p=1$, giving
  \[
    \int_x^y f(s)\fii(s)\ds \le \int_x^y f(s)v(s)\ds \ \esssup_{s\in(x,y)} \fii(s)v^{-1}(s)\ds.
  \]  
The remaining part of Proposition \ref{107} is then also true, provided that the expressions are modified accordingly.

The next proposition was proved in \cite[Proposition 2.1]{GHS}, more comments may be found e.g.\ in \cite{K6}. It is a~fundamental part of the discretization method.

\begin{prop}\label{3}
  Let $0<\alpha<\infty$ and $1<D<\infty$. Then there exists a~constant $C_{\alpha,D}\in(0,\infty)$ such that for any index set $\I$ and any two nonnegative sequences $\{b_k\}_{k\in\I}$ and $\{c_k\}_{k\in\I}$, satisfying 
    \[
      b_{(\!k\!+\!1\!)}\ge D\, b_k \textit{\ for all\ } k\in\I\setminus\{k_{\max}\},
    \]
  it holds
    \[
      \sum_{k=k_{\min}}^{k_{\max}} \left( \sum_{m=k}^{k_{\max}} c_m \right)^\alpha b_k  \le C_{\alpha,D} \sum_{k=k_{\min}}^{k_{\max}} c_k^\alpha b_k 
    \]
  and 
    \[
      \sum_{k=k_{\min}}^{k_{\max}} \left(\sup_{k\le m \le k_{\max}} \!\! c_m \right)^\alpha b_k  \le C_{\alpha,D}  \sum_{k=k_{\min}}^{k_{\max}} c_k^\alpha b_k .
    \]
\end{prop}    

The following result is an~analogy to the previous proposition. We present a~simple proof, although the result is also well known.

\begin{prop}\label{89}
  Let $0<\alpha<\infty$ and $1<D<\infty$. Then there exists a~constant $C_{\alpha,D}\in(0,\infty)$ such that for any index set $\I$ and any two nonnegative sequences $\{b_k\}_{k\in\I}$ and $\{c_k\}_{k\in\I}$, satisfying 
    \[
      b_{(\!k\!+\!1\!)} \ge D\, b_k   \textit{\ for all\ } k\in\I\setminus\{k_{\max}\},
    \]
  it holds
    \[
      \sup_{k_{\min}\le k \le k_{\max}} \left( \sum_{m=k}^{k_{\max}} c_m \right)^\alpha b_k  \le C_{\alpha,D} \sup_{k_{\min}\le k \le k_{\max}} c_k^\alpha b_k. 
    \]
\end{prop}

\begin{proof}
  It holds
    \begin{align*}
      \sup_{k_{\min}\le k \le k_{\max}} \lt \sum_{m=k}^{k_{\max}} c_m \rt^\alpha b_k     & =   \sup_{k_{\min}\le k \le k_{\max}} \lt \sum_{m=k}^{k_{\max}} c_m b_m^{-\!\frac1\alpha} b_m^{\!\frac1\alpha} \rt^\alpha b_k \\
                                                                                         & \le \sup_{k_{\min}\le k \le k_{\max}} \lt \sum_{m=k}^{k_{\max}} b_m^{-\!\frac1\alpha} \rt^\alpha b_k \sup_{k\le i\le k_{\max}} c_i^\alpha b_i \\
                                                                                         & \le \sup_{k_{\min}\le k \le k_{\max}} \lt b_k^{-\!\frac1\alpha} \sum_{m=0}^{k_{\max}-k} D^{-\!\frac m\alpha} \rt^\alpha b_k \sup_{k\le i\le k_{\max}} c_i^\alpha b_i^\alpha \\
                                                                                         & \le \lt \sum_{m=0}^{\infty} D^{-\!\frac m\alpha} \rt^\alpha \sup_{k_{\min}\le i\le k_{\max}} c_i^\alpha b_i^\alpha \\
                                                                                         & =   \frac D{(D^{\frac 1\alpha}\!-1)^\alpha} \sup_{k_{\min}\le i\le k_{\max}} c_i^\alpha b_i^\alpha. 
    \end{align*} 
\end{proof}

Applying Proposition \ref{3}, one obtains the next result. It is useful to handle inequalities involving a~$\theta$-regular kernel. 

\begin{prop}\label{4}
  Let $0<\alpha<\infty$ and $\theta\in [1,\infty)$. Let $U$ be a~$\theta$-regular kernel. Then there exists a~constant $C_{\alpha,\theta}\in(0,\infty)$ such that, for any index set $\I$, any increasing sequence $\{t_k\}_{k\in\I}$ of points from $(0,\infty]$ and any nonnegative sequence $\{a_k\}_{k\in\I\setminus\{k_{\max}\}}$ satisfying
    \begin{equation}\label{6}
      a_{(\!k\!+\!1\!)}\ge 2\theta^\alpha a_k \textit{\ for all\ } k\in\I\setminus\{k_{\max},k_{\max}-1\},
    \end{equation}
  it holds
    \[
      \sum_{k=k_{\min}}^{k_{\max}-1} a_k U^\alpha(t_k,t_{k_{\max}}) \le C_{\alpha,\theta} \sum_{k=k_{\min}}^{k_{\max}-1} a_k U^\alpha(t_k,t_{(\!k\!+\!1\!)}) .
    \]
\end{prop}

\begin{proof}
  Naturally, we may assume that $\I$ contains at least three indices. Let $k\in\I\setminus\{k_{\max}\}$. By iterating the inequality
    \begin{equation}\label{61}
      U(x,z) \le \theta U(x,y) + \theta U(z,y) \qquad (x<y<z)
    \end{equation}
  from the definition of the $\theta$-regular kernel, we get
    \begin{align*} 
      U(t_k, t_{k_{\max}}) & \le \theta U(t_k, t_{(\!k\!+\!1\!)}) + \theta U(t_{(\!k\!+\!1\!)}, t_{k_{\max}}) \\
                           & \le \theta U(t_k, t_{(\!k\!+\!1\!)}) + \theta^2 U(t_{(\!k\!+\!1\!)}, t_{(k+2)}) + \theta^2 U(t_{(k+2)}, t_{k_{\max}}) \\
                           & \ \vdots \\
                           & \le \sum_{m=k}^{k_{\max}-1} \theta^{m-k+1} U(t_m,t_{(m+1)}) \\
                           & = \theta^{-k} \sum_{m=k}^{k_{\max}-1} \theta^{m+1} U(t_m,t_{(m+1)}). 
    \end{align*}
  Set $b_k := \theta^{-\alpha k}a_k$ for $k\in\I\setminus\{k_{\max}\}$. Then, by \eqref{6},  for all $k\in\I\setminus\{k_{\max},k_{\max}-1\}$ it holds $b_{(\!k\!+\!1\!)}\ge 2 b_k$. We obtain 
    \begin{align}
      \sum_{k=k_{\min}}^{k_{\max}-1} a_k U^\alpha(t_k,t_{k_{\max}}) & \le       \sum_{k=k_{\min}}^{k_{\max}-1} \theta^{-\alpha k} a_k \left( \sum_{m=k}^{k_{\max}-1} \theta^{m+1} U(t_m,t_{(m+1)} \right)^\alpha \nonumber\\
                                                                    & =         \sum_{k=k_{\min}}^{k_{\max}-1} b_k \left( \sum_{m=k}^{k_{\max}-1} \theta^{m+1} U(t_m,t_{(m+1)}) \right)^\alpha \nonumber\\
                                                                    & \le       C_\alpha \sum_{k=k_{\min}}^{k_{\max}-1} b_k \theta^{\alpha(\!k\!+\!1\!)} U^\alpha(t_k,t_{(\!k\!+\!1\!)}) \label{60}\\
                                                                    & =         C_\alpha \theta^\alpha \sum_{k=k_{\min}}^{k_{\max}-1} a_k U^\alpha(t_k,t_{(\!k\!+\!1\!)}). \nonumber
    \end{align}
  To get the inequality \eqref{60}, we used Proposition \ref{3}, setting $D:=2$ and $c_m:=U(t_m,t_{(m+1)})$ for the relevant indices $m$. This proves the statement. 
\end{proof}

Notice that, by the definitions at the beginning of this section, we consider only finite index sets (and therefore also finite covering sequences later on). However, all the results of this section hold for infinite sequences as well. This may be easily shown by using a~limit argument. We will nevertheless continue working with finite index sets and covering sequences only. The notion of supremum is used regularly even where it relates to a~finite set and where it therefore could be replaced by a~maximum. For further remarks see the last part of Section 3.  

The final basic result concerns $\theta$-regular kernels and reads as follows. 

\begin{prop}\label{59}
  Let $0\le a<b<c\le\infty$, $0<\alpha<\infty$ and $1\le \theta<\infty$. Let $U$ be a~$\theta$-regular kernel and $\psi$ be a~nonincreasing nonnegative function defined on $(0,\infty)$. Then
    \[
      \sup_{z\in[a,c)} U^\alpha(a,z)\psi(z) \le (1+\theta) \lt \sup_{z\in[a,b]} U^\alpha(a,z)\psi(z) + \sup_{z\in[b,c)} U^\alpha(b,z)\psi(z) \rt.
    \]
  If $c<\infty$, the result is unchanged if the intervals $[a,c)$ and $[b,c)$ in the suprema are replaced by $[a,c]$ and $[b,c]$, respectively.  
\end{prop}

\begin{proof}
  The result is a~consequence to the following simple observation.
    \begin{align*}
      \sup_{z\in[a,c)} U^\alpha(a,z)\psi(z) & \le      \sup_{z\in[a,b]} U^\alpha(a,z)\psi(z) + \sup_{z\in[b,c)} U^\alpha(a,z)\psi(z) \\ 
                                            & \le      \sup_{z\in[a,b]} U^\alpha(a,z)\psi(z) + \theta U^\alpha(a,b) \sup_{z\in[b,c)}\psi(z) + \theta \sup_{z\in[b,c)} U^\alpha(b,z)\psi(z)\\ 
                                            & =        \sup_{z\in[a,b]} U^\alpha(a,z)\psi(z) + \theta U^\alpha(a,b) \psi(b) + \theta \sup_{z\in[b,c)} U^\alpha(b,z)\psi(z)\\ 
                                            & \le      (1+\theta) \lt \sup_{z\in[a,b]} U^\alpha(a,z)\psi(z) + \sup_{z\in[b,c)} U^\alpha(b,z)\psi(z) \rt. 
    \end{align*}
\end{proof}

\section{Main results}

This section contains the main theorems and their proofs. Remarks to the results and proof techniques can be found at the end of the section.

The notation $A\lesssim B$ means that $A\le C B$, where the constant $C$ may depend \emph{only} on the exponents $p$, $q$ and the parameter $\theta$. In particular, this $C$ is always independent on the weights $w$, $v$, on certain indices (such as $k$, $n$, $j$, $K$, $N$, $J$, $\mu,\dots$), on the number of summands involved in sums etc. We write $A\approx B$ if both $A\lesssim B$ and $B\lesssim A$.

\pagebreak
\begin{thm}\label{7}
  Let $0<q<1<p<\infty$, $r:=\frac{pq}{p-q}$ and $0<\theta<\infty$. Let $v$, $w$ be weights. Let $U$ be a~$\theta$-regular kernel. Then the following assertions are equivalent:
    \begin{enumerate}
      \item[\rm(i)]
        There exists a~constant $C\in(0,\infty)$ such that the inequality
          \begin{equation}\label{8}  
            \left( \int_0^\infty \lt \int_t^\infty f(x)U(t,x)\dx \rt^q w(t) \dt \right)^\jq \le C \lt \int_0^\infty f^p(t)v(t)\dt \rt^\jp
          \end{equation}
        holds for all functions $f\in\MM$. 
      \item[\rm(ii)]
        Both the conditions
          \[
            D_1 := \sup_{\substack{\{t_k\}_{k\in\I} \\ \textnormal{covering}\\ \textnormal{sequence}}} \sum_{k\in\I_0} \lt \int_{t_{(\!k\!-\!1\!)}}^{t_k} w(t)\dt \rt^\rq \lt \int_{\tk}^{t_{(\!k\!+\!1\!)}} \Upp(\tk,x) \vp(x)\dx \rt^\rpp<\infty
          \]   
        and
          \[
            D_2 := \sup_{\substack{\{t_k\}_{k\in\I} \\ \textnormal{covering}\\ \textnormal{sequence}}} \sum_{k\in\I_0} \lt \int_{t_{(\!k\!-\!1\!)}}^{t_k} w(t) \Uq(t,\tk)\dt \rt^\rq \lt \int_{\tk}^{t_{(\!k\!+\!1\!)}} \vp(x)\dx \rt^\rpp<\infty
          \]
        are satisfied.
      \item[\rm(iii)]
        Both the conditions   
          \[
            A_1 := \int_0^\infty \left( \int_0^t w(x)\dx \right)^\rp w(t) \left( \int_t^\infty U^{p'}(t,z)\vp(z)\dz \right)^\rpp \dt <\infty
          \]
        and
          \[
            A_2 := \int_0^\infty \left( \int_0^t w(x) \Uq(x,t) \dx \right)^\rp w(t) \sup_{z\in[t,\infty)} \Uq(t,z)\left( \int_z^\infty \vp(s)\ds \right)^\rpp \dt<\infty
          \] 
        are satisfied.    
    \end{enumerate}
  Moreover, if $C$ is the least constant such that \eqref{8} holds for all functions $f\in\MM$, then 
    \[
      C^r \approx D_1 + D_2 \approx A_1 + A_2.
    \]
\end{thm}

\pagebreak
The variant of the previous theorem for $p=1$ reads as follows.

\begin{thm}\label{144}
  Let $0<q<1=p$ and $0<\theta<\infty$. Let $v$, $w$ be weights. Let $U$ be a~$\theta$-regular kernel. Then the following assertions are equivalent:
    \begin{enumerate}
      \item[\rm(i)]
        There exists a~constant $C\in(0,\infty)$ such that the inequality \eqref{8} holds for all functions $f\in\MM$. 
      \item[\rm(ii)]
        Both the conditions
          \[
            D_3 := \sup_{\substack{\{t_k\}_{k\in\I} \\ \textnormal{covering}\\ \textnormal{sequence}}} \sum_{k\in\I_0} \lt \int_{t_{(\!k\!-\!1\!)}}^{t_k} w(t)\dt \rt^{1-q'} \!\!\! \esssup_{x\in(\tk,t_{(\!k\!+\!1\!)})} U^{-q'}(\tk,x)\, v^{q'}\!(x)\dx <\infty
          \]   
        and
          \[
            D_4 := \sup_{\substack{\{t_k\}_{k\in\I} \\ \textnormal{covering}\\ \textnormal{sequence}}} \sum_{k\in\I_0} \lt \int_{t_{(\!k\!-\!1\!)}}^{t_k} w(t) \Uq(t,\tk)\dt \rt^{1-q'} \!\!\! \esssup_{x\in(\tk,t_{(\!k\!+\!1\!)})} v^{q'}\!(x)\dx <\infty
          \]
        are satisfied.
      \item[\rm(iii)]
        Both the conditions   
          \[
            A_3 := \int_0^\infty \left( \int_0^t w(x)\dx \right)^{-q'} w(t)\  \esssup_{z\in(t,\infty)} U^{-q'}(t,z)\, v^{q'}\!(z)\, \dt <\infty
          \]
        and
          \[
            A_4 := \int_0^\infty \left( \int_0^t w(x) \Uq(x,t) \dx \right)^{-q'} w(t)\  \esssup_{z\in(t,\infty)} \Uq(t,z)\, v^{q'}\!(z) \, \dt<\infty
          \] 
        are satisfied.    
    \end{enumerate}
  Moreover, if $C$ is the least constant such that \eqref{8} holds for all functions $f\in\MM$, then 
    \[
      C^{-q'} \approx D_3 + D_4 \approx A_3 + A_4.
    \]
\end{thm}

By performing a~simple change of variables $t \to \frac1t$, one gets the two corollaries below. They are formulated without the discrete conditions, those corresponding to Corollary \ref{7*} were presented in Section 1. An~interested reader may also derive all the discrete conditions easily from their respective counterparts in Theorems \ref{7} and \ref{144}.

\begin{cor}\label{7*}
  Let $0<q<1<p<\infty$, $r:=\frac{pq}{p-q}$ and $0<\theta<\infty$. Let $v$, $w$ be weights. Let $U$ be a~$\theta$-regular kernel. Then the following assertions are equivalent:
    \begin{enumerate}
      \item[\rm(i)]
        There exists a~constant $C\in(0,\infty)$ such that the inequality
          \begin{equation}\label{8*}  
            \left( \int_0^\infty \lt \int_0^t f(x)U(x,t)\dx \rt^q w(t) \dt \right)^\jq \le C \lt \int_0^\infty f^p(t)v(t)\dt \rt^\jp
          \end{equation}
        holds for all functions $f\in\MM$. 
      \item[\rm(ii)]
        Both the conditions   
          \[
            A^*_1 := \int_0^\infty \left( \int_t^\infty w(x)\dx \right)^\rp w(t) \left( \int_0^t U^{p'}(z,t)\vp(z)\dz \right)^\rpp \dt <\infty
          \]
        and
          \[
            A^*_2 := \int_0^\infty \left( \int_t^\infty w(x) \Uq(t,x) \dx \right)^\rp w(t) \sup_{z\in(0,t]} \Uq(z,t)\left( \int_0^z \vp(s)\ds \right)^\rpp \dt<\infty
          \] 
        are satisfied.    
    \end{enumerate}
  Moreover, if $C$ is the least constant such that \eqref{8*} holds for all functions $f\in\MM$, then 
    \[
      C^r \approx A^*_1 + A^*_2.
    \]
\end{cor}

\begin{cor}\label{144*}
  Let $0<q<1=p$ and $0<\theta<\infty$. Let $v$, $w$ be weights. Let $U$ be a~$\theta$-regular kernel. Then the following assertions are equivalent:
    \begin{enumerate}
      \item[\rm(i)]
        There exists a~constant $C\in(0,\infty)$ such that the inequality \eqref{8*} holds for all functions $f\in\MM$. 
      \item[\rm(ii)]
        Both the conditions   
          \[
            A^*_3 := \int_0^\infty \left( \int_t^\infty w(x) \dx \right)^{-q'} w(t)\  \esssup_{z\in(0,t)} U^{-q'}(z,t)\, v^{q'}\!(z)\, \dt <\infty
          \]
        and
          \[
            A^*_4 := \int_0^\infty \left( \int_t^\infty w(x) \Uq(t,x) \dx \right)^{-q'} w(t)\  \esssup_{z\in(0,t)} \Uq(z,t)\, v^{q'}\!(z) \, \dt<\infty
          \] 
        are satisfied.    
    \end{enumerate}
  Moreover, if $C$ is the least constant such that \eqref{8*} holds for all functions $f\in\MM$, then 
    \[
      C^{-q'} \approx A^*_3 + A^*_4.
    \]
\end{cor}

The next part contains the proofs. The core components of the discretization method used in this article are summarized in Theorem \ref{9} below. It is presented separately for the purpose of possible future reference since this particular variant of discretization may be used even in other problems (cf.~\cite{K6}).

Throughout the text, parentheses are used in expressions that involve indices, producing symbols such as $t_{(\!k\!+\!1\!)},$ $t_{k_{(\!n\!+\!1\!)}}$, etc. The parentheses do not have a~special meaning, i.e.~$t_{(\!k\!+\!1\!)}$ simply means $t$ with the index $k+1$. They are used to make it easier to distinguish between objects as $t_{k_{(\!n\!+\!1\!)}}$ and $t_{(k_n+1)}$, which, in general, are different and both of them appear frequently in the formulas.

\begin{thm}\label{9}
  Let $0<q<\infty$ and $1\le\theta<\infty$. Define
    $$\Theta := 2 \theta^{q}.$$
  Let $U$ be a~$\theta$-regular kernel.
  Let $K\in\Z$ and $\mu\in\Z$ be such that $\mu\le K-2$. Define the index set
    \begin{equation}\label{23}
      \Zm := \{k\in\Z;\ \mu\le k \le K-1 \}.
    \end{equation}
   Let $w$ be a~weight such that $\int_0^\infty w = \Theta^K$. Let $\{t_k\}_{k=-\infty}^{K}\subset (0,\infty]$ be a~sequence of points such that
    \begin{equation}\label{17}
      \int_0^{t_k} w(x)\dx = \Tk
    \end{equation}
  for all $k\in\Z$ such that $k\le K$ and $t_K=\infty$. For all $k\in\Z$ such that $k\le K-1$, denote 
    \[
      \dk := [t_k, t_{(\!k\!+\!1\!)})
    \]
  and
    \[
      U(\dk) := U\left(t_k, t_{(\!k\!+\!1\!)}\right).
    \]
  Then there exist a~number $N\in\N$ and an~index set $\{k_n\}_{n=0}^N\subset \Zm$ with the following properties.
    \begin{itemize}
      \item[\rm(i)]
        It holds $k_0=\mu$ and $k_{(\!n\!+\!1\!)}=K$. Whenever $n\in\{0,\ldots,N\}$, then $k_n + 1 \le k_{(\!n\!+\!1\!)}$ and therefore also
          \begin{equation}\label{10}
            t_{(k_n + 1)} \le t_{k_{(\!n\!+\!1\!)}}.
          \end{equation}
        If we define
          \begin{equation}\label{20}
            \A := \{ n\in\N;\ n\le N,\ k_n + 1 < k_{(\!n\!+\!1\!)} \},
          \end{equation}
        then 
          \begin{equation}\label{11}
            \Zm = \{ k_{(\!n\!+\!1\!)}\!-1;\ n\in\N\cup\{0\},\ n\le N \} \cup \{ k;\ k\in\Z,\ n\in\A,\ k_n\le k \le k_{(\!n\!+\!1\!)}\!-2 \}.
          \end{equation}
      \item[\rm(ii)]
        For every $n\in\N$ such that $n\le N-1$ it holds
          \begin{equation}\label{12}
            \sum_{k=k_n}^{k_{(\!n\!+\!1\!)}\!-1} \Tk \Uk \ge \Theta \sum_{k=k_{(\!n\!-\!1\!)}}^{k_n-1} \Tk \Uk
          \end{equation}
        and
          \begin{equation}\label{13}
            \sum_{k=\mu}^{k_n-1} \Tk \Uk \le \frac{\Theta}{\Theta-1} \sum_{k=k_{(\!n\!-\!1\!)}}^{k_n-1} \Tk \Uk.
          \end{equation}   
      \item[\rm(iii)]
        For every $n\in\A$ it holds
          \begin{equation}\label{14}
            \sum_{k=k_n}^{k_{(\!n\!+\!1\!)}\!-2} \Tk \Uk < \Theta \sum_{k=k_{(\!n\!-\!1\!)}}^{k_n-1} \Tk \Uk.
          \end{equation}
      \item[\rm(iv)]
        For every $n\in\N$, $k\in\Zm$ and $t\in (0,\infty]$ such that $n\le N$, $k\le k_{(\!n\!+\!1\!)}\!-1$ and $t\in(t_k,t_{(\!k\!+\!1\!)}]$ it holds
          \begin{equation}\label{15}  
            \int_{t_\mu}^t w(x) \Uxt \dx \lesssim \sum_{j=k_{(\!n\!-\!1\!)}}^{k_n-1} \Tj \Uj + \Tk \Utkt.
          \end{equation} 
        If the same conditions hold and it is even satisfied that $k\le k_{(\!n\!+\!1\!)}\!-2$, then 
          \begin{equation}\label{57}  
            \int_{t_\mu}^t w(x) \Uxt \dx \lesssim \sum_{j=k_{(\!n\!-\!1\!)}}^{k_n-1} \Tj \Uj.
          \end{equation}   
      \item[\rm(v)]
        Define $k_{(-1)}:=\mu-1$. Then for every $n\in\N$ such that $n\le N$ it holds
          \begin{equation}\label{16}  
            \sum_{j=k_{(\!n\!-\!1\!)}}^{k_n-1} \Tj \Uj \lesssim \int_{t_{k_{(\!n\!-\!2\!)}}}^{t_{k_n}} w(t) U^q(t,t_{k_n}) \dt.
          \end{equation}    
    \end{itemize}
\end{thm}

\begin{proof}
  At first, observe that it is indeed possible to choose the sequence $\{t_k\}$ with the required properties because the weight $w$ is locally integrable. Since $w$ may take zero values, the sequence $\{t_k\}$ need not be unique. In that case, we choose one fixed $\{t_k\}$ satisfying the requirements. From \eqref{17} we deduce that
    \begin{equation}\label{18}
      \Tk = \int_0^{t_k}w(s)\ds = \frac1{\Theta-1} \int_\dk w(s)\ds = \frac{\Theta}{\Theta-1} \int_{\Delta_{(k-1)}}\!\!w(s)\ds
    \end{equation}
  for all $k\in\Z$ such that $k\le K-1$.

  We proceed with the construction of the index subset $\{k_n\}$. Define $k_0:=\mu$ and $k_1:=\mu+1$ and continue inductively as follows.\\
  ($\ast$) Let $k_0,\ldots,k_n$ be already defined. Then
    \begin{itemize}
      \item[(a)] 
        If $k_n=K$, define $N:=n-1$ and stop the procedure.
      \item[(b)]
        If $k_n<K$ and there exists an~index $j$ such that $k_n<j\le K$ and 
          \begin{equation}\label{19}
            \sum_{k=k_n}^{j-1} \Tk \Uk \ge \Theta \sum_{k=k_{(\!n\!-\!1\!)}}^{k_n-1} \Tk \Uk,
          \end{equation} 
        then define $k_{(\!n\!+\!1\!)}$ as the smallest index $j$ for which \eqref{19} holds. Then proceed again with step ($\ast$) with $n+1$ in place of $n$.
      \item[(c)]
        If $k_n<K$ and and \eqref{19} holds for no index $j$ such that $k_n<j\le K$, then define $N:=n$, $k_{(\!n\!+\!1\!)}:=K$ and stop the procedure.
    \end{itemize}
  In this manner, one obtains a~finite sequence of indices $\{k_0,\ldots,k_N\}\subseteq\Zm$ and the final index $k_{(\!n\!+\!1\!)}=K$.
  
  We will call each interval $\dk$ the \emph{$k$-th segment}, and each interval $[t_{k_n},t_{(k_n+1)})$ the \emph{$n$-th block}. If $n\in\N$ is such that $n\le N$, then the $n$-th block either consists of the single $k_n$-th segment, in which case it holds
    \[
      k_{(\!n\!+\!1\!)}=k_{n}+1,
    \]
  or the $n$-th segment contains more than one segment and then
    \[
      k_{(\!n\!+\!1\!)}>k_{n}+1,
    \]
  If the $n$-th block is of the second type, then $n\in\A$, according to the definition \eqref{20}. Hence, \eqref{11} is satisfied, even though the set $\A$ may be empty. The relation \eqref{11} in plain words says that each segment is either the last one (i.e., with the highest index $k$) in a~block, or it belongs to a~block consisting of more than one segment and the investigated segment is not the last one of those. We have now proved (i). 
  
  The property \eqref{12} follows directly from the construction. If $n\in\N$ is such that $n\le N$, then by iterating \eqref{12} one gets
    \[
      \sum_{k=\mu}^{k_n-1} \Tk \Uk = \sum_{i=0}^{n-1} \sum_{k=k_i}^{k_{(\!i\!+\!1\!)}-1} \!\! \Tk \Uk \le \sum_{i=0}^{n-1} \Theta^{i-n+1} \!\!\!\! \sum_{k=k_{(\!n\!-\!1\!)}}^{k_n-1} \!\! \Tk \Uk \le \frac{\Theta}{\Theta\!-\!1} \sum_{k=k_{(\!n\!-\!1\!)}}^{k_n-1} \!\! \Tk \Uk.
    \]  
  Hence, \eqref{13} holds and (ii) is then proved. 
  
  Property (iii) is again a~direct consequence of the way the blocks were constructed. We proceed with proving (iv). Let $n\in\N$, $k\in\Zm$ and $t\in (0,\infty]$ be such that $n\le N$, $k\le k_{(\!n\!+\!1\!)}\!-1$ and $t\in(t_k,t_{(\!k\!+\!1\!)}]$. Then the following sequence of inequalities is valid:
    \begin{align}
      \int_{t_\mu}^t w(x) \Uxt \dx & =        \int_{t_\mu}^{t_k} w(x) \Uxt \dx + \int_{t_k}^t w(x) \Uxt \dx \nonumber\\
                                   & \lesssim \int_{t_\mu}^{t_k} w(x) \Uxtk \dx + \int_{t_\mu}^{t_k} w(x)\dx\ \Utkt + \int_{t_k}^t w(x) \Uxt \dx \nonumber\\
                                   & \le      \sum_{j=\mu}^{k-1} \int_{\dj} w(x)\dx\ U^q(t_j,t_k) + \int_{t_\mu}^{t_{(\!k\!+\!1\!)}} w(x)\dx\ \Utkt \nonumber\\
                                   & \lesssim \sum_{j=\mu}^{k-1} \Tj U^q(t_j,t_k) + \Tk \Utkt \label{21}\\
                                   & \lesssim \sum_{j=\mu}^{k-1} \Tj \Uj + \Tk \Utkt. \label{22}
    \end{align}
  In here, step \eqref{21} follows by \eqref{18}, and step \eqref{22} by Proposition \ref{4}. If $k\le k_n$, then
    \[
      \sum_{j=\mu}^{k-1} \Tj \Uj \le \sum_{j=\mu}^{k_n-1} \Tj \Uj \lesssim \sum_{j=k_{(\!n\!-\!1\!)}}^{k_n-1} \Tj \Uj.
    \]
  The second inequality here follows by \eqref{13}. If $k>k_n$, then $n\in\A$, $k_n+1\le k \le k_{(\!n\!+\!1\!)}\!-1$ and it holds
    \[
      \sum_{j=\mu}^{k-1} \Tj \Uj \le \sum_{j=\mu}^{k_{(\!n\!+\!1\!)}\!-2} \Tj \Uj = \sum_{j=\mu}^{k_n-1} \Tj \Uj + \sum_{j=k_n}^{k_{(\!n\!+\!1\!)}\!-2} \Tj \Uj \lesssim \sum_{j=k_{(\!n\!-\!1\!)}}^{k_n-1} \Tj \Uj.
    \]
  The last inequality is granted by \eqref{13} and \eqref{14}. We have proved that
    \[     
      \sum_{j=\mu}^{k-1} \Tj \Uj \lesssim \sum_{j=k_{(\!n\!-\!1\!)}}^{k_n-1} \Tj \Uj.
    \]
  Applying this in the inequality obtained at \eqref{22}, we get the estimate \eqref{15}. If we now add the assumption $k\le k_{(\!n\!+\!1\!)}\!-2$, then \eqref{15} still holds and, in addition to that, we get
    \[
      \Tk \Utkt \le \Tk \Uk \le \sum_{j=\mu}^{k_{(\!n\!+\!1\!)}\!-2} \Tj \Uj \lesssim \sum_{j=k_{(\!n\!-\!1\!)}}^{k_n-1} \Tj \Uj.
    \]
  In here, the last inequality follows from \eqref{13} and \eqref{14}. Applying this result to \eqref{15}, we obtain \eqref{57} and (iv) is thus proved. 
  
  To prove (v), let $n\in\N$ be such that $n\le N$ and observe the following:
    \begin{align*}
      \sum_{j=k_{(\!n\!-\!1\!)}}^{k_n-1} \Tj \Uj & \lesssim \sum_{j=k_{(\!n\!-\!1\!)}}^{k_n-1}\ \int_{\Delta_{j-1}} w(t)\dt\ \Uj \le \sum_{j=k_{(\!n\!-\!1\!)}}^{k_n-1}\ \int_{\Delta_{j-1}} w(t) U^q(t,t_{k_n}) \dt \\
      & = \int_{t_{(\!k_{(\!n\!-\!1\!)}-1\!)}}^{t_{(k_n-1)}} w(t) U^q(t,t_{k_n}) \dt \le \int_{t_{ k_{(\!n\!-\!2\!)}}}^{t_{k_n}} w(t) U^q(t,t_{k_n}) \dt.
    \end{align*} 
  In the first step, \eqref{18} was used. In the last one, we used the inequality $t_{ k_{(\!n\!-\!2\!)}} \le t_{(\! k_{(\!n\!-\!1\!)}-1 \!)}$ which follows from \eqref{10}.   
\end{proof}

\begin{proof}[of Theorem \ref{7}]
  Without loss of generality, we may assume that $\theta\in[1,\infty)$. Indeed, if the kernel $U$ is $\theta$-regular with $\theta\in(0,1)$, then $U$ is obviously also $1$-regular. \\
  
  ``(ii)$\Rightarrow$(i)''. Assume that $D_1<\infty$ and $D_2<\infty$. Let us prove that \eqref{8} holds for all $f\in\MM$ with the least constant $C$ satisfying $C^r \lesssim D_1 + D_2$. 
  
  At first, let us assume that there exists $K\in\Z$ such that $\int_0^\infty w = 2^K$. Let $\mu\in\Z$ be such that $\mu\le K-2$ and define $\Zm$ by \eqref{23}. Let $\{t_k\}_{k=-\infty}^{K}\subset (0,\infty]$ be a~sequence of points such that $t_K=\infty$ and \eqref{18} holds for all $k\in\Z$ such that $k\le K$. Let $\{k_n\}_{n=0}^N\subset\Zm$ be the subsequence of indices granted by Theorem \ref{9}. Related notation from Theorem \ref{9} will be used in what follows as well. Suppose that $f\in \MM\cap L^p(v)$. Then
    \Bdef{51}\Bdef{52}\Bdef{53}
    \begin{align}
      \int_{\tm}^\infty \! \lt \int_t^\infty \!\!\! f(x)U(t,x)\dx \rt^q w(t) \dt       & =         \sum_{k\in\Zm} \int_{\dk} \lt \int_t^\infty f(x)U(t,x)\dx \rt^q w(t) \dt \nonumber\\
      & \lesssim  \sum_{k\in\Zm} \Tk \lt \int_{\tk}^\infty f(x)U(\tk,x)\dx \rt^q \label{1080}\\
      & \lesssim  \sum_{n=0}^N \sum_{k=k_n}^{k_{(\!n\!+\!1\!)}\!-1} \Tk \lt \int_{\tk}^{t_{k_{(\!n\!+\!1\!)}}} f(x)U(\tk,x)\dx \rt^q \nonumber\\
      & \quad +   \sum_{n=0}^{N-1} \sum_{k=k_n}^{k_{(\!n\!+\!1\!)}\!-1} \Tk \lt \int_{t_{k_{(\!n\!+\!1\!)}}}^\infty f(x)U(\tk,x)\dx \rt^q \nonumber\\
      & \lesssim  \sum_{n=0}^N \sum_{k=k_n}^{k_{(\!n\!+\!1\!)}\!-1} \Tk \lt \int_{\tk}^{t_{k_{(\!n\!+\!1\!)}}} f(x)U(\tk,x)\dx \rt^q \nonumber\\
      & \quad +   \sum_{n=0}^{N-1} \sum_{k=k_n}^{k_{(\!n\!+\!1\!)}\!-1} \!\! \Tk \Uq(\tk,t_{k_{(\!n\!+\!1\!)}})\lt \int_{t_{k_{(\!n\!+\!1\!)}}}^\infty \hspace{-8pt} f(x)\dx \rt^q \nonumber\\
      & \quad +   \sum_{n=0}^{N-1} \sum_{k=k_n}^{k_{(\!n\!+\!1\!)}\!-1} \Tk \lt \int_{t_{k_{(\!n\!+\!1\!)}}}^\infty \hspace{-8pt} f(x)U(t_{k_{(\!n\!+\!1\!)}},x)\dx \rt^q \nonumber\\
      & =:        \B{51} + \B{52} + \B{53}.\nonumber
    \end{align}
  Inequality \eqref{1080} follows from \eqref{18}. Furthermore, we have
    \Bdef{54}\Bdef{55}
    \begin{align}
      \B{51} & =        \sum_{n=0}^N \sum_{k=k_n}^{k_{(\!n\!+\!1\!)}\!-1} \Tk \lt \int_{\tk}^{t_{k_{(\!n\!+\!1\!)}}} f(x)U(\tk,x)\dx \rt^q \nonumber\\ 
             & \lesssim \sum_{n=0}^N \Theta^{k_{(\!n\!+\!1\!)}\!-1} \lt \int_{\Delta_{(\!k_{(\!n\!+\!1\!)}\!-1\!)}} \hspace{-15pt} f(x)U(t_{(\!k_{(\!n\!+\!1\!)}\!-1\!)},x)\dx \rt^q + \sum_{n\in\A} \sum_{k=k_n}^{k_{(\!n\!+\!1\!)}\!-2} \! \Tk \lt \int_{\tk}^{t_{k_{(\!n\!+\!1\!)}}} f(x)U(\tk,x)\dx \rt^q \nonumber\\            
             & \lesssim \sum_{n=0}^N \Theta^{k_{(\!n\!+\!1\!)}\!-1} \lt \int_{\Delta_{(\!k_{(\!n\!+\!1\!)}\!-1\!)}} \hspace{-15pt} f(x)U(t_{(\!k_{(\!n\!+\!1\!)}\!-1\!)},x)\dx \rt^q + \sum_{n\in\A} \sum_{k=k_n}^{k_{(\!n\!+\!1\!)}\!-2} \! \Tk \lt \int_{\Delta_{(\!k_{(\!n\!+\!1\!)}\!-1\!)}} \hspace{-10pt} f(x)U(\tk,x)\dx \rt^q \nonumber\\
             & \quad +  \sum_{n\in\A} \sum_{k=k_n}^{k_{(\!n\!+\!1\!)}\!-2} \Tk \lt \int_{\tk}^{t_{(\!k_{(\!n\!+\!1\!)}\!-1\!)}} \!\! f(x)U(\tk,x)\dx \rt^q \nonumber\\
             & \lesssim \sum_{n=0}^N \Theta^{k_{(\!n\!+\!1\!)}\!-1} \lt \int_{\Delta_{(\!k_{(\!n\!+\!1\!)}\!-1\!)}} \hspace{-18pt} f(x)U(t_{(\!k_{(\!n\!+\!1\!)}\!-1\!)},x)\dx \rt^q + \sum_{n\in\A} \! \sum_{k=k_n}^{k_{(\!n\!+\!1\!)}\!-2} \! \Tk \lt \int_{\Delta_{(\!k_{(\!n\!+\!1\!)}\!-1\!)}} \hspace{-18pt} f(x)U(t_{(\!k_{(\!n\!+\!1\!)}\!-1\!)},x)\dx \rt^q \nonumber\\
             & \quad +  \sum_{n\in\A} \! \sum_{k=k_n}^{k_{(\!n\!+\!1\!)}\!-2} \!\!\! \Tk \Uq(\tk,t_{(\!k_{(\!n\!+\!1\!)}\!-1\!)})\lt \int_{\Delta_{(\!k_{(\!n\!+\!1\!)}\!-1\!)}} \hspace{-20pt} f(x)\dx \rt^q + \sum_{n\in\A} \! \sum_{k=k_n}^{k_{(\!n\!+\!1\!)}\!-2} \!\!\! \Tk \lt \int_{\tk}^{t_{(\!k_{(\!n\!+\!1\!)}\!-1\!)}} \!\!\! f(x)U(\tk,x)\dx \rt^q \nonumber\\
             & \lesssim \sum_{n=0}^N \Theta^{k_{(\!n\!+\!1\!)}} \lt \int_{\Delta_{(\!k_{(\!n\!+\!1\!)}\!-1\!)}} \hspace{-20pt} f(x)U(t_{(\!k_{(\!n\!+\!1\!)}\!-1\!)},x)\dx \rt^q + \sum_{n\in\A} \! \sum_{k=k_n}^{k_{(\!n\!+\!1\!)}\!-2} \!\!\! \Tk \Uq(\tk,t_{(\!k_{(\!n\!+\!1\!)}\!-1\!)}) \lt \int_{\tk}^{t_{k_{(\!n\!+\!1\!)}}} \!\!\! f(x)\dx \rt^q \nonumber\\
             & =:       \B{54} + \B{55}.\nonumber
    \end{align}  
  For the role of the symbol $\A$, see \eqref{20}. In the next step, for formal reasons define $t_{(k_{(\!N+2)}-1)}:=\infty$. Then we get
    \begin{align}
      \B{54} & =        \sum_{n=0}^N \Theta^{k_{(\!n\!+\!1\!)}} \lt \int_{\Delta_{(\!k_{(\!n\!+\!1\!)}\!-1\!)}} \hspace{-12pt} f(x)U(t_{(\!k_{(\!n\!+\!1\!)}\!-1\!)},x)\dx \rt^q \nonumber\\
             & \le      \sum_{n=0}^N \Theta^{k_{(\!n\!+\!1\!)}} \lt \int_{\Delta_{(\!k_{(\!n\!+\!1\!)}\!-1\!)}} \hspace{-12pt} \Upp(t_{(\!k_{(\!n\!+\!1\!)}\!-1\!)},x) \vp\!(x)\dx \rt^{\frac{q}{p'}} \lt \int_{\Delta_{(\!k_{(\!n\!+\!1\!)}\!-1\!)}} \hspace{-15pt} f^p(x)v(x)\dx \rt^{\frac qp} \label{109}\\
             & \le      \lt \sum_{n=0}^N \Theta^{\rq k_{(\!n\!+\!1\!)}} \lt \int_{\Delta_{(\!k_{(\!n\!+\!1\!)}\!-1\!)}} \hspace{-17pt} \Upp\!(t_{(\!k_{(\!n\!+\!1\!)}\!-1\!)},x) \vp\!(x)\dx \rt^{\!\frac{r}{p'}} \rt^{\!\frac{q}r} \! \lt \sum_{n=0}^N \int_{\Delta_{(\!k_{(\!n\!+\!1\!)}\!-1\!)}} \hspace{-17pt} f^p(x)v(x)\dx \rt^{\frac qp} \label{110}\\
             & \le      \lt \sum_{n=0}^N \Theta^{\rq k_{(\!n\!+\!1\!)}} \lt \int_{\Delta_{(\!k_{(\!n\!+\!1\!)}\!-1\!)}} \hspace{-15pt} \Upp(t_{(\!k_{(\!n\!+\!1\!)}\!-1\!)},x) \vp(x)\dx \rt^{\frac{r}{p'}} \rt^{\frac{q}r} \|f\|_{L^p(v)}^q \nonumber\\
             & \lesssim \lt \sum_{n=0}^N \lt \int_{\Delta_{(\!k_{(\!n\!+\!1\!)}-2\!)}} \hspace{-10pt} w(x)\dx \rt^\rq \lt \int_{\Delta_{(\!k_{(\!n\!+\!1\!)}\!-1\!)}} \hspace{-15pt} \Upp(t_{(\!k_{(\!n\!+\!1\!)}\!-1\!)},x) \vp\!(x)\dx \rt^{\!\frac{r}{p'}} \rt^{\!\frac{q}r} \|f\|_{L^p(v)}^q \label{111}\\
             & \le      \lt \sum_{n=0}^N \lt \int_{t_{(\!k_n\!-\!1\!)}}^{t_{(\!k_{(\!n\!+\!1\!)}\!-1\!)}} \hspace{-5pt} w(x)\dx \rt^{\!\rq} \lt \int_{t_{(\!k_{(\!n\!+\!1\!)}\!-1\!)}}^{t_{(\!k_{(\!n\!+\!2\!)}-\!1\!)}}  \Upp\!\!(t_{(\!k_{(\!n\!+\!1\!)}\!-1\!)},x) \vp\!(x)\dx \rt^{\!\frac{r}{p'}} \rt^{\!\frac{q}r} \|f\|_{L^p(v)}^q \label{112}\\
             & \le      D_1^{\frac qr} \|f\|^q_{L^p(v)}.\nonumber
    \end{align}
  The H\"older inequality for functions was used in \eqref{109}, and its discrete version (see Proposition \ref{108}) was used in \eqref{110}. Step \eqref{111} follows from \eqref{18}. In \eqref{112} we used the inequalities $t_{(k_n\!-1)} \le t_{(k_{(\!n\!+\!1\!)}-2)}$ and $t_{k_{(\!n\!+\!1\!)}} \le t_{(k_{(\!n\!+\!2\!)}-1)}$ which hold for all $n\in\{0,\ldots,N\}$ and both follow from \eqref{10} or the additional formal definition in the case $n=N$. Step \eqref{112} ensures that the sequence $\{t_{(k_n-1)}\}_{n=0}^N$ can be extended into a~covering sequence (formally, $\{t_{(k_n-1)}\}_{n=0}^N$ itself is not a~covering sequence since $t_{(k_0\!-1)}=t_{(\mu\!-1\!)}>0$).     
  
  Regarding the term $\B{55}$, one has 
    \begin{align}
      \B{55} & =        \sum_{n\in\A} \sum_{k=k_n}^{k_{(\!n\!+\!1\!)}\!-2} \!\!\! \Tk \Uq(\tk,t_{(\!k_{(\!n\!+\!1\!)}\!-1\!)}) \lt \int_{\tk}^{t_{k_{(\!n\!+\!1\!)}}} \!\! f(x)\dx \rt^q \nonumber\\  
             & \le      \sum_{n\in\A} \sum_{k=k_n}^{k_{(\!n\!+\!1\!)}\!-2} \!\!\! \Tk \Uq(\tk,t_{(\!k_{(\!n\!+\!1\!)}\!-1\!)}) \lt \int_{t_{k_n}}^{t_{k_{(\!n\!+\!1\!)}}} \!\! f(x)\dx \rt^q \nonumber\\  
             & \lesssim \sum_{n\in\A} \sum_{k=k_n}^{k_{(\!n\!+\!1\!)}\!-2} \!\!\! \Tk \Uq(\dk) \lt \int_{t_{k_n}}^{t_{k_{(\!n\!+\!1\!)}}} \!\! f(x)\dx \rt^q \label{113}\\  
             & \le      \sum_{n\in\A} \sum_{k=k_n}^{k_{(\!n\!+\!1\!)}\!-2} \!\!\! \Tk \Uq(\dk) \lt \int_{t_{k_n}}^{t_{k_{(\!n\!+\!1\!)}}} \!\! \vp(x)\dx \rt^\frac{q}{p'} \lt \int_{t_{k_n}}^{t_{k_{(\!n\!+\!1\!)}}} \!\! f^p(x)v(x)\dx \rt^\frac{q}{p} \label{115}\\  
             & \le      \lt \sum_{n\in\A} \lt \sum_{k=k_n}^{k_{(\!n\!+\!1\!)}\!-2} \!\!\! \Tk \Uq(\dk) \rt^{\!\rq} \!\lt \int_{t_{k_n}}^{t_{k_{(\!n\!+\!1\!)}}} \!\! \vp\!(x)\dx \rt^{\!\frac{r}{p'}} \rt^{\!\frac qr} \lt \sum_{n\in\A} \int_{t_{k_n}}^{t_{k_{(\!n\!+\!1\!)}}} \!\! f^p(x)v(x)\dx \rt^\frac{q}{p} \label{116}\\  
             & \le      \lt \sum_{n\in\A} \lt \int_{t_{k_{(\!n\!-\!2\!)}}}^{t_{k_n}} w(t) U^q(t,t_{k_n}) \dt \rt^\rq \!\lt \int_{t_{k_n}}^{t_{k_{(\!n\!+\!1\!)}}} \!\! \vp\!(x)\dx \rt^{\!\frac{r}{p'}} \rt^{\!\frac qr} \|f\|_{L^p(v)}^q \label{117}\\  
             & \le      D_2^{\frac qr} \|f\|_{L^p(v)}^q. \nonumber       
    \end{align}
  Inequality \eqref{113} follows from Proposition \ref{4}. In steps \eqref{115} and \eqref{116} we used the appropriate versions of the H\"older inequality, cf.~Propositions \ref{107} and \ref{108}. Inequalities \eqref{14} and \eqref{16} give the estimate \eqref{117}. We proved
    \[
      \B{51} \lesssim \B{54} + \B{55} \lesssim (D_1 + D_2)^\frac qr \|f\|_{L^p(v)}^q.
    \]
  We continue with the term $\B{52}$.
    \begin{align}
      \B{52} & =        \sum_{n=0}^{N-1} \sum_{k=k_n}^{k_{(\!n\!+\!1\!)}\!-1} \!\! \Tk \Uq(\tk,t_{k_{(\!n\!+\!1\!)}})\lt \int_{t_{k_{(\!n\!+\!1\!)}}}^\infty \hspace{-8pt} f(x)\dx \rt^q \nonumber\\
             & \lesssim \sum_{n=0}^{N-1} \sum_{k=k_n}^{k_{(\!n\!+\!1\!)}\!-1} \!\! \Tk \Uq(\dk)\lt \int_{t_{k_{(\!n\!+\!1\!)}}}^\infty \hspace{-8pt} f(x)\dx \rt^q \label{118}\\
             & =        \sum_{n=0}^{N-1} \sum_{k=k_n}^{k_{(\!n\!+\!1\!)}\!-1} \!\! \Tk \Uq(\dk)\lt \sum_{i=n+1}^N \int_{t_{k_i}}^{t_{k_{(\!i\!+\!1\!)}}} \hspace{-2pt} f(x)\dx \rt^q \nonumber\\
             & \lesssim \sum_{n=0}^{N-1} \sum_{k=k_n}^{k_{(\!n\!+\!1\!)}\!-1} \!\! \Tk \Uq(\dk)\lt \int_{t_{k_{(\!n\!+\!1\!)}}}^{t_{k_{(\!n\!+\!2\!)}}} \hspace{-2pt} f(x)\dx \rt^q \label{119}\\
             & \le      \sum_{n=0}^{N-1} \sum_{k=k_n}^{k_{(\!n\!+\!1\!)}\!-1} \!\! \Tk \Uq(\dk)\lt \int_{t_{k_{(\!n\!+\!1\!)}}}^{t_{k_{(\!n\!+\!2\!)}}} \hspace{-2pt} \vp\!(x)\dx \rt^\frac{q}{p'} \lt \int_{t_{k_{(\!n\!+\!1\!)}}}^{t_{k_{(\!n\!+\!2\!)}}} \hspace{-2pt} f^p(x)v(x) \dx \rt^\frac{q}{p} \label{120}\\
             & \le      \lt \sum_{n=0}^{N-1} \lt \sum_{k=k_n}^{k_{(\!n\!+\!1\!)}\!-1} \!\! \Tk \Uq(\dk) \rt^{\!\rq} \! \lt \int_{t_{k_{(\!n\!+\!1\!)}}}^{t_{k_{(\!n\!+\!2\!)}}} \hspace{-2pt} \vp\!(x)\dx \rt^{\!\frac{r}{p'}} \rt^{\!\frac qr} \! \lt \sum_{n=0}^{N-1} \int_{t_{k_{(\!n\!+\!1\!)}}}^{t_{k_{(\!n\!+\!2\!)}}} \hspace{-2pt} f^p(x)v(x) \dx \rt^{\!\frac{q}{p}} \label{121}\\
             & \lesssim \lt \sum_{n=0}^{N-1} \lt \int_{t_{k_{(\!n\!-\!1\!)}}}^{t_{k_{(\!n\!+\!1\!)}}} w(t) U^q(t,t_{k_{(\!n\!+\!1\!)}}) \dt \rt^{\!\rq} \! \lt \int_{t_{k_{(\!n\!+\!1\!)}}}^{t_{k_{(\!n\!+\!2\!)}}} \hspace{-2pt} \vp\!(x)\dx \rt^{\!\frac{r}{p'}} \rt^{\!\frac qr} \|f\|_{L^p(v)}^q \label{122}\\
             & \le      D_2^{\frac qr} \|f\|_{L^p(v)}^q. \nonumber
    \end{align}         
  Step \eqref{118} follows from Proposition \ref{4}. Proposition \ref{3} supplied with \eqref{12} gives \eqref{119}. In \eqref{120} and \eqref{121} we used the H\"older inequality (see Propositions \ref{107} and \ref{108}). To get \eqref{122}, one uses \eqref{16}. We obtained
    \[
      \B{52} \lesssim D_2^\frac qr \|f\|_{L^p(v)}^q.            
    \]
  In what follows, without loss of generality we will assume that $N\ge 2$. If $N=1$, then the terms involving $\sum_{j=0}^{N-2}$ (or similar) are simply not present in the calculations below.
  
  The term $\B{53}$ is treated as follows.         
    \Bdef{56}\Bdef{57}         
    \begin{align}
      \B{53} & =        \sum_{n=0}^{N-1} \sum_{k=k_n}^{k_{(\!n\!+\!1\!)}\!-1} \Tk \lt \int_{t_{k_{(\!n\!+\!1\!)}}}^\infty \hspace{-8pt} f(x)U(t_{k_{(\!n\!+\!1\!)}},x)\dx \rt^q \nonumber\\
             & \lesssim \sum_{n=0}^{N-1} \Theta^{k_{(\!n\!+\!1\!)}} \lt \int_{t_{k_{(\!n\!+\!1\!)}}}^\infty \hspace{-8pt} f(x)U(t_{k_{(\!n\!+\!1\!)}},x)\dx \rt^q \nonumber\\
             & =        \sum_{n=0}^{N-1} \Theta^{k_{(\!n\!+\!1\!)}} \lt \sum_{i=n\!+\!1}^{N} \int_{t_{k_i}}^{t_{k_{(\!i\!+\!1\!)}}} \hspace{-2pt} f(x)U(t_{k_{(\!n\!+\!1\!)}},x)\dx \rt^q \nonumber\\
             & \lesssim \sum_{n=0}^{N-1} \Theta^{k_{(\!n\!+\!1\!)}} \lt \sum_{i=n\!+\!1}^{N} \int_{t_{k_i}}^{t_{k_{(\!i\!+\!1\!)}}} \hspace{-3pt} f(x)U(t_{k_i},x)\dx \rt^q \! + \sum_{n=0}^{N-2} \Theta^{k_{(\!n\!+\!1\!)}} \lt \sum_{i=n\!+\!2}^{N} U(t_{k_{(\!n\!+\!1\!)}},t_{k_i}) \! \int_{t_{k_i}}^{t_{k_{(\!i\!+\!1\!)}}} \hspace{-3pt} f(x)\dx \rt^q \nonumber\\
             & =:       \B{56} + \B{57}.\nonumber
    \end{align}
  Furthermore, it holds
    \begin{align}
      \B{56} & =        \sum_{n=0}^{N-1} \Theta^{k_{(\!n\!+\!1\!)}} \lt \sum_{i=n\!+\!1}^{N} \int_{t_{k_i}}^{t_{k_{(\!i\!+\!1\!)}}} \hspace{-3pt} f(x)U(t_{k_i},x)\dx \rt^q \nonumber\\
             & \lesssim \sum_{n=0}^{N-1} \Theta^{k_{(\!n\!+\!1\!)}} \lt \int_{t_{k_{(\!n\!+\!1\!)}}}^{t_{k_{(\!n\!+\!2\!)}}} \hspace{-3pt} f(x)U(t_{k_{(\!n\!+\!1\!)}},x)\dx \rt^q \label{123}\\
             & \le      \sum_{n=0}^{N-1} \Theta^{k_{(\!n\!+\!1\!)}} \lt \int_{t_{k_{(\!n\!+\!1\!)}}}^{t_{k_{(\!n\!+\!2\!)}}} \hspace{-3pt} U(t_{k_{(\!n\!+\!1\!)}},x)\vp\!(x)\dx \rt^{\frac q{p'}} \lt \int_{t_{k_{(\!n\!+\!1\!)}}}^{t_{k_{(\!n\!+\!2\!)}}} \hspace{-3pt} f^p(x)v(x) \dx \rt^\frac{q}{p} \label{124}\\
             & \le      \lt \sum_{n=0}^{N-1} \Theta^{\rq k_{(\!n\!+\!1\!)}} \lt \int_{t_{k_{(\!n\!+\!1\!)}}}^{t_{k_{(\!n\!+\!2\!)}}} \hspace{-3pt} U(t_{k_{(\!n\!+\!1\!)}},x) \vp\!(x) \dx \rt^{\frac r{p'}} \rt^{\frac qr} \lt \sum_{n=0}^{N-1} \int_{t_{k_{(\!n\!+\!1\!)}}}^{t_{k_{(\!n\!+\!2\!)}}} \hspace{-3pt} f^p(x)v(x) \dx \rt^\frac{q}{p} \label{125}\\
             & \lesssim \lt \sum_{n=0}^{N-1} \lt \int_{t_{(\!k_{(\!n\!+\!1\!)}\!-1\!)}}^{t_{k_{(\!n\!+\!1\!)}}} w(x) \dx \rt^\rq \lt \int_{t_{k_{(\!n\!+\!1\!)}}}^{t_{k_{(\!n\!+\!2\!)}}} \hspace{-3pt} U(t_{k_{(\!n\!+\!1\!)}},x) \vp\!(x) \dx \rt^{\frac r{p'}} \rt^{\frac qr} \|f\|_{L^p(v)}^q \label{126}\\
             & \le      \lt \sum_{n=0}^{N-1} \lt \int_{t_{k_n}}^{t_{k_{(\!n\!+\!1\!)}}} w(x) \dx \rt^\rq \lt \int_{t_{k_{(\!n\!+\!1\!)}}}^{t_{k_{(\!n\!+\!2\!)}}} \hspace{-3pt} U(t_{k_{(\!n\!+\!1\!)}},x) \vp\!(x) \dx \rt^{\frac r{p'}} \rt^{\frac qr} \|f\|_{L^p(v)}^q \label{127}\\
             & \le      D_1^{\frac qr} \|f\|_{L^p(v)}^q.\nonumber  
    \end{align}   
  Step \eqref{123} follows by Proposition \ref{3}. As usual, in \eqref{124} and \eqref{125} we used the H\"older inequality. The inequality \eqref{126} is granted by \eqref{18}, and \eqref{127} is a~consequence of \eqref{10}.
  
  Next, for the term $\B{57}$ we have
    \begin{align}
      \B{57} & =        \sum_{n=0}^{N-2} \Theta^{k_{(\!n\!+\!1\!)}} \lt \sum_{i=n\!+\!2}^{N} U(t_{k_{(\!n\!+\!1\!)}},t_{k_i}) \! \int_{t_{k_i}}^{t_{k_{(\!i\!+\!1\!)}}} \hspace{-3pt} f(x)\dx \rt^q \nonumber\\
             & \le      \sum_{n=0}^{N-2} \Theta^{k_{(\!n\!+\!1\!)}} \sum_{i=n\!+\!2}^{N} \Uq(t_{k_{(\!n\!+\!1\!)}},t_{k_i}) \lt \int_{t_{k_i}}^{t_{k_{(\!i\!+\!1\!)}}} \hspace{-3pt} f(x)\dx \rt^q \label{128}\\
             & \le      \sum_{i=2}^{N} \sum_{n=0}^{i-2} \Theta^{k_{(\!n\!+\!1\!)}} \Uq(t_{k_{(\!n\!+\!1\!)}},t_{k_i}) \lt \int_{t_{k_i}}^{t_{k_{(\!i\!+\!1\!)}}} \hspace{-3pt} f(x)\dx \rt^q \nonumber\\
             & \le      \sum_{i=2}^{N} \sum_{k=\mu}^{k_i-1} \Tk \Uq(\tk,t_{k_i}) \lt \int_{t_{k_i}}^{t_{k_{(\!i\!+\!1\!)}}} \hspace{-3pt} f(x)\dx \rt^q \nonumber\\
             & \lesssim \sum_{i=2}^{N} \sum_{k=\mu}^{k_i-1} \Tk \Uk \lt \int_{t_{k_i}}^{t_{k_{(\!i\!+\!1\!)}}} \hspace{-3pt} f(x)\dx \rt^q \label{1280}\\
             & \le      \sum_{i=2}^{N} \sum_{k=\mu}^{k_i-1} \Tk \Uk \lt \int_{t_{k_i}}^{t_{k_{(\!i\!+\!1\!)}}} \hspace{-3pt} \vp\!(x)\dx \rt^{\frac q{p'}} \lt \int_{t_{k_i}}^{t_{k_{(\!i\!+\!1\!)}}} \hspace{-3pt} f^p(x)v(x)\dx \rt^{\frac qp} \label{129}\\
             & \le      \lt \sum_{i=2}^{N} \lt \sum_{k=\mu}^{k_i-1} \Tk \Uk \rt^{\rq} \lt \int_{t_{k_i}}^{t_{k_{(\!i\!+\!1\!)}}} \hspace{-3pt} \vp\!(x)\dx \rt^{\frac r{p'}} \rt^{\frac qr} \lt \sum_{n=0}^{N-1} \int_{t_{k_i}}^{t_{k_{(\!i\!+\!1\!)}}} \hspace{-3pt} f^p(x)v(x)\dx \rt^{\frac qp} \label{130}\\
             & \lesssim \lt \sum_{i=2}^{N} \lt \int_{t_{k_{(\!i\!-\!2\!)}}}^{t_{k_i}} w(t) U^q(t,t_{k_i}) \dt \rt^{\!\rq} \lt \int_{t_{k_i}}^{t_{k_{(\!i\!+\!1\!)}}} \hspace{-3pt} \vp\!(x)\dx \rt^{\frac r{p'}} \rt^{\frac qr} \|f\|_{L^p(v)}^q \label{131}\\
             & \le      D_2^{\frac qr} \|f\|_{L^p(v)}^q.\nonumber
    \end{align}     
  Inequality \eqref{128} follows from concavity of the $q$-th power for $q<1$. In \eqref{1280} one uses Proposition \ref{4}. The H\"older inequality gives \eqref{129} and \eqref{130}. Estimate \eqref{131} follows from \eqref{13} and \eqref{16}. We proved    
    \[
      \B{53} \lesssim \B{56} + \B{57} \lesssim (D_1 + D_2)^\frac qr \|f\|_{L^p(v)}^q.
    \]                      
  Combined with the other estimates of $\B{51}$ and $\B{52}$, this yields 
    \[
      \int_{\tm}^\infty \! \lt \int_t^\infty \!\!\! f(x)U(t,x)\dx \rt^q w(t) \dt \lesssim (D_1 + D_2)^\frac qr \|f\|_{L^p(v)}^q.
    \]
  Observe that the constant related to the symbol ``$\lesssim$'' in here does not depend on the choice of $\mu$. The reader may nevertheless notice that the construction of the $n$-blocks in fact depends on $\mu$. However, the constants in the ``$\lesssim$''-estimates proved with help of that construction are indeed independent of $\mu$. Hence, we may perform the limit pass $\mu\to -\infty$. Since $\tm\to 0$ as $\mu\to -\infty$, the monotone convergence theorem (and taking the $q$-th root) yields     
    \[
      \lt \int_{0}^\infty \! \lt \int_t^\infty \!\!\! f(x)U(t,x)\dx \rt^q w(t) \dt \rt^\jq \lesssim (D_1 + D_2)^\frac 1r \|f\|_{L^p(v)}
    \]   
  for the fixed function $f\in \MM\cap L^p(v)$. Since the function $f$ was chosen arbitrarily and the constant represented in ``$\lesssim$'' does not depend on $f$, the inequality \eqref{8} holds with $C=(D_1 + D_2)^\frac 1r$ for all functions $f\in \MM$. Clearly, if $C$ is the least constant such that \eqref{8} holds for all $f\in\MM$, then
    \begin{equation}\label{132}
      C^r \lesssim D_1 + D_2.
    \end{equation}
  At this point, recall that so far we have assumed that $\int_0^\infty w(x)\dx = \Theta^K$ for a~$K\in\Z$. Let us now complete the proof of this part for a~general weight $w$.
  
  At first, if $\int_0^\infty w(x)\dx$ is finite but not equal to any integer power of $\Theta$, the result is simply obtained by multiplying $w$ by a~constant $c\in(1,2)$ such that $\int_0^\infty cw(x)\dx = \Theta^K$ for a~$K\in\Z$, and then using homogeneity of the expressions $\int_{0}^\infty \! \lt \int_t^\infty \!\!\! f(x)U(t,x)\dx \rt^q w(t) \dt$, $D^{\frac qr}_1$ and $D^{\frac qr}_2$ with respect to $w$. 
  
  Finally, let us suppose that $\int_0^\infty w(x)\dx=\infty$. Choose an~arbitrary function $f\in \MM\cap L^p(v)$. For each $m\in\N$ define $w_m:=w\chi_{[0,m]}$ and denote by $D_{1,m}$ the expression $D_1$ with $w$ replaced by $w_m$. Similarly we define $D_{2,m}$. Since the weight $w$ is locally integrable, for each $m\in\N$ it holds $\int_0^\infty w_m(x)\dx<\infty$. Hence, by  the previous part of the proof we get
    \[
      \lt \int_{0}^\infty \! \lt \int_t^\infty \!\!\! f(x)U(t,x)\dx \rt^q w_m(t) \dt \rt^\jq \lesssim (D_{1,m} + D_{2,m})^\frac 1r \|f\|_{L^p(v)}.
    \]
  Obviously, for all $m\in\N$ it holds $w_m \le w$ pointwise, hence $D_{1,m} \le D_1$ and $D_{2,m} \le D_2$. Thus, we get
    \[
      \lt \int_{0}^\infty \! \lt \int_t^\infty \!\!\! f(x)U(t,x)\dx \rt^q w_m(t) \dt \rt^\jq \lesssim (D_1 + D_2)^\frac 1r \|f\|_{L^p(v)}.
    \]
  The constant in ``$\lesssim$'' does not depend on $m$ or $f$ and the latter was arbitrarily chosen. Since $w_m\uparrow w$ pointwise as $m\to\infty$, the monotone convergence theorem (for $m\to\infty$) yields that \eqref{8} holds for all functions $f\in \MM$ and the best constant $C$ in \eqref{8} satisfies \eqref{132}. The proof of this part is now complete.\\ 
  
  ``(i)$\Rightarrow$(ii)''. Suppose that \eqref{8} holds for all $f\in \MM$ and $C\in(0,\infty)$ is the least constant such that this is true. We need to show that $D_1 + D_2 \lesssim C^r$.
  
  Let $\{t_k\}_{k\in\I}$ be a~covering sequence indexed by a~set $\I=\{k_{\min},\ldots,k_{\max}\}\subset\Z$. At first, let us show that
    \begin{equation}\label{133}
      \int_{\tk}^{t_{(\!k\!+\!1\!)}} \Upp(\tk,x)\vp(x)\dx < \infty \quad \textnormal{for\ all\ }k\in\I_0.
    \end{equation}
  Suppose, for a~contradiction, that $k\in\I_0$ and $\int_{\tk}^{t_{(\!k\!+\!1\!)}} \Upp(\tk,x)\vp(x)\dt = \infty$. Then, by Proposition \ref{107}, for every $M\in\N$ there exists a~function $g_M$ supported in $[{\tk},{t_{(\!k\!+\!1\!)}}]$ and such that $\int_{\tk}^{t_{(\!k\!+\!1\!)}} g_M^p(t)v(t)\dt = 1$ and $\int_{\tk}^{t_{(\!k\!+\!1\!)}} g_M(x)U(\tk,x)\dx > M$. Since $\tk>0$, by definition of a~weight it holds $\int_0^{\tk} w(t)\dt >0$. Thus, for every $M\in\N$ one gets
    \begin{align*}
      \lt \int_0^\infty \lt \int_t^\infty g_M(x)U(t,x)\dx \rt^q w(t)\dt \rt^\jq & \ge \lt \int_0^{\tk} w(t)\dt \rt^\jq \int_{\tk}^{t_{(\!k\!+\!1\!)}} g_M(x)U(\tk,x)\dx \\
                                                                                & >   M \lt \int_0^{\tk} w(t)\dt \rt^\jq \\
                                                                                & =   M \lt \int_0^{\tk} w(t)\dt \rt^\jq \|g_M\|_{L^p(v)},
    \end{align*}
  which contradicts \eqref{8}. Hence, \eqref{133} must be satisfied. Since $\{t_k\}$ was chosen arbitrarily, \eqref{133} together with local integrability of $w$ is in fact sufficient to prove that $D_1<\infty$. However, we aim to prove a~stronger assertion, namely that $D_1\lesssim C^r$. To do so, we proceed as follows.
  
  Having verified \eqref{133}, for each $k\in\I_0$ we may use Proposition \ref{107} to find a~measurable function $g_k$ supported in $[\tk, t_{(\!k\!+\!1\!)}]$ and such that $\|g_k\|_{L^p(v)}=1$ as well as
    \begin{equation}\label{134}
      \lt \int_{\tk}^{t_{(\!k\!+\!1\!)}} \Upp(\tk,x)\vp(x)\dx \rt^{\frac 1{p'}} \le 2 \int_{\tk}^{t_{(\!k\!+\!1\!)}} g_k(x) U(t_k,x)\dx.
    \end{equation}
  Furthermore, it holds
    \[
      \sum_{k\in\I_0} \lt \int_{t_{(\!k\!-\!1\!)}}^{\tk} w(t)\dt \rt^\rq \lt \int_{\tk}^{t_{(\!k\!+\!1\!)}} \Upp(\tk,x)\vp(x)\dx \rt^{\rpp} <\infty
    \] 
  since $w$ is locally integrable, \eqref{133} holds and $\I_0$ consists of a~finite number of indices. Hence, by Proposition \ref{108} we can find a~nonnegative sequence $\{c_k\}_{k\in\I_0}$ such that $\sum_{k\in\I_0} c_k^p = 1$ and
    \begin{multline}
      \lt \sum_{k\in\I_0} \lt \int_{t_{(\!k\!-\!1\!)}}^{\tk} w(t)\dt \rt^\rq \lt \int_{\tk}^{t_{(\!k\!+\!1\!)}} \Upp(\tk,x)\vp(x)\dx \rt^{\rpp} \rt^{\jr} \\
        \le 2 \lt \sum_{k\in\I_0} c_k^q \int_{t_{(\!k\!-\!1\!)}}^{\tk} w(t)\dt \lt \int_{\tk}^{t_{(\!k\!+\!1\!)}} \Upp(\tk,x)\vp(x)\dx \rt^{\frac{q}{p'}} \rt^\jq. \label{135}   
    \end{multline}
  Define a~function $g:=\sum_{k\in\I_0} c_k g_k$ and recall that each $g_k$ is supported in $[\tk, t_{(\!k\!+\!1\!)}]$. Hence,
    \begin{equation}\label{136}
      \|g\|_{L^p(v)} = \lt \sum_{k\in\I_0} c_k^p \|g_k\|_{L^p(v)}^p \rt^\jp = \lt \sum_{k\in\I_0} c_k^p \rt^\jp = 1.
    \end{equation}
  Finally, we get the following estimate.
    \begin{align}
      &           \sum_{k\in\I_0} \lt \int_{t_{(\!k\!-\!1\!)}}^{\tk} w(t)\dt \rt^\rq \lt \int_{\tk}^{t_{(\!k\!+\!1\!)}} \Upp(\tk,x)\vp(x)\dx \rt^{\rpp} \nonumber\\
      & \lesssim  \lt \sum_{k\in\I_0} c_k^q \int_{t_{(\!k\!-\!1\!)}}^{\tk} w(t)\dt \lt \int_{\tk}^{t_{(\!k\!+\!1\!)}} \Upp(\tk,x)\vp(x)\dx \rt^{\frac{q}{p'}} \rt^\rq \label{137}\\
      & \lesssim  \lt \sum_{k\in\I_0} c_k^q \int_{t_{(\!k\!-\!1\!)}}^{\tk} w(t)\dt \lt \int_{\tk}^{t_{(\!k\!+\!1\!)}} U(\tk,x) g_k(x)\dx \rt^q \rt^\rq \label{138}\\
      & =         \lt \sum_{k\in\I_0} \int_{t_{(\!k\!-\!1\!)}}^{\tk} w(t)\dt \lt \int_{\tk}^{t_{(\!k\!+\!1\!)}} U(\tk,x) g(x)\dx \rt^q \rt^\rq \nonumber\\
      & \le       \lt \sum_{k\in\I_0} \int_{t_{(\!k\!-\!1\!)}}^{\tk} w(t) \lt \int_{t}^{t_{(\!k\!+\!1\!)}} U(t,x) g(x)\dx \rt^q \dt \rt^\rq \nonumber\\
      & \le       \lt \int_0^\infty w(t) \lt \int_{t}^{\infty} U(t,x) g(x)\dx \rt^q \dt \rt^\rq \nonumber\\
      & \le       C^r \|g\|_{L^p(v)}^r \label{139}\\
      & =         C^r. \label{140}
    \end{align}
  In steps \eqref{137}, \eqref{138}, \eqref{139} and \eqref{140} we used \eqref{135}, \eqref{134}, \eqref{8} and \eqref{136}, respectively. Since the covering sequence $\{t_k\}_{k\in\I}$ was chosen arbitrarily, by taking supremum over all covering sequences we obtain
    \[
      D_1 \lesssim C^r.
    \]
  In what follows, we are going to prove a~similar estimate for $D_2$. Again, let $\{t_k\}_{k\in\I}$ be a~covering sequence indexed by a~set $\I=\{k_{\min},\ldots,k_{\max}\}\subset\Z$. Then it holds
    \begin{equation}\label{142}
      \int_{t_{(\!k\!-\!1\!)}}^{\tk} w(t) \Uq(t,\tk)\dt <\infty \quad \textnormal{for\ all\ }k\in\I_0,
    \end{equation}
  and 
    \begin{equation}\label{141}
      \int_{\tk}^{t_{(\!k\!+\!1\!)}} \vp(x)\dx <\infty \quad \textnormal{for\ all\ }k\in\I_0.
    \end{equation}
  Let us prove these claims. At first, suppose that there exists $k\in\I_0$ such that $\int_{t_{(\!k\!-\!1\!)}}^{\tk} w(t) \Uq(t,\tk)\dt =\infty$. By definition, the weight $v$ is locally integrable, thus the function $\chi_{[\tk,\tk+1]}$ belongs to $L^p(v)$. Then
    \begin{align*}
      \infty & =         \int_{t_{(\!k\!-\!1\!)}}^{\tk} w(t) \Uq(t,\tk)\dt \\
             & =         \int_{t_{(\!k\!-\!1\!)}}^{\tk} w(t) \Uq(t,\tk)\dt \lt \int_{\tk}^\infty \chi_{[\tk,\tk+1]}(x)\dx \rt^q \\ 
             & =         \int_{t_{(\!k\!-\!1\!)}}^{\tk} w(t) \lt \int_{\tk}^\infty U(t,x) \chi_{[\tk,\tk+1]}(x)\dx \rt^q \dt \\ 
             & \le       \int_0^\infty w(t) \lt \int_t^\infty U(t,x) \chi_{[\tk,\tk+1]}(x)\dx \rt^q \dt, 
    \end{align*}
  whereas $\|\chi_{[\tk,\tk+1]}\|_{L^p(v)}<\infty$. This contradicts \eqref{8}, hence \eqref{142} holds. Next, assume that there exists $k\in\I_0$ such that $\int_{\tk}^{t_{(\!k\!+\!1\!)}} \vp(x)\dx =\infty.$ Then, by Proposition \ref{107}, for every $M\in\N$ there exists a~function $g_M$ such that $\|g_M\|_{L^p(v)}=1$ and $\int_{\tk}^{t_{(\!k\!+\!1\!)}} g_m(x) \dx > M$. By the definition of a~weight and a~$\theta$-regular kernel, the term $\int_{t_{(\!k\!-\!1\!)}}^{\tk} w(t) \Uq(t,\tk) \dt$ is strictly positive. We get
    \begin{align*}
      \int_0^\infty w(t) \lt \int_t^\infty U(t,x) g_M(x)\dx \rt^q \dt & \ge \int_{t_{(\!k\!-\!1\!)}}^{\tk} w(t) \lt \int_{\tk}^{t_{(\!k\!+\!1\!)}} U(t,x) g_M(x)\dx \rt^q \dt \\
                                                                      & \ge \int_{t_{(\!k\!-\!1\!)}}^{\tk} w(t) \Uq(t,\tk) \dt \lt \int_{\tk}^{t_{(\!k\!+\!1\!)}} g_M(x)\dx \rt^q  \\
                                                                      & \ge M^q \int_{t_{(\!k\!-\!1\!)}}^{\tk} w(t) \Uq(t,\tk) \dt \\
                                                                      & =   M^q \int_{t_{(\!k\!-\!1\!)}}^{\tk} w(t) \Uq(t,\tk) \dt\ \|g_M\|_{L^p(v)}
    \end{align*}
  for all $M\in\N$. This is a~contradiction with \eqref{8}. Hence, \eqref{141} must be true. 
  
  Thanks to \eqref{142}, Proposition \ref{107} yields that for every $k\in\I_0$ we can find a~function $h_k$ supported in $[\tk, t_{(\!k\!+\!1\!)}]$ and such that $\int_{\tk}^{t_{(\!k\!+\!1\!)}} h_k^p(x)v(x)\dx = 1$ and
    \[
      \lt \int_{\tk}^{t_{(\!k\!+\!1\!)}} \vp(x)\dx \rt^\rpp \le 2 \int_{\tk}^{t_{(\!k\!+\!1\!)}} h_k(x)\dx.
    \]
  Furthermore, it holds
    \[
      \sum_{k\in\I_0} \lt \int_{t_{(\!k\!-\!1\!)}}^{\tk} w(t) \Uq(t,\tk)\dt \rt^\rq \lt \int_{\tk}^{t_{(\!k\!+\!1\!)}} \vp(x)\dx \rt^\rpp <\infty,
    \]
  since the sum involves a~finite number of terms and each of them is finite due to \eqref{142} and \eqref{141}. By Proposition \ref{108}, we may find a~nonnegative sequence $\{d_k\}_{k\in\I_0}$ such that $\sum_{k\in\I_0} d_k^p = 1$ and
    \begin{multline*}
      \lt \sum_{k\in\I_0} \lt \int_{t_{(\!k\!-\!1\!)}}^{\tk} w(t) \Uq(t,\tk)\dt \rt^\rq \lt \int_{\tk}^{t_{(\!k\!+\!1\!)}} \vp(x)\dx \rt^\rpp \rt^\jr \\
        \le 2 \lt \sum_{k\in\I_0} d_k^q \int_{t_{(\!k\!-\!1\!)}}^{\tk} w(t) \Uq(t,\tk)\dt \lt \int_{\tk}^{t_{(\!k\!+\!1\!)}} \vp(x)\dx \rt^\frac{q}{p'} \rt^\jq.
    \end{multline*} 
  Define the function $h:=\sum_{k\in\I_0} d_k h_k$. Then it is easy to verify that $\|h\|_{L^p(v)}=1$. Moreover, we get the following estimate.
    \begin{align*}
                 \sum_{k\in\I_0} \! \lt \int_{t_{\!(\!k\!-\!1\!)}}^{\tk} \hspace{-8pt} w(t) \Uq(t,\tk\!)\dt \rt^{\!\rq} \!\! \lt \int_{\tk}^{t_{(\!k\!+\!1\!)}}\hspace{-8pt} \vp\!\!(x)\dx \rt^{\!\rpp}  \hspace{-6pt} &\lesssim  \lt \sum_{k\in\I_0} \! d_k^q \!\! \int_{t_{\!(\!k\!-\!1\!)}}^{\tk} \hspace{-8pt} w(t) \Uq(t,\tk\!)\dt \lt \int_{\tk}^{t_{\!(\!k\!+\!1\!)}} \hspace{-7pt} \vp\!\!(x)\dx \rt^{\!\frac{q}{p'}} \rt^{\!\rq} \\
      & \lesssim  \lt \sum_{k\in\I_0} \! d_k^q \!\! \int_{t_{(\!k\!-\!1\!)}}^{\tk} \hspace{-6pt} w(t) \Uq(t,\tk)\dt \lt \int_{\tk}^{t_{\!(\!k\!+\!1\!)}}\!\! h_k(x)\dx \rt^{\!q} \rt^{\!\rq} \\
      & =         \lt \sum_{k\in\I_0} \int_{t_{(\!k\!-\!1\!)}}^{\tk} w(t) \Uq(t,\tk)\dt \lt \int_{\tk}^{t_{(\!k\!+\!1\!)}} h(x)\dx \rt^q \rt^\rq \\
      & \le       \lt \sum_{k\in\I_0} \int_{t_{(\!k\!-\!1\!)}}^{\tk} w(t) \lt \int_{\tk}^{t_{(\!k\!+\!1\!)}} h(x) U(t,x) \dx \rt^q \dt \rt^\rq \\
      & \le       \lt \sum_{k\in\I_0} \int_0^\infty w(t) \lt \int_{t}^{\infty} h(x) U(t,x) \dx \rt^q \dt \rt^\rq \\
      & \le       C^r \|h\|_{L^p(v)}\\
      & =         C^r.
    \end{align*}
  The covering sequence $\{t_k\}_{k\in\I}$ was arbitrarily chosen in the beginning, hence we may take the supremum over all covering sequences, obtaining the relation
    \[
      D_2 \lesssim C^r.
    \]
  The proof of the implication ``(i)$\Rightarrow$(ii)'' and of the related estimates is then finished. \\
            
  ``(iii)$\Rightarrow$(ii)''. Assume that $A_1<\infty$ and $A_2<\infty$. We will prove the inequality $D_1+D_2\lesssim A_1+A_2$. Let $\{t_k\}_{k\in\I}$ be an~arbitrary covering sequence indexed by a~set $\I$. Then it holds
    \begin{align*}
      &         \sum_{k\in\I_0} \lt \int_{t_{(\!k\!-\!1\!)}}^{t_k} w(x)\dx \rt^\rq \lt \int_{\tk}^{t_{(\!k\!+\!1\!)}} \Upp(\tk,t) \vp(t)\dt \rt^\rpp \\
      &\approx  \sum_{k\in\I_0} \int_{t_{(\!k\!-\!1\!)}}^{t_k} \lt \int_{t_{(\!k\!-\!1\!)}}^x w(s)\ds \rt^\rp w(x)\dx \lt \int_{\tk}^{t_{(\!k\!+\!1\!)}} \Upp(\tk,t) \vp(t)\dt \rt^\rpp \\
      &\le      \sum_{k\in\I_0} \int_{t_{(\!k\!-\!1\!)}}^{t_k} \lt \int_{0}^x w(s)\ds \rt^\rp w(x)\dx \lt \int_{x}^{\infty} \Upp(x,t) \vp(t)\dt \rt^\rpp \\
      &=        A_1.
    \end{align*}
  Taking the supremum over all covering sequences, we obtain $D_1\lesssim A_1$. Similarly, for any fixed covering sequence $\{t_k\}_{k\in\I}$ we get 
    \begin{align*}
      &          \sum_{k\in\I_0} \lt \int_{t_{(\!k\!-\!1\!)}}^{t_k} w(t) \Uq(t,\tk)\dt \rt^\rq \lt \int_{\tk}^{t_{(\!k\!+\!1\!)}} \vp(s)\ds \rt^\rpp \\
      &\approx   \sum_{k\in\I_0} \int_{t_{(\!k\!-\!1\!)}}^{t_k} \lt \int_{t_{(\!k\!-\!1\!)}}^t w(x) \Uq(x,\tk)\dx \rt^\rp  w(t) \Uq(t,\tk)\dt \, \lt \int_{\tk}^{t_{(\!k\!+\!1\!)}} \vp(s)\ds \rt^\rpp \\
      &\lesssim  \sum_{k\in\I_0} \int_{t_{(\!k\!-\!1\!)}}^{t_k} \lt \int_{t_{(\!k\!-\!1\!)}}^t w(x) \Uq(x,t)\dx \rt^\rp  w(t) \Uq(t,\tk)\dt \, \lt \int_{\tk}^{t_{(\!k\!+\!1\!)}} \vp(s)\ds \rt^\rpp \\
      &\qquad +  \sum_{k\in\I_0} \int_{t_{(\!k\!-\!1\!)}}^{t_k} \lt \int_{t_{(\!k\!-\!1\!)}}^t w(x) \dx \rt^\rp w(t) U^r(t,\tk) \dt \, \lt \int_{\tk}^{t_{(\!k\!+\!1\!)}} \vp(s)\ds \rt^\rpp \\
      &\le       \sum_{k\in\I_0} \int_{t_{(\!k\!-\!1\!)}}^{t_k} \lt \int_{0}^t w(x) \Uq(x,t)\dx \rt^\rp  w(t) \Uq(t,\tk)\dt \, \lt \int_{\tk}^\infty \vp(s)\ds \rt^\rpp \\
      &\qquad +  \sum_{k\in\I_0} \int_{t_{(\!k\!-\!1\!)}}^{t_k} \lt \int_{0}^t w(x) \dx \rt^\rp w(t) U^r(t,\tk) \dt \, \lt \int_{\tk}^\infty \vp(s)\ds \rt^\rpp \\
      &\le       \sum_{k\in\I_0} \int_{t_{(\!k\!-\!1\!)}}^{t_k} \lt \int_{0}^t w(x) \Uq(x,t)\dx \rt^\rp  w(t) \sup_{z\in[t,\infty)} \Uq(t,z) \lt \int_{z}^\infty \vp(s)\ds \rt^\rpp \dt \\
      &\qquad +  \sum_{k\in\I_0} \int_{t_{(\!k\!-\!1\!)}}^{t_k} \lt \int_{0}^t w(x) \dx \rt^\rp w(t) \lt \int_t^\infty \Upp(t,s) \vp(s)\ds \rt^\rpp \dt\\
      &=         A_2 + A_1.
    \end{align*}
  Once again, taking the supremum over all covering sequences, we get $D_2\lesssim A_2+A_1$. Hence, we have shown that $D_1+D_2\lesssim A_1+A_2$ and the implication ``(iii)$\Rightarrow$(ii)'' is proved. \\

  \pagebreak

  ``(ii) $\Rightarrow$ (iii)''. Suppose that $D_1<\infty$ and $D_2<\infty$ and let us show that $A_1+A_2\lesssim D_1+D_2$ then. 
  
  Similarly as in the proof of ``(ii) $\Rightarrow$ (i)'', let us first assume that $\int_0^\infty w = 2^K$ for some $K\in\Z$ . Let $\mu\in\Z$ be such that $\mu\le K-2$ and define $\Zm$ by \eqref{23}. Let $\{t_k\}_{k=-\infty}^{K}\subset (0,\infty]$ be the~sequence of points from Theorem \ref{9} and $\{k_n\}_{n=0}^N\subset\Zm$ be the subsequence of indices granted by the \nopagebreak[1] same theorem. Then 
    \Bdef{1}\Bdef{2}\Bdef{3} 
    \begin{align}
                & \int_{\tm}^\infty \lt \int_0^t w(x)\dx \rt^\rp w(t) \lt \int_t^\infty \Upp(t,z)\vp(z)\dz \rt^\rpp \dt \nonumber\\
      & =         \sum_{k\in\Zm} \int_{\dk} \lt \int_0^t w(x)\dx \rt^\rp w(t) \lt \int_t^\infty \Upp(t,z)\vp(z)\dz \rt^\rpp \dt \nonumber\\
      & \le       \sum_{k\in\Zm} \int_0^{t_{(\!k\!+\!1\!)}} \lt \int_0^t w(x)\dx \rt^\rp w(t) \dt\lt \int_{t_k}^\infty \Upp(t_k,z)\vp(z)\dz \rt^\rpp \nonumber\\
      & \lesssim  \sum_{k\in\Zm} \Tkrq \lt \int_{t_k}^\infty \Upp(t_k,z)\vp(z)\dz \rt^\rpp \label{24}\\
      & \approx   \sum_{n=0}^N \sum_{k=k_n}^{k_{(\!n\!+\!1\!)}\!-1} \Tkrq \lt \int_{t_k}^{t_{k_{(\!n\!+\!1\!)}}} \Upp(t_k,z)\vp(z)\dz \rt^\rpp \nonumber\\
      & \qquad +  \sum_{n=0}^{N-1} \sum_{k=k_n}^{k_{(\!n\!+\!1\!)}\!-1} \Tkrq \Ur(t_k,t_{k_{(\!n\!+\!1\!)}}) \lt \int_{t_{k_{(\!n\!+\!1\!)}}}^\infty \vp(z)\dz \rt^\rpp \nonumber\\
      & \qquad +  \sum_{n=0}^{N-1} \sum_{k=k_n}^{k_{(\!n\!+\!1\!)}\!-1} \Tkrq \lt \int_{t_{k_{(\!n\!+\!1\!)}}}^\infty \Upp(t_{k_{(\!n\!+\!1\!)}},z)\vp(z)\dz \rt^\rpp \nonumber\\
      & =: \B{1}+\B{2}+\B{3}\nonumber.
    \end{align}
  In step \eqref{24} we used \eqref{18}. We continue by estimating each of the separate terms. 
    \Bdef{4}\Bdef{5}
    \begin{align*}
      \B{1}\! & =        \sum_{n=0}^N \sum_{k=k_n}^{k_{(\!n\!+\!1\!)}\!-1} \Tkrq \lt \int_{t_k}^{t_{k_{(\!n\!+\!1\!)}}} \Upp(t_k,z)\vp(z)\dz \rt^\rpp \\
            & =        \sum_{n=0}^N \! \Theta^{(\! k_{(\!n\!+\!1\!)}\!-\! 1 \!)\rq} \! \lt \int_{\Delta_{(\!k_{(\!n\!+\!1\!)}\!-\!1\!)}} \hspace{-26pt} \Upp \! (t_{(\!k_{(\!n\!+\!1\!)}\!-1\!)},z)\vp\!(z)\dz \rt^{\!\rpp} \hspace{-8pt} + \! \sum_{n\in\A} \!\! \sum_{k=k_n}^{k_{(\!n\!+\!1\!)}\!-2} \hspace{-8pt} \Tkrq \! \lt \int_{\tk}^{t_{k_{(\!n\!+\!1\!)}}} \! \Upp\!(\tk,z)\vp\!(z)\dz \rt^{\!\rpp} \\
            & \lesssim \sum_{n=0}^N \! \Theta^{k_{(\!n\!+\!1\!)}\rq} \! \lt \int_{\Delta_{(\!k_{(\!n\!+\!1\!)}\!-\!1\!)}} \hspace{-26pt} \Upp \! (t_{(\!k_{(\!n\!+\!1\!)}\!-1\!)},z)\vp\!(z)\dz \rt^{\!\rpp} \hspace{-8pt} + \!\sum_{n\in\A} \!\! \sum_{k=k_n}^{k_{(\!n\!+\!1\!)}\!-2} \hspace{-8pt} \Tkrq \! \lt \int_{\tk}^{t_{(\!k_{(\!n\!+\!1\!)}\!-\!1\!)}} \!\! \Upp\!(\tk,z)\vp\!(z)\dz \rt^{\!\rpp} \\
            & \qquad + \sum_{n\in\A} \sum_{k=k_n}^{k_{(\!n\!+\!1\!)}\!-2} \Tkrq \lt \int_{\Delta_{(\!k_{(\!n\!+\!1\!)}\!-1\!)}} \hspace{-24pt} \Upp(\tk,z)\vp(z)\dz \rt^\rpp \\
            & \lesssim \sum_{n=0}^N \! \Theta^{k_{(\!n\!+\!1\!)}\rq} \! \lt \int_{\Delta_{(\!k_{(\!n\!+\!1\!)}\!-\!1\!)}} \hspace{-26pt} \Upp \! (t_{(\!k_{(\!n\!+\!1\!)}\!-\!1\!)},z)\vp\!(z)\dz \rt^{\!\rpp} \hspace{-8pt} + \! \sum_{n\in\A} \!\! \sum_{k=k_n}^{k_{(\!n\!+\!1\!)}\!-2} \hspace{-8pt} \Tkrq \Ur \! (\tk,t_{(\!k_{(\!n\!+\!1\!)}\!-\!1\!)}) \lt \int_{\Delta_{(\!k_{(\!n\!+\!1\!)}\!-\!1\!)}} \hspace{-24pt} \vp\!(z)\dz \rt^{\!\rpp} \\
            & \qquad + \sum_{n\in\A} \sum_{k=k_n}^{k_{(\!n\!+\!1\!)}\!-2} \Tkrq \Ur(\tk,t_{(\!k_{(\!n\!+\!1\!)}\!-1\!)}) \lt \int_{\tk}^{t_{(\!k_{(\!n\!+\!1\!)}\!-1\!)}} \vp(z)\dz \rt^\rpp \\
            & \qquad + \sum_{n\in\A} \sum_{k=k_n}^{k_{(\!n\!+\!1\!)}\!-2} \Tkrq \lt \int_{\Delta_{(\!k_{(\!n\!+\!1\!)}\!-1\!)}} \hspace{-20pt} \Upp(t_{(\!k_{(\!n\!+\!1\!)}\!-1\!)},z)\vp(z)\dz \rt^\rpp \\
            & \lesssim \sum_{n=0}^N \! \Theta^{k_{(\!n\!+\!1\!)}\rq} \! \lt \int_{\Delta_{(\!k_{\!(\!n\!+\!1\!)}\!-\!1\!)}} \hspace{-27pt} \Upp \!\! (t_{(\!k_{\!(\!n\!+\!1\!)}\!-\!1\!)},z)\vp\!\!(z)\dz \rt^{\!\rpp} \hspace{-8pt} + \! \sum_{n\in\A} \!\! \sum_{k=k_n}^{k_{(\!n\!+\!1\!)}\!-2} \hspace{-8pt} \Tkrq \Ur\!(\tk,t_{(\!k_{\!(\!n\!+\!1\!)}\!-\!1\!)}\!) \lt \int_{\tk}^{t_{k_{\!(\!n\!+\!1\!)}}} \!\! \vp\!\!(z)\dz \rt^{\!\rpp} \\
            & \qquad + \sum_{n\in\A} \sum_{k=k_n}^{k_{(\!n\!+\!1\!)}\!-2} \Tkrq \lt \int_{\Delta_{(\!k_{(\!n\!+\!1\!)}\!-1\!)}} \hspace{-20pt} \Upp(t_{(\!k_{(\!n\!+\!1\!)}\!-1\!)},z)\vp(z)\dz \rt^\rpp \\
            & \lesssim \sum_{n=0}^N \! \Theta^{k_{(\!n\!+\!1\!)}\rq} \! \lt \int_{\Delta_{(\!k_{\!(\!n\!+\!1\!)}\!-\!1\!)}} \hspace{-27pt} \Upp \!\! (t_{(\!k_{\!(\!n\!+\!1\!)}\!-\!1\!)},z)\vp\!\!(z)\dz \rt^{\!\rpp} \hspace{-8pt} + \! \sum_{n\in\A} \!\! \sum_{k=k_n}^{k_{(\!n\!+\!1\!)}\!-2} \hspace{-8pt} \Tkrq \Ur\!(\tk,t_{(\!k_{\!(\!n\!+\!1\!)}\!-\!1\!)}\!) \lt \int_{\tk}^{t_{k_{\!(\!n\!+\!1\!)}}} \!\! \vp\!\!(z)\dz \rt^{\!\rpp} \\
            & := \B{4} + \B{5}.
    \end{align*}
  For $\B{4}$ we have
    \begin{align}
      \B{4} & =        \sum_{n=0}^N \Theta^{k_{(\!n\!+\!1\!)}\rq} \lt \int_{\Delta_{(\!k_{(\!n\!+\!1\!)}\!-1\!)}} \Upp(t_{(\!k_{(\!n\!+\!1\!)}\!-1\!)},z)\vp(z)\dz \rt^\rpp \nonumber\\  
            & \lesssim \sum_{n=0}^N \lt \int_{\Delta_{(\!k_{(\!n\!+\!1\!)}\!-2\!)}} w(x) \dx \rt^\rq \lt \int_{\Delta_{(\!k_{(\!n\!+\!1\!)}\!-1\!)}} \Upp(t_{(\!k_{(\!n\!+\!1\!)}\!-1\!)},z)\vp(z)\dz \rt^\rpp \label{25}\\  
            & \le      \sum_{k\in\Zm} \lt \int_{\Delta_{(\!k-\!1\!)}} w(x) \dx \rt^\rq \lt \int_{\dk} \Upp(t_{(\!k-1\!)},z)\vp(z)\dz \rt^\rpp \nonumber\\  
            & \le      D_1. \nonumber
    \end{align}              
  In step \eqref{25} we used \eqref{18}. Let us formally define $k_{(-1)}:=\mu-1$ and proceed with estimating $\B{5}$.
    \begin{align}
      \B{5} & =        \sum_{n\in\A} \sum_{k=k_n}^{k_{(\!n\!+\!1\!)}\!-2} \Tkrq \Ur(\tk,t_{(\!k_{(\!n\!+\!1\!)}\!-1\!)}) \lt \int_{t_{k_n}}^{t_{k_{(\!n\!+\!1\!)}}} \vp(z)\dz \rt^\rpp \nonumber\\
            & \le      \sum_{n\in\A} \lt \sum_{k=k_n}^{k_{(\!n\!+\!1\!)}\!-2} \Tk \Uq(\tk,t_{(\!k_{(\!n\!+\!1\!)}\!-1\!)}) \rt^\rq \lt \int_{t_{k_n}}^{t_{k_{(\!n\!+\!1\!)}}} \vp(z)\dz \rt^\rpp \label{26}\\
            & \lesssim \sum_{n\in\A} \lt \sum_{k=k_n}^{k_{(\!n\!+\!1\!)}\!-2} \Tk \Uq(\dk) \rt^\rq \lt \int_{t_{k_n}}^{t_{k_{(\!n\!+\!1\!)}}} \vp(z)\dz \rt^\rpp \label{27}\\
            & \lesssim \sum_{n\in\A} \lt \sum_{k=k_{(\!n\!-\!1\!)}}^{k_n-1} \Tk \Uq(\dk) \rt^\rq \lt \int_{t_{k_n}}^{t_{k_{(\!n\!+\!1\!)}}} \vp(z)\dz \rt^\rpp \label{28}\\
            & \lesssim \sum_{n\in\A} \lt \int_{t_{k_{(\!n\!-\!2\!)}}}^{t_{k_n}} w(x) \Uq(x,t_{k_n}) \dx \rt^\rq \lt \int_{t_{k_n}}^{t_{k_{(\!n\!+\!1\!)}}} \vp(z)\dz \rt^\rpp \label{29}\\
            & \le      \sum_{n=1}^{N} \lt \int_{t_{k_{(\!n\!-\!2\!)}}}^{t_{k_n}} w(x) \Uq(x,t_{k_n}) \dx \rt^\rq \lt \int_{t_{k_n}}^{t_{k_{(\!n\!+\!1\!)}}} \vp(z)\dz \rt^\rpp \nonumber\\
            & =        \sum_{i=0}^1 \sum_{\substack{{1\le n\le N}\\ {n\mod 2=i}}} \lt \int_{t_{k_{(\!n\!-\!2\!)}}}^{t_{k_n}} w(x) \Uq(x,t_{k_n}) \dx \rt^\rq \lt \int_{t_{k_n}}^{t_{k_{(\!n\!+\!1\!)}}} \vp(z)\dz \rt^\rpp \label{31}\\                         
            & \lesssim D_2.\nonumber
    \end{align}
  Since $\rq>1$, the estimate \eqref{26} follows by convexity of the $\rq$-th power. Step \eqref{27} is due to Proposition \ref{4}. Step \eqref{28} then follows by \eqref{14}, and step \eqref{29} by \eqref{16}. Finally, in \eqref{31} we split the even and odd indices $n$, so that the intervals $(t_{k_{(\!n\!-\!2\!)}},t_{k_n})$ involved in each $n$-indexed sum do not overlap. This standard step will be also used in other estimates further on. 
  
  So far we have proved
    \[
      \B{1} \lesssim \B{4} + \B{5} \lesssim D_1 + D_2.
    \] 
  The term $\B{2}$ is estimated as follows.
    \begin{align}
      \B{2} & =        \sum_{n=0}^{N-1} \sum_{k=k_n}^{k_{(\!n\!+\!1\!)}\!-1} \Tkrq \Ur(t_k,t_{k_{(\!n\!+\!1\!)}}) \lt \int_{t_{k_{(\!n\!+\!1\!)}}}^\infty \vp(z)\dz \rt^\rpp \nonumber\\
            & \le      \sum_{n=0}^{N-1} \lt \sum_{k=k_n}^{k_{(\!n\!+\!1\!)}\!-1} \Tk \Uq(t_k,t_{k_{(\!n\!+\!1\!)}}) \rt^\rq \lt \int_{t_{k_{(\!n\!+\!1\!)}}}^\infty \vp(z)\dz \rt^\rpp \label{32}\\
            & \lesssim \sum_{n=0}^{N-1} \lt \sum_{k=k_n}^{k_{(\!n\!+\!1\!)}\!-1} \Tk \Uq(\dk) \rt^\rq \lt \int_{t_{k_{(\!n\!+\!1\!)}}}^\infty \vp(z)\dz \rt^\rpp \label{33}\\
            & =        \sum_{n=0}^{N-1} \lt \sum_{k=k_n}^{k_{(\!n\!+\!1\!)}\!-1} \Tk \Uq(\dk) \rt^\rq \lt \sum_{j=n+1}^N \int_{t_{k_j}}^{t_{k_{(\!j\!+\!1\!)}}} \vp(z)\dz \rt^\rpp \nonumber\\
            & \lesssim \sum_{n=0}^{N-1} \lt \sum_{k=k_n}^{k_{(\!n\!+\!1\!)}\!-1} \Tk \Uq(\dk) \rt^\rq \lt \int_{t_{k_{(\!n\!+\!1\!)}}}^{t_{k_{(\!n\!+\!2\!)}}} \vp(z)\dz \rt^\rpp \label{34}\\
            & \lesssim \sum_{n=0}^{N-1} \lt \int_{t_{k_{(\!n\!-\!1\!)}}}^{t_{k_{(\!n\!+\!1\!)}}} w(x) \Uq(x,t_{k_{(\!n\!+\!1\!)}}) \dx \rt^\rq \lt \int_{t_{k_{(\!n\!+\!1\!)}}}^{t_{k_{(\!n\!+\!2\!)}}} \vp(z)\dz \rt^\rpp \label{35}\\
            & =        \sum_{i=0}^1 \sum_{\substack{{1\le n\le N}\\ {n\mod 2=i}}} \lt \int_{t_{k_{(\!n\!-\!1\!)}}}^{t_{k_{(\!n\!+\!1\!)}}} w(x) \Uq(x,t_{k_{(\!n\!+\!1\!)}}) \dx \rt^\rq \lt \int_{t_{k_{(\!n\!+\!1\!)}}}^{t_{k_{(\!n\!+\!2\!)}}} \vp(z)\dz \rt^\rpp \nonumber\\
            & \lesssim D_2. \nonumber        
    \end{align}
  We used convexity of the $\rq$-th power to get \eqref{32}. Step \eqref{33} follows by Proposition \ref{4}. Inequality \eqref{34} is granted by Proposition \ref{3} equipped with \eqref{12}. Step \eqref{35} follows by \eqref{16}. We proved
  \[
    \B{2} \lesssim D_2.
  \]
  The term $\B{3}$ is first handled in the following way.
    \Bdef{6}\Bdef{7}
    \begin{align}
      \B{3} & =        \sum_{n=0}^{N-1} \sum_{k=k_n}^{k_{(\!n\!+\!1\!)}\!-1} \Tkrq \lt \int_{t_{k_{(\!n\!+\!1\!)}}}^\infty \Upp(t_{k_{(\!n\!+\!1\!)}},z)\vp(z)\dz \rt^\rpp \nonumber\\
            & \lesssim \sum_{n=0}^{N-1} \Theta^{k_{(\!n\!+\!1\!)}\rq} \lt \int_{t_{k_{(\!n\!+\!1\!)}}}^\infty \Upp(t_{k_{(\!n\!+\!1\!)}},z)\vp(z)\dz \rt^\rpp \nonumber\\
            & =        \sum_{n=0}^{N-1} \Theta^{k_{(\!n\!+\!1\!)}\rq} \lt \sum_{j=n+1}^N \int_{t_{k_j}}^{t_{k_{(\!j\!+\!1\!)}}} \Upp(t_{k_{(\!n\!+\!1\!)}},z)\vp(z)\dz \rt^\rpp \nonumber\\
            & \lesssim \sum_{n=0}^{N-2} \Theta^{k_{(\!n\!+\!1\!)}\rq} \lt \sum_{j=n+2}^N \Upp(t_{k_{(\!n\!+\!1\!)}},t_{k_j}) \int_{t_{k_j}}^{t_{k_{(\!j\!+\!1\!)}}} \vp(z)\dz \rt^\rpp \nonumber\\
            & \quad  + \sum_{n=0}^{N-1} \Theta^{k_{(\!n\!+\!1\!)}\rq} \lt \sum_{j=n+1}^N \int_{t_{k_j}}^{t_{k_{(\!j\!+\!1\!)}}} \Upp(t_{k_j},z)\vp(z)\dz \rt^\rpp \nonumber\\
            & =: \B{6}+\B{7}.\nonumber
    \end{align}
  Then, for $\B{6}$ we have
    \begin{align}
      \B{6} & =        \sum_{n=0}^{N-2} \Theta^{k_{(\!n\!+\!1\!)}\rq} \lt \sum_{j=n+2}^N \Upp(t_{k_{(\!n\!+\!1\!)}},t_{k_j}) \int_{t_{k_j}}^{t_{k_{(\!j\!+\!1\!)}}} \vp(z)\dz \rt^\rpp \nonumber\\
            & \le      \sum_{n=0}^{N-2} \Theta^{k_{(\!n\!+\!1\!)}\rq} \sum_{j=n+2}^N \Ur(t_{k_{(\!n\!+\!1\!)}},t_{k_j}) \lt \int_{t_{k_j}}^{t_{k_{(\!j\!+\!1\!)}}} \vp(z)\dz \rt^\rpp \label{36}\\
            & =        \sum_{j=2}^N \sum_{n=0}^{j-2} \Theta^{k_{(\!n\!+\!1\!)}\rq} \Ur(t_{k_{(\!n\!+\!1\!)}},t_{k_j}) \lt \int_{t_{k_j}}^{t_{k_{(\!j\!+\!1\!)}}} \vp(z)\dz \rt^\rpp \nonumber\\
            & \le      \sum_{j=2}^N \lt \sum_{n=0}^{j-2} \Theta^{k_{(\!n\!+\!1\!)}} \Uq(t_{k_{(\!n\!+\!1\!)}},t_{k_j}) \rt^\rq \lt \int_{t_{k_j}}^{t_{k_{(\!j\!+\!1\!)}}} \vp(z)\dz \rt^\rpp \label{38}\\
            & \le      \sum_{j=2}^N \lt \sum_{k=\mu}^{k_{(\!j\!-\!1\!)}} \Tk \Uq(t_k,t_{k_j}) \rt^\rq \lt \int_{t_{k_j}}^{t_{k_{(\!j\!+\!1\!)}}} \vp(z)\dz \rt^\rpp \nonumber\\
            & \lesssim \sum_{j=2}^N \lt \sum_{k=\mu}^{k_{(\!j\!-\!1\!)}} \Tk \Uq(\dk) \rt^\rq \lt \int_{t_{k_j}}^{t_{k_{(\!j\!+\!1\!)}}} \vp(z)\dz \rt^\rpp \label{39}\\
            & \le      \sum_{j=2}^N \lt \sum_{k=\mu}^{k_j-1} \Tk \Uq(\dk) \rt^\rq \lt \int_{t_{k_j}}^{t_{k_{(\!j\!+\!1\!)}}} \vp(z)\dz \rt^\rpp \nonumber\\
            & \lesssim \sum_{j=2}^N \lt \sum_{k=k_{(\!j\!-\!1\!)}}^{k_j-1} \Tk \Uq(\dk) \rt^\rq \lt \int_{t_{k_j}}^{t_{k_{(\!j\!+\!1\!)}}} \vp(z)\dz \rt^\rpp \label{58}\\
            & \lesssim \sum_{j=2}^N \lt \int_{t_{k_{(\!j\!-\!2\!)}}}^{t_{k_j}} w(x) \Uq(x,t_{k_j})\dx \rt^\rq \lt \int_{t_{k_j}}^{t_{k_{(\!j\!+\!1\!)}}} \vp(z)\dz \rt^\rpp \label{40}\\
            & =        \sum_{i=0}^1 \sum_{\substack{{2\le j\le N}\\ {j\mod 2=i}}} \lt \int_{t_{k_{(\!j\!-\!2\!)}}}^{t_{k_j}} w(x) \Uq(x,t_{k_j})\dx \rt^\rq \lt \int_{t_{k_j}}^{t_{k_{(\!j\!+\!1\!)}}} \vp(z)\dz \rt^\rpp \nonumber\\
            & \lesssim D_2. \nonumber
    \end{align}        
  Inequality \eqref{36} follows from concavity of the $\rpp$-th power since $\rpp<1$. Similarly, convexity of the $\rq$-th power yields \eqref{38}. Step \eqref{39} is due to Proposition \ref{4}, step \eqref{58} follows by \eqref{13}, and in step \eqref{40} we used \eqref{16}. We continue as follows.
    \begin{align}
      \B{7} & =        \sum_{n=0}^{N-1} \Theta^{k_{(\!n\!+\!1\!)}\rq} \lt \sum_{j=n+1}^N \int_{t_{k_j}}^{t_{k_{(\!j\!+\!1\!)}}} \Upp(t_{k_j},z)\vp(z)\dz \rt^\rpp \nonumber\\
            & \lesssim \sum_{n=0}^{N-1} \Theta^{k_{(\!n\!+\!1\!)}\rq} \lt \int_{t_{k_{(\!n\!+\!1\!)}}}^{t_{k_{(\!n\!+\!2\!)}}} \Upp(t_{k_j},z)\vp(z)\dz \rt^\rpp \label{41}\\
            & \lesssim \sum_{n=0}^{N-1} \lt \int_{t_{k_n}}^{t_{k_{(\!n\!+\!1\!)}}} w(x)\dx \rt^\rq \lt \int_{t_{k_{(\!n\!+\!1\!)}}}^{t_{k_{(\!n\!+\!2\!)}}} \Upp(t_{k_j},z)\vp(z)\dz \rt^\rpp \label{42}\\
            & \le      D_1. \nonumber
    \end{align}            
  To get \eqref{41}, we used Proposition \ref{3}, and in \eqref{42} we applied \eqref{18}. We have proved
    \[
      \B{3} \lesssim \B{6}+ \B{7} \lesssim D_1 + D_2.
    \]
  Combining all the estimates we have obtained so far, we get
    \begin{equation}\label{43}
      \int_{\tm}^\infty \lt \int_0^t w(x)\dx \rt^\rp w(t) \lt \int_t^\infty \Upp(t,z)\vp(z)\dz \rt^\rpp \dt\,  \lesssim\,  D_1 + D_2.
    \end{equation}

  In the following part, we are going to perform estimates related to the term $A_2$. We have
    \Bdef{8}\Bdef{9}
    \begin{align}
      &          \int_{\tm}^\infty \lt \int_{\tm}^t w(x) \Uq(x,t) \dx \rt^\rp w(t) \sup_{z\in[t,\infty)} \Uq(t,z) \lt \int_z^\infty \vp(s)\ds \rt^\rpp \dt \nonumber\\
      & =        \sum_{n=0}^{N} \int_{\Delta_{(\!k_{(\!n\!+\!1\!)}\!-1\!)}} \lt \int_{\tm}^t w(x) \Uq(x,t) \dx \rt^\rp w(t) \sup_{z\in[t,\infty)} \Uq(t,z) \lt \int_z^\infty \vp(s)\ds \rt^\rpp \dt \nonumber\\
      & \quad +  \sum_{n\in\A} \int_{t_{k_n}}^{t_{(\!k_{(\!n\!+\!1\!)}\!-1\!)}} \lt \int_{\tm}^t w(x) \Uq(x,t) \dx \rt^\rp w(t) \sup_{z\in[t,\infty)} \Uq(t,z) \lt \int_z^\infty \vp(s)\ds \rt^\rpp \dt \nonumber\\
      & =:       \B{8}+\B{9}. \nonumber
    \end{align}                   
  By \eqref{15}, the term $\B{8}$ is further estimated as follows.
    \Bdef{10}\Bdef{11}
    \begin{align*}
      \B{8} & =         \sum_{n=0}^{N} \int_{\Delta_{(\!k_{(\!n\!+\!1\!)}\!-1\!)}} \lt \int_{\tm}^t w(x) \Uq(x,t) \dx \rt^\rp w(t) \sup_{z\in[t,\infty)} \Uq(t,z) \lt \int_z^\infty \vp(s)\ds \rt^\rpp \dt \\
            & \lesssim  \sum_{n=1}^{N} \lt \sum_{j=k_{(n\!-\!1\!)}}^{k_n\!-1} \Tj \Uj \rt^\rp \int_{\Delta_{(\!k_{(\!n\!+\!1\!)}\!-1\!)}} \hspace{-10pt} w(t) \sup_{z\in[t,\infty)} \Uq(t,z) \lt \int_z^\infty \vp(s)\ds \rt^\rpp \dt \\
            & \quad +   \sum_{n=0}^{N} \int_{\Delta_{(\!k_{(\!n\!+\!1\!)}\!-1\!)}} \hspace{-10pt} \Theta^{\rp (\!k_{(\!n\!+\!1\!)}\!-1\!)} U^{\frac{rq}p} (t_{(\!k_{(\!n\!+\!1\!)}\!-1\!)},t) \, w(t) \sup_{z\in[t,\infty)} \Uq(t,z) \lt \int_z^\infty \vp(s)\ds \rt^\rpp \dt \\
            & =:        \B{10}+\B{11}.   
    \end{align*}
  Notice that that, in $\B{10}$, the term corresponding to $n=0$ is indeed omitted, since for any $t\in\Delta_{\mu}$ it holds $\int_{\tm}^t w(x) \Uq(x,t) \dx \lesssim \Theta^\mu \Uq(\tm,t)$ and the right-hand side is thus already represented by the $0$-th term in $\B{11}$. 
  
  Let us note that in what follows, expressions such as $\sup_{x\in(y,\infty]} \fii(x)$ appear even where the argument $\fii(x)$ is undefined for $x=\infty$. To fix this formal detail, suppose that, in such cases, $\sup_{x\in(y,\infty]} \fii(x)$ is simply redefined as $\sup_{x\in(y,\infty)} \fii(x)$. This will make expressions such as $\sum_{n=1}^{N} \sup_{x\in[t_{k_n},t_{k_{(\!n\!+\!1\!)}}]} \fii(x)$ formally correct without need of treating the $(N\!+\!1)$-st summand separately. Besides this, the standard notation $\overline{\Delta}_k$ is used to denote the closure of $\dk$, i.e.~the interval $[\tk,t_{(\!k\!+\!1\!)}]$.
  
  We then estimate $\B{10}$.
    \Bdef{12}\Bdef{13}
    \begin{align}
      \B{10} & =        \sum_{n=1}^{N} \lt \sum_{j=k_{(n\!-\!1\!)}}^{k_n\!-1} \Tj \Uj \rt^\rp \int_{\Delta_{(\!k_{(\!n\!+\!1\!)}\!-1\!)}} \hspace{-10pt} w(t) \sup_{z\in[t,\infty)} \Uq(t,z) \lt \int_z^\infty \vp(s)\ds \rt^\rpp \dt \nonumber \\
             & \lesssim \sum_{n=1}^{N} \lt \sum_{j=k_{(n\!-\!1\!)}}^{k_n\!-1} \Tj \Uj \rt^\rp \Theta^{k_{(\!n\!+\!1\!)}\!-1} \hspace{-10pt} \sup_{z\in[t_{(\!k_{(\!n\!+\!1\!)}\!-1\!)},\infty)} \hspace{-7pt} \Uq(t_{(\!k_{(\!n\!+\!1\!)}\!-1\!)},z) \lt \int_z^\infty \vp(s)\ds \rt^\rpp  \label{200} \\
             & \lesssim \sum_{n=1}^{N} \lt \sum_{j=k_{(n\!-\!1\!)}}^{k_n\!-1} \Tj \Uj \rt^\rp \Theta^{k_{(\!n\!+\!1\!)}\!-1} \hspace{-10pt} \sup_{z\in\overline{\Delta}_{(\!k_{(\!n\!+\!1\!)}\!-1\!)}} \hspace{-7pt} \Uq(t_{(\!k_{(\!n\!+\!1\!)}\!-1\!)},z) \lt \int_z^\infty \vp(s)\ds \rt^\rpp  \label{201} \\
             & \quad +  \sum_{n=1}^{N\!-\!1} \lt \sum_{j=k_{(n\!-\!1\!)}}^{k_n\!-1} \Tj \Uj \rt^\rp \Theta^{k_{(\!n\!+\!1\!)}\!-1} \hspace{-10pt} \sup_{z\in[t_{k_{(\!n\!+\!1\!)}},\infty)} \hspace{-7pt} \Uq(t_{k_{(\!n\!+\!1\!)}},z) \lt \int_z^\infty \vp(s)\ds \rt^\rpp \nonumber \\
             & \lesssim \sum_{n=1}^{N} \lt \sum_{j=k_{(n\!-\!1\!)}}^{k_n\!-1} \Tj \Uj \rt^\rp \Theta^{k_{(\!n\!+\!1\!)}\!-1} \hspace{-10pt} \sup_{z\in\overline{\Delta}_{(\!k_{(\!n\!+\!1\!)}\!-1\!)}} \hspace{-7pt} \Uq(t_{(\!k_{(\!n\!+\!1\!)}\!-1\!)},z) \lt \int_z^\infty \vp(s)\ds \rt^\rpp  \label{004} \\
             & \quad +  \sum_{n=1}^{N\!-\!1} \lt \sum_{j=k_n}^{k_{(n\!+\!1\!)}-\!1} \Tj \Uj \rt^\rp \Theta^{k_{(\!n\!+\!1\!)}\!-1} \hspace{-10pt} \sup_{z\in[t_{k_{(\!n\!+\!1\!)}},\infty)} \hspace{-7pt} \Uq(t_{k_{(\!n\!+\!1\!)}},z) \lt \int_z^\infty \vp(s)\ds \rt^\rpp \nonumber \\
             & =:       \B{12}+\B{13}. \nonumber
    \end{align}
  Inequality \eqref{200} holds by \eqref{18}, and \eqref{201} is due to Proposition \ref{59}. In \eqref{004} we used \eqref{12}. Next, we have
    \Bdef{14}\Bdef{15}
    \begin{align}
      \B{12} & =        \sum_{n=1}^{N} \lt \sum_{j=k_{(n\!-\!1\!)}}^{k_n\!-1} \Tj \Uj \rt^\rp \Theta^{k_{(\!n\!+\!1\!)}\!-1} \hspace{-10pt} \sup_{z\in\overline{\Delta}_{(\!k_{(\!n\!+\!1\!)}\!-1\!)}} \hspace{-7pt} \Uq(t_{(\!k_{(\!n\!+\!1\!)}\!-1\!)},z) \lt \int_z^\infty \vp(s)\ds \rt^\rpp  \nonumber \\
             & \lesssim \sum_{n=1}^{N} \lt \sum_{j=k_{(n\!-\!1\!)}}^{k_n\!-1} \Tj \Uj \rt^\rp \Theta^{k_{(\!n\!+\!1\!)}\!-1} \hspace{-10pt} \sup_{z\in\overline{\Delta}_{(\!k_{(\!n\!+\!1\!)}\!-1\!)}} \hspace{-7pt} \Uq(t_{(\!k_{(\!n\!+\!1\!)}\!-1\!)},z) \lt \int_z^{t_{k_{(\!n\!+\!1\!)}}} \vp(s)\ds \rt^\rpp  \nonumber \\
             & \quad +  \sum_{n=1}^{N\!-1} \lt \sum_{j=k_{(n\!-\!1\!)}}^{k_n\!-1} \Tj \Uj \rt^\rp \Theta^{k_{(\!n\!+\!1\!)}\!-1} \Uq(\Delta_{(\!k_{(\!n\!+\!1\!)}\!-1\!)}) \lt \int_{t_{k_{(\!n\!+\!1\!)}}}^\infty \vp(s)\ds \rt^\rpp  \nonumber \\
             & \lesssim \sum_{n=1}^{N} \lt \sum_{j=k_{(n\!-\!1\!)}}^{k_n\!-1} \Tj \Uj \rt^\rp \Theta^{k_{(\!n\!+\!1\!)}\!-1} \hspace{-10pt} \sup_{z\in\overline{\Delta}_{(\!k_{(\!n\!+\!1\!)}\!-1\!)}} \hspace{-7pt} \Uq(t_{(\!k_{(\!n\!+\!1\!)}\!-1\!)},z) \lt \int_z^{t_{k_{(\!n\!+\!1\!)}}} \vp(s)\ds \rt^\rpp  \label{207} \\
             & \quad +  \sum_{n=1}^{N\!-1} \lt \sum_{j=k_n}^{k_{(n\!+\!1\!)}-\!1} \Tj \Uj \rt^\rq \lt \int_{t_{k_{(\!n\!+\!1\!)}}}^\infty \vp(s)\ds \rt^\rpp \nonumber  \\
             & =:       \B{14}+\B{15}. \nonumber 
    \end{align}
  Step \eqref{207} is based on \eqref{12}. For each $n\in\{1,\ldots,N\}$ there exists a~point $z_{(\!n\!+\!1\!)}\in\overline{\Delta}_{(\!k_{(\!n\!+\!1\!)}\!-\!1\!)}$ such that 
    \begin{equation}\label{202}
      \sup_{z\in\overline{\Delta}_{(\!k_{(\!n\!+\!1\!)}\!-\!1\!)}} \hspace{-7pt} \Uq(t_{(\!k_{(\!n\!+\!1\!)}\!-\!1\!)},z) \lt \int_z^{t_{k_{(\!n\!+\!1\!)}}} \!\! \vp\!\!(s)\ds \rt^{\!\rpp} \!\!\! \le 2 \Uq(t_{(\!k_{(\!n\!+\!1\!)}\!-\!1\!)},z_{(\!n\!+\!1\!)}) \lt \int_{z_{(\!n\!+\!1\!)}}^{t_{k_{(\!n\!+\!1\!)}}} \!\! \vp\!\!(s)\ds \rt^{\!\rpp}\!\!.
    \end{equation}
  Define also $z_{(-1)}:=0$ and $z_{(\!N\!+2\!)}:=\infty$. One then gets
    \begin{align}
      \B{14} & =        \sum_{n=1}^{N} \lt \sum_{j=k_{(\!n\!-\!1\!)}}^{k_n\!-1} \Tj \Uj \rt^\rp \Theta^{k_{(\!n\!+\!1\!)}\!-1} \hspace{-10pt} \sup_{z\in\overline{\Delta}_{(\!k_{(\!n\!+\!1\!)}\!-1\!)}} \hspace{-7pt} \Uq(t_{(\!k_{(\!n\!+\!1\!)}\!-1\!)},z) \lt \int_z^{t_{k_{(\!n\!+\!1\!)}}} \! \vp(s)\ds \rt^\rpp  \nonumber \\
             & \lesssim \sum_{n=1}^{N} \lt \sum_{j=k_{(\!n\!-\!1\!)}}^{k_n\!-1} \Tj \Uj \rt^\rp \Theta^{k_{(\!n\!+\!1\!)}\!-1} \Uq(t_{(\!k_{(\!n\!+\!1\!)}\!-1\!)},z_{(\!n\!+\!1\!)}) \lt \int_{z_{(\!n\!+\!1\!)}}^{t_{k_{(\!n\!+\!1\!)}}} \! \vp(s)\ds \rt^\rpp  \label{203} \\
             & \lesssim \sum_{n=1}^{N} \lt \sum_{j=k_{(\!n\!-\!1\!)}}^{k_n\!-1} \Tj \Uj \rt^\rp \!\!\! \int_{\Delta_{(\!k_{(\!n\!+\!1\!)}\!-2\!)}} \hspace{-15pt} w(t) \dt\ \Uq(t_{(\!k_{(\!n\!+\!1\!)}\!-1\!)},z_{(\!n\!+\!1\!)}) \lt \int_{z_{(\!n\!+\!1\!)}}^{t_{k_{(\!n\!+\!1\!)}}} \! \vp\!(s)\ds \rt^\rpp  \label{204} \\
             & \le      \sum_{n=1}^{N} \lt \sum_{j=k_{(\!n\!-\!1\!)}}^{k_n\!-1} \Tj \Uj \rt^\rp \int_{t_{(\!k_{(\!n\!+\!1\!)}\!-2\!)}}^{z_{(\!n\!+\!1\!)}} w(t) \Uq(t,z_{(\!n\!+\!1\!)}) \dt\, \lt \int_{z_{(\!n\!+\!1\!)}}^{t_{k_{(\!n\!+\!1\!)}}} \! \vp(s)\ds \rt^\rpp  \nonumber \\
             & \le      \sum_{n=1}^{N} \lt \int_{t_{k_{(\!n\!-\!2\!)}}}^{t_{k_n}} \hspace{-10pt} w(t) U^q(t,t_{k_n}) \dt \rt^\rp \int_{t_{(\!k_{(\!n\!+\!1\!)}\!-2\!)}}^{z_{(\!n\!+\!1\!)}} \hspace{-10pt} w(t) \Uq(t,z_{(\!n\!+\!1\!)}) \dt \, \lt \int_{z_{(\!n\!+\!1\!)}}^{t_{k_{(\!n\!+\!1\!)}}} \! \vp\!(s)\ds \rt^\rpp  \label{205} \\
             & \le      \sum_{n=1}^{N} \lt \int_{z_{(\!n\!-\!2\!)}}^{z_{(\!n\!+\!1\!)}} \hspace{-2pt} w(t) U^q(t,{z_{(\!n\!+\!1\!)}}) \dt \rt^\rq \lt \int_{z_{(\!n\!+\!1\!)}}^{z_{(\!n\!+\!2\!)}} \! \vp(s)\ds \rt^\rpp  \label{206} \\
             & =        \sum_{i=0}^3  \sum_{\substack{1\le n\le N\\ n \mod 4 = i}} \lt \int_{z_{(\!n\!-\!2\!)}}^{z_{(\!n\!+\!1\!)}} \hspace{-2pt} w(t) U^q(t,{z_{(\!n\!+\!1\!)}}) \dt \rt^\rq \lt \int_{z_{(\!n\!+\!1\!)}}^{z_{(\!n\!+\!2\!)}} \! \vp(s)\ds \rt^\rpp  \nonumber \\
             & \lesssim D_2. \nonumber
    \end{align}
  We used \eqref{202} in \eqref{203}, and \eqref{18} in \eqref{204}. Estimate \eqref{205} follows from \eqref{16}. To get \eqref{206}, we used the relation $z_{(\!n\!-\!1\!)} \le t_{k_{(\!n\!-\!1\!)}} \le t_{(\!k_{(\!n\!+\!1\!)}\!-2\!)}$ which holds for all relevant indices $n$. The second inequality $t_{k_{(\!n\!-\!1\!)}} \le t_{(\!k_{(\!n\!+\!1\!)}\!-2\!)}$ follows from \eqref{10}.

  Concerning $\B{15}$, we obtain
    \begin{align}
      \B{15} & =        \sum_{n=1}^{N\!-1} \lt \sum_{j=k_n}^{k_{(n\!+\!1\!)}-\!1} \Tj \Uj \rt^\rq \lt \int_{t_{k_{(\!n\!+\!1\!)}}}^\infty \vp(s)\ds \rt^\rpp  \nonumber \\
             & =        \sum_{n=1}^{N\!-1} \lt \sum_{j=k_n}^{k_{(n\!+\!1\!)}-\!1} \Tj \Uj \rt^\rq \lt \sum_{i=n\!+\!1}^{N\!-1} \int_{t_{k_i}}^{t_{k_{(\!i\!+\!1\!)}}} \! \vp(s)\ds \rt^\rpp  \nonumber \\
             & \lesssim \sum_{n=1}^{N\!-1} \lt \sum_{j=k_n}^{k_{(n\!+\!1\!)}-\!1} \Tj \Uj \rt^\rq \lt \int_{t_{k_{(\!n\!+\!1\!)}}}^{t_{k_{(\!n\!+\!2\!)}}} \! \vp(s)\ds \rt^\rpp  \label{208} \\
             & \lesssim \sum_{n=1}^{N\!-1} \lt \int_{t_{k_{(\!n\!-\!1\!)}}}^{t_{k_{(\!n\!+\!1\!)}}} \! w(t) U^q(t,t_{k_n}) \dt \rt^\rq \lt \int_{t_{k_{(\!n\!+\!1\!)}}}^{t_{k_{(\!n\!+\!2\!)}}} \! \vp(s)\ds \rt^\rpp  \label{209} \\
             & =        \sum_{i=0}^2  \sum_{\substack{1\le n\le N\!-\!1\\ n \mod 3 = i}} \lt \int_{t_{k_{(\!n\!-\!1\!)}}}^{t_{k_{(\!n\!+\!1\!)}}} \! w(t) U^q(t,t_{k_n}) \dt \rt^\rq \lt \int_{t_{k_{(\!n\!+\!1\!)}}}^{t_{k_{(\!n\!+\!2\!)}}} \! \vp(s)\ds \rt^\rpp \nonumber \\
             & \lesssim D_2. \nonumber
    \end{align}
  Proposition \ref{3} together with \eqref{12} yields \eqref{208}. Estimate \eqref{209} follows from \eqref{16}. We have proved
    \[
      \B{12} \lesssim \B{14}+\B{15} \lesssim D_2.
    \]
  We proceed with the term $\B{13}$.
    \Bdef{16}\Bdef{17}
    \begin{align}
      \B{13} & =        \sum_{n=1}^{N\!-\!1} \lt \sum_{j=k_n}^{k_{(\!n\!+\!1\!)}\!-\!1} \hspace{-6pt} \Tj \Uj \rt^\rp \!\! \Theta^{k_{(\!n\!+\!1\!)}\!-1} \hspace{-6pt} \sup_{z\in[t_{k_{(\!n\!+\!1\!)}},\infty)} \hspace{-5pt} \Uq(t_{k_{(\!n\!+\!1\!)}},z) \lt \int_z^\infty \vp(s)\ds \rt^\rpp \nonumber \\
             & \le      \sum_{n=1}^{N\!-\!1} \lt \sum_{j=k_n}^{k_{(\!n\!+\!1\!)}\!-\!1} \hspace{-6pt} \Tj \Uj \rt^{\!\rp} \!\! \Theta^{k_{(\!n\!+\!1\!)}} \hspace{-6pt} \sup_{i\in\{n\!+\!1,\ldots,N\}} \sup_{z\in[t_{k_i}\!,\, t_{k_{(\!i\!+\!1\!)}}]} \hspace{-5pt} \Uq(t_{k_{(\!n\!+\!1\!)}},z) \lt \int_z^\infty \hspace{-5pt} \vp\!(s)\ds \rt^\rpp \nonumber \\
             & \lesssim \sum_{n=1}^{N\!-\!1} \lt \sum_{j=k_n}^{k_{(\!n\!+\!1\!)}\!-\!1} \hspace{-6pt} \Tj \Uj \rt^{\!\rp} \!\! \Theta^{k_{(\!n\!+\!1\!)}} \hspace{-6pt} \sup_{i\in\{n\!+\!1,\ldots,N\}} \sup_{z\in[t_{k_i}\!,\, t_{k_{(\!i\!+\!1\!)}}]} \hspace{-5pt} \Uq(t_{k_i},z) \lt \int_z^\infty \hspace{-5pt} \vp\!(s)\ds \rt^\rpp \nonumber \\
             & \quad +  \sum_{n=1}^{N\!-\!2} \lt \sum_{j=k_n}^{k_{(\!n\!+\!1\!)}\!-\!1} \hspace{-6pt} \Tj \Uj \rt^{\!\rp} \!\! \Theta^{k_{(\!n\!+\!1\!)}} \hspace{-6pt} \sup_{i\in\{n\!+\!2,\ldots,N\}} \hspace{-5pt} \Uq(t_{k_{(\!n\!+\!1\!)}},t_{k_i}) \lt \int_{t_{k_i}}^\infty \hspace{-5pt} \vp\!(s)\ds \rt^\rpp \nonumber \\                            
             & =:       \B{16}+\B{17}. \nonumber
    \end{align}
  For $\B{16}$ we have
    \Bdef{18}\Bdef{19}
    \begin{align}
      \B{16} & =        \sum_{n=1}^{N\!-\!1} \lt \sum_{j=k_n}^{k_{(\!n\!+\!1\!)}\!-\!1} \hspace{-6pt} \Tj \Uj \rt^{\!\rp}  \Theta^{k_{(\!n\!+\!1\!)}} \hspace{-6pt} \sup_{i\in\{n\!+\!1,\ldots,N\}} \sup_{z\in[t_{k_i}\!,\, t_{k_{(\!i\!+\!1\!)}}]} \hspace{-5pt} \Uq(t_{k_i},z) \lt \int_z^\infty \hspace{-5pt} \vp\!(s)\ds \rt^\rpp \nonumber \\  
             & \lesssim \sum_{n=1}^{N\!-\!1} \lt \sum_{j=k_n}^{k_{(\!n\!+\!1\!)}\!-\!1} \hspace{-6pt} \Tj \Uj \rt^{\!\rp}  \Theta^{k_{(\!n\!+\!1\!)}} \hspace{-6pt} \sup_{z\in[t_{k_{(\!n\!+\!1\!)}},t_{k_{(\!n\!+\!2\!)}} ]} \hspace{-5pt} \Uq(t_{k_{(\!n\!+\!1\!)}},z) \lt \int_z^\infty \hspace{-5pt} \vp\!(s)\ds \rt^\rpp \label{210} \\
             & \lesssim \sum_{n=1}^{N\!-\!1} \lt \sum_{j=k_n}^{k_{(\!n\!+\!1\!)}\!-\!1} \hspace{-6pt} \Tj \Uj \rt^{\!\rp}  \Theta^{k_{(\!n\!+\!1\!)}} \hspace{-6pt} \sup_{z\in[t_{k_{(\!n\!+\!1\!)}},t_{k_{(\!n\!+\!2\!)}} )} \hspace{-5pt} \Uq(t_{k_{(\!n\!+\!1\!)}},z) \lt \int_z^{t_{k_{(\!n\!+\!2\!)}}}  \vp\!(s)\ds \rt^\rpp \nonumber \\
             & \quad +  \sum_{n=1}^{N\!-\!2} \lt \sum_{j=k_n}^{k_{(\!n\!+\!1\!)}\!-\!1} \hspace{-6pt} \Tj \Uj \rt^{\!\rp}  \Theta^{k_{(\!n\!+\!1\!)}} \Uq(t_{k_{(\!n\!+\!1\!)}},t_{k_{(\!n\!+\!2\!)}}) \lt \int_{t_{k_{(\!n\!+\!2\!)}}}^\infty  \vp\!(s)\ds \rt^\rpp \nonumber \\
             & =:       \B{18}+\B{19}. \nonumber
    \end{align}
  In step \eqref{210} we used Proposition \ref{3}, considering also \eqref{12}. For each $n\in\{0,\ldots,N\!-\!1\}$ there exists a~point $y_{(\!n\!+\!1\!)}\in[t_{k_{(\!n\!+\!1\!)}},t_{k_{(\!n\!+\!2\!)}} ]$ such that 
    \begin{equation}\label{211}
      \sup_{z\in[t_{k_{(\!n\!+\!1\!)}},t_{k_{(\!n\!+\!2\!)}} ]} \hspace{-12pt} \Uq(t_{k_{(\!n\!+\!1\!)}},z) \lt \int_z^{t_{k_{(\!n\!+\!2\!)}}} \!\! \vp\!(s)\ds \rt^\rpp \!\!\! \le 2 \Uq(t_{k_{(\!n\!+\!1\!)}},y_{(\!n\!+\!1\!)}) \lt \int_{y_{(\!n\!+\!1\!)}}^{t_{k_{(\!n\!+\!2\!)}}} \!\! \vp\!(s)\ds \rt^\rpp \!\!\!\!.
    \end{equation}
  Define also $y_{(-1)}:=0$ and $y_{(\!N\!+2\!)}:=\infty$.
    \begin{align}
      \B{18} & =        \sum_{n=1}^{N\!-\!1} \lt \sum_{j=k_n}^{k_{(\!n\!+\!1\!)}\!-\!1} \hspace{-6pt} \Tj \Uj \rt^{\!\rp}  \Theta^{k_{(\!n\!+\!1\!)}} \hspace{-6pt} \sup_{z\in[t_{k_{(\!n\!+\!1\!)}},t_{k_{(\!n\!+\!2\!)}} ]} \hspace{-5pt} \Uq(t_{k_{(\!n\!+\!1\!)}},z) \lt \int_z^{t_{k_{(\!n\!+\!2\!)}}}  \vp\!(s)\ds \rt^\rpp \nonumber \\
             & \lesssim \sum_{n=1}^{N\!-\!1} \lt \sum_{j=k_n}^{k_{(\!n\!+\!1\!)}\!-\!1} \hspace{-6pt} \Tj \Uj \rt^{\!\rp}  \Theta^{k_{(\!n\!+\!1\!)}} \Uq(t_{k_{(\!n\!+\!1\!)}},y_{(\!n\!+\!1\!)}) \lt \int_{y_{(\!n\!+\!1\!)}}^{t_{k_{(\!n\!+\!2\!)}}} \! \vp\!(s)\ds \rt^\rpp \label{212} \\
             & \lesssim \sum_{n=1}^{N\!-\!1} \lt \sum_{j=k_n}^{k_{(\!n\!+\!1\!)}\!-\!1} \hspace{-6pt} \Tj \Uj \rt^{\!\rp} \!\! \int_{\Delta_{(\!k_{\!(\!n\!+\!1\!)}\!-\!1\!)}} \hspace{-12pt} w(t)\dt\ \Uq(t_{k_{(\!n\!+\!1\!)}},y_{(\!n\!+\!1\!)}) \lt \int_{y_{(\!n\!+\!1\!)}}^{t_{k_{(\!n\!+\!2\!)}}} \! \vp\!(s)\ds \rt^\rpp \label{213} \\
             & \le      \sum_{n=1}^{N\!-\!1} \lt \sum_{j=k_n}^{k_{(\!n\!+\!1\!)}\!-\!1} \hspace{-6pt} \Tj \Uj \rt^{\!\rp} \!\! \int_{y_{(\!n\!-\!2\!)}}^{y_{(\!n\!+\!1\!)}} \hspace{-2pt} w(t) \Uq(t,y_{(\!n\!+\!1\!)}) \dt\, \lt \int_{y_{(\!n\!+\!1\!)}}^{t_{k_{(\!n\!+\!2\!)}}} \! \vp\!(s)\ds \rt^\rpp \label{214} \\
             & \lesssim \sum_{n=1}^{N\!-\!1} \lt \int_{y_{(\!n\!-\!2\!)}}^{y_{(\!n\!+\!1\!)}} \hspace{-2pt} w(t) \Uq(t,y_{(\!n\!+\!1\!)}) \dt \rt^{\rq} \lt \int_{y_{(\!n\!+\!1\!)}}^{y_{(\!n\!+\!2\!)}} \! \vp\!(s)\ds \rt^\rpp \label{215} \\
             & =        \sum_{i=0}^3  \sum_{\substack{1\le n\le N\!-\!1\\ n \mod 4 = i}} \lt \int_{y_{(\!n\!-\!2\!)}}^{y_{(\!n\!+\!1\!)}} \hspace{-2pt} w(t) \Uq(t,y_{(\!n\!+\!1\!)}) \dt \rt^{\rq} \lt \int_{y_{(\!n\!+\!1\!)}}^{y_{(\!n\!+\!2\!)}} \! \vp\!(s)\ds \rt^\rpp \nonumber \\
             & \lesssim D_2. \nonumber
    \end{align}
  In \eqref{212} we used \eqref{211}. Inequality \eqref{213} follows from \eqref{18}. To get \eqref{214}, we used the inequality $y_{(\!n\!-\!2\!)} \le t_{k_{(\!n\!-\!1\!)}} \le t_{(\!k_{\!(\!n\!+\!1\!)}\!-\!1\!)}$ (cf.~\eqref{10}) satisfied for all relevant indices $n$. This inequality, together with \eqref{16}, also yields \eqref{215}. 

  Next, the term $\B{19}$ is treated as follows.
    \begin{align}
      \B{19} & =        \sum_{n=1}^{N\!-\!2} \lt \sum_{j=k_n}^{k_{(\!n\!+\!1\!)}\!-\!1} \hspace{-2pt} \Tj \Uj \rt^{\!\rp}  \Theta^{k_{(\!n\!+\!1\!)}} \Uq(t_{k_{(\!n\!+\!1\!)}},t_{k_{(\!n\!+\!2\!)}}) \lt \int_{t_{k_{(\!n\!+\!2\!)}}}^\infty  \vp\!(s)\ds \rt^\rpp \nonumber \\
             & \lesssim \sum_{n=1}^{N\!-\!2} \lt \sum_{j=k_{(\!n\!+\!1\!)}}^{k_{(\!n\!+\!2\!)}\!-\!1} \hspace{-2pt} \Tj \Uj \rt^{\!\rp}  \Theta^{k_{(\!n\!+\!1\!)}} \Uq(t_{k_{(\!n\!+\!1\!)}},t_{k_{(\!n\!+\!2\!)}}) \lt \int_{t_{k_{(\!n\!+\!2\!)}}}^\infty  \vp\!(s)\ds \rt^\rpp \label{216} \\
             & \le      \sum_{n=1}^{N\!-\!2} \lt \sum_{j=k_{(\!n\!+\!1\!)}}^{k_{(\!n\!+\!2\!)}\!-\!1} \hspace{-2pt} \Tj \Uj \rt^{\!\rp}  \sum_{j=k_{(\!n\!+\!1\!)}}^{k_{(\!n\!+\!2\!)}\!-\!1} \hspace{-2pt} \Tj \Uq(t_j,t_{k_{(\!n\!+\!2\!)}}) \lt \int_{t_{k_{(\!n\!+\!2\!)}}}^\infty  \vp\!(s)\ds \rt^\rpp \nonumber \\
             & \lesssim \sum_{n=1}^{N\!-\!2} \lt \sum_{j=k_{(\!n\!+\!1\!)}}^{k_{(\!n\!+\!2\!)}\!-\!1} \hspace{-2pt} \Tj \Uj \rt^{\!\rq}  \lt \int_{t_{k_{(\!n\!+\!2\!)}}}^\infty  \vp\!(s)\ds \rt^\rpp \label{217} \\
             & \le      \B{15} \nonumber\\
             & \lesssim D_2. \nonumber
    \end{align}
  Inequality \eqref{216} is obtained by using \eqref{12}, and inequality \eqref{217} by Proposition \ref{4}. The final estimate $\B{15}\lesssim D_2$ was already proved before. We have obtained
    \[
      \B{16} \lesssim \B{18}+\B{19} \lesssim D_2.
    \]
    
  Let us now return to the term $\B{17}$. It holds
    \begin{align}
      \B{17} & =        \sum_{n=1}^{N\!-\!2} \lt \sum_{j=k_n}^{k_{(\!n\!+\!1\!)}\!-\!1} \hspace{-6pt} \Tj \Uj \rt^{\!\rp} \!\! \Theta^{k_{(\!n\!+\!1\!)}} \hspace{-6pt} \sup_{i\in\{n\!+\!2,\ldots,N\}} \hspace{-5pt} \Uq(t_{k_{(\!n\!+\!1\!)}},t_{k_i}) \lt \int_{t_{k_i}}^\infty \hspace{-5pt} \vp\!(s)\ds \rt^\rpp \nonumber \\                            
             & \le      \sum_{n=1}^{N\!-\!2} \lt \sum_{j=k_n}^{k_{(\!n\!+\!1\!)}\!-\!1} \hspace{-6pt} \Tj \Uj \rt^{\!\rp} \hspace{-4pt} \sup_{i\in\{n\!+\!2,\ldots,N\}} \sum_{j=\mu}^{k_i\!-\!1} \Tj \Uq(t_j,t_{k_i}) \lt \int_{t_{k_i}}^\infty \hspace{-5pt} \vp\!(s)\ds \rt^\rpp \nonumber \\                            
             & \lesssim \sum_{n=1}^{N\!-\!2} \lt \sum_{j=k_n}^{k_{(\!n\!+\!1\!)}\!-\!1} \hspace{-6pt} \Tj \Uj \rt^{\!\rp} \hspace{-4pt} \sup_{i\in\{n\!+\!2,\ldots,N\}} \sum_{j=\mu}^{k_i\!-\!1} \Tj \Uj \lt \int_{t_{k_i}}^\infty \hspace{-5pt} \vp\!(s)\ds \rt^\rpp \label{218} \\                            
             & \lesssim \sum_{n=1}^{N\!-\!2} \lt \sum_{j=k_n}^{k_{(\!n\!+\!1\!)}\!-\!1} \hspace{-6pt} \Tj \Uj \rt^{\!\rp} \hspace{-4pt} \sup_{i\in\{n\!+\!2,\ldots,N\}} \sum_{j=k_{\!(\!i\!-\!1\!)}}^{k_i\!-\!1} \hspace{-6pt} \Tj \Uj \lt \int_{t_{k_i}}^\infty \hspace{-5pt} \vp\!(s)\ds \rt^\rpp \label{219} \\                            
             & =        \sum_{n=1}^{N\!-\!2} \lt \sum_{j=k_n}^{k_{(\!n\!+\!1\!)}\!-\!1} \hspace{-6pt} \Tj \Uj \rt^{\!\rp} \hspace{-4pt} \sup_{i\in\{n\!+\!2,\ldots,N\}} \sum_{j=k_{\!(\!i\!-\!1\!)}}^{k_i\!-\!1} \hspace{-6pt} \Tj \Uj \lt \sum_{m=i}^{N} \int_{t_{k_m}}^{t_{k_{(\!m\!+\!1\!)}}} \vp\!(s)\ds \rt^\rpp \nonumber \\                            
             & \lesssim \sum_{n=1}^{N\!-\!2} \lt \sum_{j=k_n}^{k_{(\!n\!+\!1\!)}\!-\!1} \hspace{-6pt} \Tj \Uj \rt^{\!\rp} \hspace{-4pt} \sup_{i\in\{n\!+\!2,\ldots,N\}} \sum_{j=k_{\!(\!i\!-\!1\!)}}^{k_i\!-\!1} \hspace{-6pt} \Tj \Uj \lt \int_{t_{k_i}}^{t_{k_{(\!i\!+\!1\!)}}} \vp\!(s)\ds \rt^\rpp \label{220} \\                            
             & \lesssim \sum_{n=1}^{N\!-\!2} \lt \sum_{j=k_n}^{k_{(\!n\!+\!1\!)}\!-\!1} \hspace{-6pt} \Tj \Uj \rt^{\!\rp} \sum_{j=k_{\!(\!n\!+\!1\!)}}^{k_{\!(\!n\!+\!2\!)}\!-\!1} \hspace{-6pt} \Tj \Uj \lt \int_{t_{k_{(\!n\!+\!2\!)}}}^{t_{k_{(\!n\!+\!3\!)}}} \vp\!(s)\ds \rt^\rpp \label{221} \\                            
             & \lesssim \sum_{n=1}^{N\!-\!2} \lt \sum_{j=k_{\!(\!n\!+\!1\!)}}^{k_{\!(\!n\!+\!2\!)}\!-\!1} \hspace{-6pt} \Tj \Uj \rt^\rq \lt \int_{t_{k_{(\!n\!+\!2\!)}}}^{t_{k_{(\!n\!+\!3\!)}}} \vp\!(s)\ds \rt^\rpp \label{222} \\                            
             & \lesssim \sum_{n=1}^{N\!-\!2} \lt \int_{t_{k_n}}^{t_{k_{(\!n\!+\!2\!)}}} \hspace{-2pt} w(t) \Uq(t,{t_{k_{(\!n\!+\!2\!)}}}) \dt \rt^{\rq} \lt \int_{t_{k_{(\!n\!+\!2\!)}}}^{t_{k_{(\!n\!+\!3\!)}}} \vp\!(s)\ds \rt^\rpp \label{223} \\                            
             & =        \sum_{i=0}^2  \sum_{\substack{1\le n\le N\!-\!2\\ n \mod 3 = i}} \lt \int_{t_{k_n}}^{t_{k_{(\!n\!+\!2\!)}}} \hspace{-2pt} w(t) \Uq(t,{t_{k_{(\!n\!+\!2\!)}}}) \dt \rt^{\rq} \lt \int_{t_{k_{(\!n\!+\!2\!)}}}^{t_{k_{(\!n\!+\!3\!)}}} \vp\!(s)\ds \rt^\rpp \nonumber \\
             & \lesssim D_2. \nonumber
    \end{align}
  Inequality \eqref{218} follows from Proposition \ref{4}, and inequality \eqref{219} from \eqref{13}. To get \eqref{220}, one uses Proposition \ref{89}, considering also \eqref{12}. Proposition \ref{3}, again with \eqref{12}, yields \eqref{221}. Step \eqref{222} follows from \eqref{12}. In \eqref{223} we applied \eqref{16}. Having proved $\B{17} \lesssim D_2$, we may now complete several more estimates, namely
    \[
      \B{13} \lesssim \B{16} + \B{17} \lesssim D_2,
    \]
  which, combined with the earlier results, gives
    \[
      \B{10} \lesssim \B{12} + \B{13} \lesssim D_2.
    \]
  
  The next untreated expression is $\B{11}$. It is estimated in the following way.
    \Bdef{20}\Bdef{21}\Bdef{22}
    \begin{align}
      \B{11} & =         \sum_{n=0}^{N} \int_{\Delta_{(\!k_{(\!n\!+\!1\!)}\!-1\!)}} \hspace{-25pt} \Theta^{\rp (\!k_{(\!n\!+\!1\!)}\!-1\!)} U^{\frac{rq}p} (t_{(\!k_{(\!n\!+\!1\!)}\!-1\!)},t) \, w(t) \sup_{z\in[t,\infty)} \! \Uq(t,z) \lt \int_z^\infty \!\!\! \vp\!(s)\ds \rt^\rpp \dt \nonumber \\
             & \lesssim  \sum_{n=0}^{N} \Theta^{\rp (\!k_{(\!n\!+\!1\!)}\!-1\!)}  \!\! \int_{\Delta_{(\!k_{(\!n\!+\!1\!)}\!-1\!)}} \hspace{-25pt} U^{\frac{rq}p} (t_{(\!k_{(\!n\!+\!1\!)}\!-1\!)},t) \, w(t) \hspace{-6pt} \sup_{z\in[t,t_{k_{(\!n\!+\!1\!)}}]} \hspace{-8pt} \Uq(t,z) \lt \int_z^\infty \!\!\! \vp\!(s)\ds \rt^\rpp \dt \label{001} \\
             & \quad +   \sum_{n=0}^{N\!-\!1} \! \Theta^{\rp (\!k_{(\!n\!+\!1\!)}\!-1\!)} \hspace{-6pt} \int_{\Delta_{(\!k_{(\!n\!+\!1\!)}\!-1\!)}} \hspace{-25pt} U^{\frac{rq}p}\!(t_{(\!k_{(\!n\!+\!1\!)}\!-1\!)},t) \, w(t) \dt\hspace{-8pt} \sup_{z\in[t_{k_{(\!n\!+\!1\!)}},\infty)} \hspace{-12pt} \Uq(t_{k_{(\!n\!+\!1\!)}},z) \lt \int_z^\infty \!\!\!\! \vp\!\!(s)\ds \rt^\rpp \nonumber \\
             & \lesssim  \sum_{n=0}^{N} \! \Theta^{\rp (\!k_{(\!n\!+\!1\!)}\!-1\!)}  \hspace{-6pt} \int_{\Delta_{(\!k_{(\!n\!+\!1\!)}\!-1\!)}} \hspace{-25pt} U^{\frac{rq}p} (t_{(\!k_{(\!n\!+\!1\!)}\!-1\!)},t) \, w(t) \hspace{-6pt} \sup_{z\in[t,t_{k_{(\!n\!+\!1\!)}}]} \hspace{-8pt} \Uq(t,z) \lt \int_z^{t_{k_{(\!n\!+\!1\!)}}} \! \vp\!(s)\ds \rt^\rpp \!\!\dt \nonumber \\
             & \quad +   \sum_{n=0}^{N} \! \Theta^{\rp (\!k_{(\!n\!+\!1\!)}\!-1\!)}  \hspace{-6pt} \int_{\Delta_{(\!k_{(\!n\!+\!1\!)}\!-1\!)}} \hspace{-25pt} U^{\frac{rq}p} (t_{(\!k_{(\!n\!+\!1\!)}\!-1\!)},t) \, w(t) \Uq(t,t_{k_{(\!n\!+\!1\!)}}) \lt \int_{t_{k_{(\!n\!+\!1\!)}}}^\infty \!\!\! \vp\!(s)\ds \rt^\rpp \!\! \dt \nonumber \\
             & \quad +   \sum_{n=0}^{N\!-\!1} \! \Theta^{\rp (\!k_{(\!n\!+\!1\!)}\!-1\!)} \hspace{-6pt} \int_{\Delta_{(\!k_{(\!n\!+\!1\!)}\!-1\!)}} \hspace{-25pt} U^{\frac{rq}p}\! (t_{(\!k_{(\!n\!+\!1\!)}\!-1\!)},t) \, w(t) \dt\hspace{-8pt} \sup_{z\in[t_{k_{(\!n\!+\!1\!)}},\infty)} \hspace{-12pt} \Uq(t_{k_{(\!n\!+\!1\!)}},z) \lt \int_z^\infty \!\!\!\! \vp\!\!(s)\ds \rt^\rpp \nonumber \\
             & =:        \B{20} + \B{21} + \B{22}. \nonumber
    \end{align}
  Inequality \eqref{001} follows from Proposition \ref{59}. Define $t_{(\!k_{(\!N\!+\!2\!)}-\!1\!)} := \infty$. Then we have
    \begin{align}
      \B{20} & =         \sum_{n=0}^{N} \Theta^{\rp (\!k_{(\!n\!+\!1\!)}\!-1\!)}  \!\!\! \int_{\Delta_{(\!k_{(\!n\!+\!1\!)}\!-1\!)}} \hspace{-25pt} U^{\frac{rq}p} (t_{(\!k_{(\!n\!+\!1\!)}\!-1\!)},t) \, w(t) \hspace{-6pt} \sup_{z\in[t,t_{k_{(\!n\!+\!1\!)}}]} \hspace{-8pt} \Uq(t,z) \lt \int_z^{t_{k_{(\!n\!+\!1\!)}}} \! \vp\!(s)\ds \rt^\rpp \!\!\!\dt \nonumber \\
             & \le       \sum_{n=0}^{N} \Theta^{\rp (\!k_{(\!n\!+\!1\!)}\!-1\!)}  \!\! \int_{\Delta_{(\!k_{(\!n\!+\!1\!)}\!-1\!)}} \hspace{-15pt} w(t) \dt \sup_{z\in\overline{\Delta}_{(\!k_{(\!n\!+\!1\!)}\!-1\!)}} \hspace{-4pt} \Ur(t_{(\!k_{(\!n\!+\!1\!)}\!-1\!)},z) \lt \int_z^{t_{k_{(\!n\!+\!1\!)}}} \! \vp\!(s)\ds \rt^\rpp \nonumber \\
             & \lesssim  \sum_{n=0}^{N} \Theta^{\rq (\!k_{(\!n\!+\!1\!)}\!-1\!)} \hspace{-5pt} \sup_{z\in{\overline{\Delta}_{(\!k_{(\!n\!+\!1\!)}\!-1\!)}}} \hspace{-5pt} \Ur(t_{(\!k_{(\!n\!+\!1\!)}\!-1\!)},z) \lt \int_z^{t_{k_{(\!n\!+\!1\!)}}} \vp\!(s)\ds \rt^\rpp \label{002} \\
             & \le       \sum_{n=0}^{N} \Theta^{\rq (\!k_{(\!n\!+\!1\!)}\!-1\!)} \lt \int_{\Delta_{(\!k_{(\!n\!+\!1\!)}\!-1\!)}} \hspace{-10pt} \Upp(t_{(\!k_{(\!n\!+\!1\!)}\!-1\!)},s)  \vp\!(s)\ds \rt^\rpp \nonumber \\
             & \lesssim  \sum_{n=0}^{N} \lt \int_{t_{(\!k_n\!-\!1\!)}}^{t_{(\!k_{(\!n\!+\!1\!)}\!-1\!)}} w(t) \dt \rt^\rq \lt \int_{t_{(\!k_{(\!n\!+\!1\!)}\!-1\!)}}^{t_{(\!k_{(\!n\!+\!2\!)}-\!1\!)}}  \Upp(t_{(\!k_{(\!n\!+\!1\!)}\!-1\!)},s)  \vp\!(s)\ds \rt^\rpp \label{226} \\
             & \le       D_1. \nonumber
    \end{align}
  Step \eqref{002} follows from \eqref{18}. In \eqref{226} we used \eqref{18} and the inequalities $t_{(\!k_n\!-\!1\!)} \le t_{(\!k_{(\!n\!+\!1\!)}\!-2\!)}$ and $t_{k_{(\!n\!+\!1\!)}}\le t_{(\!k_{(\!n\!+\!2\!)}\!-1\!)}$ which hold for all $n\in\{0,\ldots,N\}$ thanks to \eqref{10} and the definition of $t_{(\!k_{(\!N\!+\!2\!)}-\!1\!)}$. 
  
  We continue with the term $\B{21}$, for which we get
    \begin{align}
      \B{21} & =         \sum_{n=0}^{N} \Theta^{\rp (\!k_{(\!n\!+\!1\!)}\!-1\!)}  \!\! \int_{\Delta_{(\!k_{(\!n\!+\!1\!)}\!-1\!)}} \hspace{-25pt} U^{\frac{rq}p} (t_{(\!k_{(\!n\!+\!1\!)}\!-1\!)},t) \, w(t) \Uq(t,t_{k_{(\!n\!+\!1\!)}}) \lt \int_{t_{k_{(\!n\!+\!1\!)}}}^\infty \!\!\! \vp\!(s)\ds \rt^\rpp \!\! \dt \nonumber \\
             & \le       \sum_{n=0}^{N\!-\!1} \Theta^{\rp (\!k_{(\!n\!+\!1\!)}\!-1\!)} \int_{\Delta_{(\!k_{(\!n\!+\!1\!)}\!-1\!)}} \hspace{-15pt} w(t) \dt\ \Ur({\Delta_{(\!k_{(\!n\!+\!1\!)}\!-1\!)}}) \lt \int_{t_{k_{(\!n\!+\!1\!)}}}^\infty \vp\!(s)\ds \rt^\rpp \nonumber \\
             & \lesssim  \sum_{n=0}^{N\!-\!1} \Theta^{\rq (\!k_{(\!n\!+\!1\!)}\!-1\!)} \Ur({\Delta_{(\!k_{(\!n\!+\!1\!)}\!-1\!)}}) \lt \int_{t_{k_{(\!n\!+\!1\!)}}}^\infty \vp\!(s)\ds \rt^\rpp \label{003} \\
             & \le       \sum_{n=0}^{N\!-\!1} \lt \sum_{j=k_n}^{k_{(\!n\!+\!1\!)}\!-1} \Tj \Uj \rt^\rq \lt \int_{t_{k_{(\!n\!+\!1\!)}}}^\infty \vp\!(s)\ds \rt^\rpp \nonumber \\
             & =         \lt \Theta^\mu \Uq(\Delta_\mu) \rt^\rq \lt \int_{t_{k_{1}}}^\infty \vp\!(s)\ds \rt^\rpp  + \B{15} \label{227} \\
             & \lesssim  \lt \int_0^{\tm} w(t)\dt\, \Uq(\Delta_\mu) \rt^\rq \lt \int_{t_{k_{1}}}^\infty \vp\!(s)\ds \rt^\rpp  + \B{15} \label{006} \\
             & \lesssim  \lt \int_0^{t_{(\!\mu\!+\!1\!)}} w(t) \Uq(t,t_{(\!\mu\!+\!1\!)}) \dt \rt^\rq \lt \int_{t_{k_{1}}}^\infty \vp\!(s)\ds \rt^\rpp  + \B{15} \nonumber \\
             & \lesssim  D_2. \label{228}
    \end{align}
  To get \eqref{003}, we made use of \eqref{18}. In \eqref{227} we used the fact 
    \begin{equation}\label{007}
      \sum_{j=k_0}^{k_{1}\!-1} \Tj \Uj = \Theta^\mu \Uq(\Delta_\mu)
    \end{equation}  
  (recall that $k_0=\mu$ and $k_1=\mu+1$). Inequality \eqref{006} is a~consequence of\eqref{18}. The final estimate \eqref{228} follows from the relation $\B{15}\lesssim D_2$ which was proved earlier. 
  
  Concerning $\B{22}$, we may write
    \begin{align}
      \B{22} & =         \sum_{n=0}^{N\!-\!1} \Theta^{\rp (\!k_{(\!n\!+\!1\!)}\!-1\!)} \!\! \int_{\Delta_{(\!k_{(\!n\!+\!1\!)}\!-1\!)}} \hspace{-25pt} U^{\frac{rq}p} (t_{(\!k_{(\!n\!+\!1\!)}\!-1\!)},t) \, w(t) \dt\hspace{-8pt} \sup_{z\in[t_{k_{(\!n\!+\!1\!)}},\infty)} \hspace{-12pt} \Uq(t_{k_{(\!n\!+\!1\!)}},z) \lt \int_z^\infty \!\!\! \vp\!(s)\ds \rt^\rpp \nonumber \\
             & \lesssim  \sum_{n=0}^{N\!-\!1} \Theta^{\rp (\!k_{(\!n\!+\!1\!)}\!-1\!)}  U^{\frac{rq}p}(\Delta_{(\!k_{(\!n\!+\!1\!)}\!-1\!)})\ \Theta^{k_{(\!n\!+\!1\!)}\!-1}\hspace{-8pt} \sup_{z\in[t_{k_{(\!n\!+\!1\!)}},\infty)} \hspace{-12pt} \Uq(t_{k_{(\!n\!+\!1\!)}},z) \lt \int_z^\infty \!\!\! \vp\!(s)\ds \rt^\rpp \label{005} \\
             & \le       \sum_{n=0}^{N\!-\!1} \lt \sum_{j=k_n}^{k_{(\!n\!+\!1\!)}\!-\!1} \hspace{-6pt} \Tj \Uj \rt^\rp \!\! \Theta^{k_{(\!n\!+\!1\!)}\!-1} \hspace{-6pt} \sup_{z\in[t_{k_{(\!n\!+\!1\!)}},\infty)} \hspace{-5pt} \Uq(t_{k_{(\!n\!+\!1\!)}},z) \lt \int_z^\infty \vp(s)\ds \rt^\rpp \nonumber \\
             & =         \Theta^{\frac{r\mu}{q}} U^\frac{rq}{p}(\Delta_\mu) \sup_{z\in[t_{k_1},\infty)} \hspace{-5pt} \Uq(t_{k_1},z) \lt \int_z^\infty \vp(s)\ds \rt^\rpp  + \B{13} \label{008} \\
             & \le       \Theta^{\frac{r\mu}{q}} \lt \int_{\tm}^\infty \Upp(\tm,s) \vp\!(s)\ds \rt^\rpp  + \B{13} \nonumber \\
             & \lesssim  \lt \int_0^{\tm} w(t)\dt \rt^\rq \lt \int_{\tm}^\infty \Upp(\tm,s) \vp\!(s)\ds \rt^\rpp  + \B{13} \label{009} \\
             & \lesssim  D_1 + D_2. \label{010}
    \end{align}  
  Step \eqref{005} follows from \eqref{18}, step \eqref{008} from \eqref{007}, and step \eqref{009} from \eqref{18}. To obtain \eqref{010}, we used the estimate $\B{13}\lesssim D_2$ which was proved earlier. We have proved
    \[
      \B{11} \lesssim \B{20} + \B{21} + \B{22} \lesssim D_1 + D_2.   
    \]
  Together with the estimate of $\B{10}$ we obtained earlier, this also yields  
    \[
      \B{8} \lesssim \B{10} + \B{11} \lesssim D_2.   
    \]  
  
  In the next part, we return to the expression $\B{9}$. It holds    
    \begin{align}
      \B{9} & =        \sum_{n\in\A} \int_{t_{k_n}}^{t_{(\!k_{(\!n\!+\!1\!)}\!-1\!)}} \! \lt \int_{\tm}^t w(x) \Uq(x,t) \dx \rt^\rp \! w(t) \sup_{z\in[t,\infty)} \Uq(t,z) \lt \int_z^\infty \!\! \vp\!(s)\ds \rt^\rpp \!\! \dt \nonumber\\
            & \lesssim \sum_{n\in\A} \lt \sum_{j=k_{(\!n\!-\!1\!)}}^{k_n\!-\!1} \hspace{-6pt} \Tj \Uj \rt^\rp \hspace{-4pt} \int_{t_{k_n}}^{t_{(\!k_{\!(\!n\!+\!1\!)}-\!1\!)}}  \! w(t) \sup_{z\in[t,\infty)} \Uq(t,z) \lt \int_z^\infty \!\! \vp\!(s)\ds \rt^\rpp \!\! \dt \label{250}\\
            & \lesssim \sum_{n\in\A} \lt \sum_{j=k_{(\!n\!-\!1\!)}}^{k_n\!-\!1} \hspace{-6pt} \Tj \Uj \rt^\rp \hspace{-4pt} \int_{t_{k_n}}^{t_{(\!k_{\!(\!n\!+\!1\!)}-\!1\!)}}  \! w(t) \sup_{z\in[t,t_{k_{(\!n\!+\!1\!)}}]} \Uq(t,z) \lt \int_z^\infty \!\! \vp\!(s)\ds \rt^\rpp \!\! \dt \label{251}\\
            & \quad +  \sum_{n\in\A} \lt \sum_{j=k_{(\!n\!-\!1\!)}}^{k_n\!-\!1} \hspace{-6pt} \Tj \Uj \rt^\rp \hspace{-6pt} \int_{t_{k_n}}^{t_{(\!k_{\!(\!n\!+\!1\!)}-\!1\!)}}  \!\!\! w(t)\dt \hspace{-6pt} \sup_{z\in[t_{k_{(\!n\!+\!1\!)}},\infty)} \hspace{-9pt} \Uq(t_{k_{(\!n\!+\!1\!)}},z) \lt \int_z^\infty \!\!\!\! \vp\!(s)\ds \rt^\rpp \nonumber\\
            & \lesssim \sum_{n\in\A} \lt \sum_{j=k_{(\!n\!-\!1\!)}}^{k_n\!-\!1} \hspace{-6pt} \Tj \Uj \rt^\rp \hspace{-4pt} \int_{t_{k_n}}^{t_{(\!k_{\!(\!n\!+\!1\!)}-\!1\!)}}  \! w(t) \Uq(t,t_{(\!k_{(\!n\!+\!1\!)}\!-1\!)}) \dt\, \lt \int_{t_{k_n}}^\infty \!\! \vp\!(s)\ds \rt^\rpp \label{252}\\
            & \quad +  \sum_{n\in\A} \lt \sum_{j=k_{(\!n\!-\!1\!)}}^{k_n\!-\!1} \hspace{-6pt} \Tj \Uj \rt^\rp \hspace{-2pt} \Theta^{k_{(\!n\!+\!1\!)}\!-1} \hspace{-8pt} \sup_{z\in[t_{k_{(\!n\!+\!1\!)}},\infty)} \hspace{-8pt} \Uq(t_{k_{(\!n\!+\!1\!)}},z) \lt \int_z^\infty \!\!\! \vp\!(s)\ds \rt^\rpp \nonumber\\
            & \lesssim \sum_{n=1}^{N} \lt \sum_{j=k_{(\!n\!-\!1\!)}}^{k_n\!-\!1} \hspace{-2pt} \Tj \Uj \rt^\rq \! \lt \int_{t_{k_n}}^\infty \! \vp\!(s)\ds \rt^\rpp \label{253}\\
            & \quad +  \sum_{n=1}^{N\!-\!1} \lt \sum_{j=k_{(\!n\!-\!1\!)}}^{k_n\!-\!1} \hspace{-6pt} \Tj \Uj \rt^\rp \hspace{-2pt} \Theta^{k_{(\!n\!+\!1\!)}\!-1} \hspace{-8pt} \sup_{z\in[t_{k_{(\!n\!+\!1\!)}},\infty)} \hspace{-8pt} \Uq(t_{k_{(\!n\!+\!1\!)}},z) \lt \int_z^\infty \!\!\! \vp\!(s)\ds \rt^\rpp \nonumber\\
            &  =       \lt \Theta^\mu \Uq(\Delta_\mu) \rt^\rq \lt \int_{t_{k_{1}}}^\infty \vp\!(s)\ds \rt^\rpp  + \B{15}+ \B{13} \label{254} \\
            & \lesssim \lt \int_0^{\tm} w(t)\dt\, \Uq(\Delta_\mu) \rt^\rq \lt \int_{t_{k_{1}}}^\infty \vp\!(s)\ds \rt^\rpp  + \B{15} + \B{13} \label{255} \\
            & \lesssim \lt \int_0^{t_{(\!\mu\!+\!1\!)}} w(t) \Uq(t,t_{(\!\mu\!+\!1\!)}) \dt \rt^\rq \lt \int_{t_{k_{1}}}^\infty \vp\!(s)\ds \rt^\rpp  + \B{15} + \B{13}\nonumber \\
            & \lesssim D_2. \label{256}
    \end{align} 
  Estimate \eqref{250} is granted by \eqref{57}, and estimate \eqref{251} by Proposition \ref{59}. Step \eqref{252} is based on \eqref{18}. In \eqref{253} we again applied \eqref{57}. To get the relations \eqref{254} and \eqref{255}, we used \eqref{007} and \eqref{18}, respectively. The final inequality \eqref{256} follows from the already known relations $\B{15} \lesssim D_2$ and $\B{13} \lesssim D_2$. We have shown
    \[  
      \B{9} \lesssim D_2,
    \]
  and thus also
    \[
      \int_{\tm}^\infty \lt \int_{\tm}^t w(x) \Uq(x,t) \dx \rt^\rp w(t) \sup_{z\in[t,\infty)} \Uq(t,z) \lt \int_z^\infty \! \vp\!(s)\ds \rt^\rpp \!\! \dt\ \lesssim \B{8} + \B{9} \lesssim D_1 + D_2.
    \]
  If we combine this inequality with \eqref{43}, we reach
    \begin{align}
      & \int_{\tm}^\infty \lt \int_0^t w(x)\dx \rt^\rp w(t) \lt \int_t^\infty \Upp(t,z)\vp(z)\dz \rt^\rpp \dt \nonumber\\ 
      & \quad + \int_{\tm}^\infty \lt \int_{\tm}^t w(x) \Uq(x,t) \dx \rt^\rp w(t) \sup_{z\in[t,\infty)} \Uq(t,z) \lt \int_z^\infty \vp(s)\ds \rt^\rpp \dt \nonumber\\
      & \lesssim  D_1 + D_2. \nonumber
    \end{align}
  The constant related to the symbol ``$\lesssim$'' in here does not depend on the choice of $\mu$, thus passing $\mu\to -\infty$ (notice $\tm\to 0$ as $\mu\to -\infty$) and applying the monotone convergence theorem yields 
    \[
      A_1 + A_2 \lesssim D_1 + D_2.
    \]
  We have so far assumed that $\int_0^\infty w(x)\dx = \Theta^K$ for a~$K\in\Z$. The result is extended to general weights $w$ by the same procedure as the one used at the end of the proof of the implication ``(ii)$\Rightarrow$(i)''. The proof of the whole theorem is now complete.
\end{proof}

\begin{proof}[of Theorem \ref{144}]
  Theorem \ref{144} is proved in almost exactly the same way as Theorem \ref{7}. The difference is just in the use of appropriate ``limit variants'' of certain expressions for $p=1$. Namely, 
    \[
      \lt \int_y^z \Upp(y,x) \vp(x) \dx \rt^\jpp \textnormal{\quad is\ replaced\ by\quad} \esssup_{x\in(y,z)} U(y,x)v^{-1}(x)
    \]
  and
    \[
      \lt \int_y^z \vp(x)\dx \rt^\jpp \textnormal{\quad is\ replaced\ by\quad} \esssup_{x\in(y,z)} v^{-1}(x),
    \]
  whenever these expressions appear with some $0\le y<z\le\infty$. To clarify the correspondence between $A_2$ and $A_4$, let us note that
    \[
      \sup_{z\in[t,\infty)} \Uq(t,z) \esssup_{s\in(z,\infty)} v^{q'}(s) = \esssup_{s\in(t,\infty)} v^{q'}(s) \sup_{z\in[t,s)} \Uq(t,z) = \esssup_{s\in(t,\infty)} \Uq(t,s) v^{q'}(s)
    \]
  is true for all $t>0$. Naturally, the limit variant of Proposition \ref{107} for $p=1$ is used in the proof as well. All the estimates are then analogous to their counterparts in the proof of Theorem \ref{7}. Therefore, we do not repeat them in here.
\end{proof}

\begin{rem}
  (i) Theorem \ref{7}, which relates to the inequality \eqref{8}, i.e.~to the operator $H^*$, is the one proved here, while the result for $H$ (i.e.~for \eqref{8*}) is presented as Corollary \ref{7*}. Of course, the opposite order could have been chosen, since the version with $H$ instead of $H^*$ can be proved in an~exactly analogous way. As mentioned before, the variants for $H$ and $H^*$ are equivalent by a~change of variables in the integrals. The reason why the proof of the ``dual'' version is shown here is that the discretization-related notation is then the same as in \cite{K6}.

  (ii) Discretization based on finite covering sequences is used here, although the double-infinite (indexed by $\Z$) variant is far more usual in the literature (cf.~\cite{GS,L,S2}). The advantage of the finite version is that the proof works for $L^1$-weights $w$ and then it is easily extrapolated for the non-$L^1$ weights by the final approximation argument. In order to work with infinite partitions, one needs to assume $w\notin L^1$. The pass to the $L^1$-weights then cannot be done in such an~easy way as in the opposite order. The authors usually omit the case $w\in L^1$ (see e.g.~\cite{GS}). Besides that, there is no essential difference between in the techniques based on finite and infinite partitions. 

  (iii) In Theorems \ref{7} and \ref{144}, the equivalence ``(i)$\Leftrightarrow$(ii)'' was known before \cite{L} and it is reproved here using another method than in \cite{L}. The main achievement is the equivalence ``(i)$\Leftrightarrow$(iii)'' which can also be proved directly, by the same technique and without need for the discrete $D$-conditions (cf.~\cite{K6}). Doing so would however require constructing more different special functions (such as $g$ and $h$ in the ``(i)$\Rightarrow$(ii)'' part of Theorem \ref{7}) and therefore also introducing additional notation. 

  (iv) The kernel $U$ is not assumed to be continuous. However, for every $t>0$ the function $U(t,\cdot)$ is nondecreasing, hence continuous almost everywhere on $(0,\infty)$. Thus, so is the function $\Uq(t,\cdot)\lt \int_\cdot^\infty \vp(s)\ds \rt^\rpp$. Therefore, the value of the expression $A_2$ remains unchanged if ``$\sup_{z\in[t,\infty)}\!$'' in there is replaced by ``$\esssup_{z\in[t,\infty)}\!$''. Although the latter variant may seem to be the ``proper'' one, both are correct in this case. Besides that, the range $z\in[t,\infty)$ in the supremum or essential supremum may obviously be replaced by $z\in(t,\infty)$ without changing the value of $A_2$. 
  \end{rem}

\section{Applications}

The integral conditions for the boundedness $H:L^p(v)\to L^q(w)$ with $0<q<1\le p<\infty$ may be used to complete \cite[Theorem 5.1]{GS} with two missing cases. (These cases are in fact included in \cite{GS} but covered there only by discrete conditions.) 

Denote by $\MMM$ the cone of all nonnegative nonincreasing functions on $(0,\infty)$. The result then reads as follows.

\begin{thm}\label{800} 
  Let $u$, $v$, $w$ be weights, $0<q<p<\infty$,\, $q<1$ and $r=\frac{pq}{p-q}$.
    \begin{itemize}
      \item[\rm(i)]
        Let $0<p\le 1$. Then the inequality 
          \begin{equation}\label{I1}
            \lt \int_0^\infty \lt \int_t^\infty f(s)u(s)\ds \rt^q w(t) \dt \rt^\jq \le C \lt \int_0^\infty f^p(t) v(t) \dt \rt^\jp
          \end{equation}
        holds for all $f\in\MMM$ if and only if
          \[
            A_5 := \lt \int_0^\infty \lt \int_0^t w(x)\dx \rt^\rp w(t) \sup_{z\in(t,\infty)} \lt \int_t^z u(s)\ds \rt^r \lt \int_0^z v(y)\dy \rt^{-\rp}\dt \rt^\jr <\infty
          \]
        and
          \[
            A_6 := \lt \int_0^\infty \! \lt \int_0^t \!\! w(x) \lt \int_x^t \!\!u(s)\ds \rt^{\!q} \!\! \dx \rt^{\!\rp} \!\! w(t) \sup_{z\in(t,\infty)} \lt \int_t^z \!\! u(s)\ds \rt^{\!q} \! \lt \int_0^z \!\! v(y)\dy \rt^{\!\!-\rp} \!\!\!\! \dt \rt^{\!\jr} \!\! < \infty.
          \]
        Moreover, the least constant $C$ such that \eqref{I1} holds for all $f\in\MMM$ satisfies $C\approx A_5 + A_6$.
      \item[\rm(ii)]
        Let $p>1$. Then \eqref{I1} holds for all $f\in\MMM$ if and only if $A_6<\infty$,
          \[
            A_7 := \lt \int_0^\infty \! \lt \int_0^t \! w(x)\dx \rt^\rp \! w(t) \lt \int_t^\infty \! \lt \int_t^z \! u(s)\ds \rt^{p'} \! \lt \int_0^z \!\! v(y)\dy \rt^{-p'} \!\!\! v(z) \dz \rt^{\rpp}\!\!\! \dt \rt^\jr \!\! <\infty
          \]
        and $A_8<\infty$, where
          \[
            A_8 := \begin{cases} 
                             \displaystyle\ \, \lt \int_0^\infty \!\! w(t) \lt \int_0^t u(s) \ds \rt^q \dt \rt^\jq \lt \int_0^\infty \!\! v(y)\dy \rt^{\!-\jp} \!\! <\infty & \textit{if\ } \displaystyle \int_0^\infty v(y)\dy < \infty,\\ 
                             \displaystyle\ \, 0 & \textit{if\ } \displaystyle \int_0^\infty \!\! v(y)\dy = \infty.
                   \end{cases}         
          \]
        Moreover, the least constant $C$ such that \eqref{I1} holds for all $f\in\MMM$ satisfies $C\!\approx\! A_6 \!+ \!A_7\! + \!A_8$.    
    \end{itemize}
\end{thm}

\begin{proof}
  (i) By \cite[Theorem 4.1]{GS}, \eqref{I1} holds for all $f\in\MMM$ if and only if 
    \begin{equation}\label{987}
      \lt \int_0^\infty \lt \int_t^\infty \lt \int_t^x u(s)\ds \rt^p h(x) \dx \rt^\frac qp w(t) \dt \rt^\frac pq \le C^p \int_0^\infty h(s) \int_0^s v(y)\dy \ds 
    \end{equation}
  holds for all $h\in\MM$. In fact, \cite[Theorem 4.1]{GS} is stated with the assumption $\int_0^\infty v(y)\dy = \infty$ which is, however, not used in the proof in \cite{GS}.  Validity of \eqref{987} for all $h\in\MM$ is equivalent to the condition $A_5 + A_6 <\infty$ by Theorem \ref{144}, since $U(x,y)=\lt \int_x^y u(s) \ds \rt^p$ is a~$\theta$-regular kernel (with $\theta=2^p$).

  (ii) By \cite[Theorem 2.1]{GS}, \eqref{I1} holds for all $f\in\MMM$ if and only if $A_8 \le \infty$ and
    \[
      \lt \int_0^\infty \lt \int_t^\infty \int_t^x u(s)\ds\ h(x) \dx \rt^q w(t) \dt \rt^\jq \le C \lt \int_0^\infty h^p(s) \lt \int_0^s v(y)\dy \rt^{p} v^{1-p}(s) \ds \rt^\jp
    \]
  holds for all $h\in\MM$. The latter is, by Theorem \ref{7}, equivalent to the condition $A_6^* + A_7 <\infty$, where
    \[
      A_6^* := \! \lt \! \int_0^{\!\infty} \!\! \lt \! \int_0^t \!\!\! w(x) \lt \! \int_x^t \!\!\!u(s)\ds \rt^{\!q} \!\! \dx \rt^{\!\rp} \!\! w(t) \!\! \sup_{z\in(t,\infty)} \! \lt \int_t^z \!\!\! u(s)\ds \rt^{\!q} \! \lt \! \int_z^{\!\infty} \!\! \lt \int_0^x \! \!\! v(y)\dy \rt^{\!\!-p'} \!\!\!\!\! v(x) \dx \rt^{\!\!\rpp} \!\!\!\! \dt \rt^{\!\!\jr}\!\!\!.
    \]
  Since 
    \[
      \int_z^\infty \lt \int_0^s \!\! v(y)\dy \rt^{-p'}\!\!\! v(s)\ds + \lt \int_0^\infty \!\!\! v(y)\dy \rt^{1-p'} \!\!\!\! \approx\, \lt \int_0^z \!\! v(y)\dy \rt^{1-p'}
    \]
  is satisfied for all $z>0$, it is easy to verify that $A_6^* \lesssim A_6$ and $A_6 \lesssim A_6^* + A_8$. 

  In both cases (i) and (ii), the estimates on the optimal constant $C$ also follow from \cite[Theorem 2.1, Theorem 4.1]{GS} and Theorems \ref{7} and \ref{144}.
\end{proof}

In the case $0<q<p\le 1$, in \cite[Theorem 4.1]{GS} it was shown that \eqref{I1} holds for all $f\in\MMM$ if and only if
  \[ 
    \lt \int_0^\infty \lt \sup_{x\in[t,\infty)} f(x) \int_t^x u(s)\ds \rt^q w(t) \dt \rt^\jq \le C \lt \int_0^\infty f^p(t)v(t)\dt \rt^\jp
  \]         
holds for all $f\in\MMM$. Theorem \ref{800} hence applies to this supremal operator inequality as well. 

Theorem \ref{800} may be further applied to prove certain weighted Young-type convolution inequalities (cf.~\cite{K1}) in parameter settings which could not be reached so far. For this particular application, it is important that the weight $w$ is not involved in any implicit conditions. For more details see \cite{K1}.

As shown e.g.~in \cite[Theorem 4.4]{SS}, certain weighted inequalities restricted to convex functions are equivalently represented by weighted inequalities involving a~Hardy-type operator with the $1$-regular Riemann-Liouville kernel $U(x,y)=(y-x)$. Hence, the results of this paper also provide characterizations of validity of those convex-function inequalities in the case $0<q<1\le p<\infty$.

\end{document}